\documentclass{amsart}

\usepackage{amsmath,amsfonts,amsthm,amssymb,mathrsfs,latexsym}
\usepackage{graphicx,xcolor,enumerate}
\usepackage[margin=2.5cm]{geometry}
\usepackage{hyperref}
\usepackage{subsupscripts}
\usepackage{mathrsfs}
\usepackage{tikz, soul}
\usepackage{hyperref}
\usepackage[all]{xy}

\newtheorem{theorem}{Theorem}[section]
\newtheorem{lemma}[theorem]{Lemma}
\newtheorem{proposition}[theorem]{Proposition}
\newtheorem{corollary}[theorem]{Corollary}

\newtheorem{mainthm}{Theorem}

\theoremstyle{definition}
\newtheorem{definition}[theorem]{Definition}
\newtheorem{example}[theorem]{Example}
\theoremstyle{remark}
\newtheorem{remark}[theorem]{Remark}

\numberwithin{equation}{section}

\newcommand{\C}{\mathbb{C}}
\newcommand{\Sp}{\mathbb{S}}
\newcommand{\Q}{\mathbb{Q}}
\newcommand{\R}{\mathbb{R}}
\newcommand{\T}{\mathbb{T}}

\newcommand{\Z}{\mathbb{Z}}
\newcommand{\N}{\mathbb{N}}

\newcommand{\OO}{\operatorname{O}}
\newcommand{\SO}{\operatorname{SO}}
\newcommand{\GL}{\operatorname{GL}}

\newcommand{\id}{\operatorname{id}}

\newcommand{\G}{\mathbf{G}}
\newcommand{\HH}{\mathbf{H}}
\newcommand{\K}{\mathbf{K}}
\newcommand{\ftimes}[2]{{\lrsubscripts{\times}{#1}{#2}}}
\newcommand{\ttimes}{{\lrsubscripts{\times}{\text{triv}}{}}}
\newcommand{\Ad}{\mathrm{Ad}}

\newcommand{\Sat}{\operatorname{Sat}}
\newcommand{\pr}{\operatorname{pr}}

\newcommand{\triv}{\text{triv}}
\newcommand{\un}{\operatorname{un}}
\newcommand{\eq}{\operatorname{eq}}

\newcommand{\HOM}{\operatorname{Hom}}

\begin{document}

\title{The universal Euler characteristic and Burnside group for definable groupoids}

\author{Carla Farsi}
\address{Department of Mathematics, University of Colorado at Boulder,
    UCB 395, Boulder, CO 80309-0395}
\email{farsi@euclid.colorado.edu}

\author{Emily Proctor}
\address{Department of Mathematics, Middlebury College, Middlebury, VT 05753}
\email{eproctor@middlebury.edu}

\author{Christopher Seaton}
\address{Mathematics and Statistics Department,
    Skidmore College, 815 North Broadway,
    Saratoga Springs, NY 12866}
\email{cseaton@skidmore.edu}

\keywords{topological groupoid, definable groupoid, universal Euler characteristic, Burnside group}
\subjclass[2020]{Primary 22A22; Secondary 14P10, 19A22, 58H05}

\begin{abstract}
We introduce the universal Euler characteristic of orbit space definable groupoids, a class of groupoids containing  cocompact proper Lie groupoids as well as translation groupoids associated to  proper  definable  group actions.  We show that every additive and multiplicative invariant of orbit space definable groupoids with an additional local triviality hypothesis arises as a ring homomorphism applied to the universal Euler characteristic. This in particular includes the $\Gamma$-orbifold Euler characteristic introduced by the first and third authors when $\Gamma$ is a finitely presented group.

For definable groupoids, where the object and arrow spaces as well as the structure maps are definable, we also introduce a Burnside group which
is an analog of the classical Burnside ring associated to compact Lie groups and, for a transitive groupoid with compact isotropy, admits a partial multiplication.
\end{abstract}

\maketitle

\tableofcontents


\section{Introduction}
\label{sec:Intro}

Definable spaces   are a powerful generalization of
semialgebraic spaces.
Semialgebraic spaces are defined as solution sets of a finite collection  of polynomial inequalities, and thus generalize solution sets of polynomial equations.
Definable spaces   \cite{vandenDriesBook} play an  important role
in our investigation since,  by using their cell decompositions,  one can define associated definable Euler characteristics, which are not homotopy invariant  in the non-compact case.
Since Euler characteristics are   finite  additive invariants, they can also be reinterpreted   as measures, and so integration  with respect to them becomes possible, see \cite{ViroIntegralEulerChar}.

Classical concepts like topological groups, continuous group actions,
quotient spaces associated to group actions, and so on, can be developed in definable categories, and one finds
many contributions to this field in the literature, see  \cite{vandenDriesBook} and references therein.
\emph{Groupoids}, which are powerful generalizations of topological groups,   have been studied extensively in the last decades.
Perhaps the most well-known example of a groupoid  is the \emph{translation groupoid} associated to a continuous group action on a space.
Many properties, originally defined for topological groups, can be  extended to topological groupoids. Moreover,  groupoids can be given definable structures.
Here we will focus on Euler characteristics of groupoids such as in  Definitions \ref{def:OrbitSpaceDefinable} and  \ref{def:DefinGrpoid},
as well as on their Burnside group structures.

In  the paper  \cite{FarSeaJLMS}, the first and third authors
introduced the $\Gamma$-orbifold Euler characteristic (where
 $\Gamma$ is a finitely presented group) for  \emph{orbit space definable groupoids}. These groupoids are topological groupoids that are endowed with a definable structure on their  quotient space.
Examples of  orbit space definable  groupoids include cocompact proper Lie groupoids as well as many examples of topological
groupoids with no smooth structure.  The $\Gamma$-orbifold Euler characteristic of an orbit space definable groupoid is defined
using integration with respect to the Euler characteristic following \cite{GZEMH-HigherOrbEulerCompactGroup} and shown to be equal to the usual Euler
characteristic of the \emph{$\Gamma$-inertia space} \cite{FarsiPflaumSeaton2}; see Definition~\ref{def:GammaInertia}.
In the case of a translation groupoid its associated inertia groupoid is
the groupoid of the group elements that fix some point in the space, labeled by the points they fix.

Here we continue the  investigation started in  \cite{FarSeaJLMS}
by introducing  the \emph{universal Euler characteristic} $\chi^{\un}(\G)$ of an orbit space definable groupoid $\G$ and detailing some of its properties,
generalizing results of \cite[Sec.~5]{GZLMH-Universal} from the case of orbifolds.
The universal Euler characteristic takes values in a ring $\mathscr{R}$ generated by isomorphism classes of compact Lie groups; see Definition \ref{def:UnivEC}.
In particular, we show that $\chi^{\un}$ is universal among additive and multiplicative invariants of orbit space definable groupoids with an
additional local triviality hypothesis in the following sense (Theorem~\ref{thrm:ECAddUnique}).

\begin{mainthm}
\label{main:ECAddUnique}
Let $I$ be a topological invariant of orbit space definable groupoids taking values in a ring $R$.
Then $I$ is additive and multiplicative on isotropy locally trivial orbit space definable groupoids if and only if there is a unique ring homomorphism $r:\mathscr{R}\to R$
such that $I(\G)=r\circ\chi^{\un}(\G)$ for each isotropy locally trivial orbit space definable groupoid $\G$.
\end{mainthm}

See Section~\ref{sec:UniversalEC} for the definitions, Corollary~\ref{cor:ECAddUniqueIntegral} for a related formulation of such an invariant $I$
as an integral with respect to the Euler characteristic, and Theorem~\ref{thrm:ECAddUniqueTop} and Proposition~\ref{prop:ECGpdAddUnique} for additive versions of this result.
The multiplicative invariants $I$ in particular include the $\Gamma$-orbifold Euler characteristic introduced in \cite{FarSeaJLMS} when $\Gamma$ is a finitely presented group.

Afterwards, inspired by the definition of the Burnside ring for actions of compact Lie groups and its relation to Euler characteristics, see \cite{tomDieckBurnsideI} \cite[Ch.~5]{tomDieckTGRepThry},  as well as the case of finite groupoids, see \cite{KaoSpinCatfied,KaoSpinBurnsideTheory}, we introduce the \emph{Burnside group} $B(\G)$ of a definable groupoid $\G$, see Section~\ref{sec:Burnside}.
In the classical setting of a compact Lie group $G$, the Burnside ring $B(G)$ of $G$ is, roughly speaking, the  collection of all sufficiently nice $G$-spaces,
and is generated by the transitive $G$-actions. This rough idea generalizes to the groupoid case, though with significant technical differences
and yielding a Burnside group that is not Morita invariant.
In the groupoid setting, transitive actions are not as readily defined, see e.g. \cite{FarsiScullWattsPAMS},
and care is required in the appropriate notion of orbit type and associated concepts, see
Definitions~\ref{def:topological groupoid}, \ref{def:OrbitTypeGrpoid} and \ref{def:OrbitTypeG-Space}.
The definition of $B(\G)$ yields a definition of an \emph{equivariant Euler
characteristic} for definable $\G$-spaces that is given in Section~\ref{subsec:BurnsideGroup} and generalizes that of a $G$-space given in
\cite[Def.~2]{GZ-EquivariantAnalogues} when $G$ is a finite group; see Definition~\ref{def:EqEC}.

In order to define the Burnside group of a definable groupoid $\G$, we restrict to $\G$-spaces that have finitely many orbit types, in the appropriate sense.
Though the orbit types of $\G$ are significantly constrained, see Proposition~\ref{prop:FiniteHomeoTypes},
the question of whether a definable groupoid can have infinitely many orbit types remains open.
Our first main result of Section~\ref{sec:Burnside} is Theorem~\ref{thrm:BurnsideGenTransitive}, a more explicit formulation of
the following which describes the generators of $B(\G)$; see Definition~\ref{def:OrbitTypeG-Space}.

\begin{mainthm}
\label{main:BurnsideGenTransitive}
Let $\G$ be a definable groupoid with finitely many orbit types, a definable quotient, and a definable orbit type partition.
Then $B(\G)$ is generated as an abelian group by the classes of transitive $\G$-spaces.
\end{mainthm}

When $\G$ is a transitive definable groupoid with compact isotropy groups, $B(\G)$ admits a partial multiplication. In order to define this operation,
we need to restrict to $\G$-spaces whose products have finitely many orbit types;
see Section~\ref{subsec:BurnsideRing}.
Proposition~\ref{prop:BSProductWellDefTransitive} and Theorem~\ref{thrm:Product} detail the multiplication structure of $B(\G)$,
in the cases in which it can be defined.

The organization of this paper is as follows. In Section~\ref{sec:Background}, we review the relevant definitions and background
information. Section~\ref{sec:Examples} details examples of definable and orbit space definable groupoids, illustrating key concepts such as the various
notions of orbit type. Section~\ref{sec:UniversalEC} develops the universal Euler characteristic and demonstrates its universality for
isotropy locally trivial orbit space definable groupoids.
Section~\ref{sec:Burnside} defines the Burnside group
of a definable groupoid and corresponding equivariant Euler characteristic.


\section*{Acknowledgements}

We would like to thank Jason DeVito, Karl H. Hofmann, and Henrik Winther for helpful answers to questions
as well as the anonymous referees for suggestions that improved this paper.
C.S. was supported by an AMS-Simons Research Enhancement Grant for PUI Faculty and a Rhodes College Faculty Development Grant.
This work was also partially supported by C. F.'s Simons Foundation Collaboration Grants for Mathematicians \#523991 and \#MPS-TSM-00007731.
C.F. also thanks the Department of Mathematics of the University of Florence for their kind hospitality 
and financial support during the final phases of this paper's writing.


\section{Background and Definitions}
\label{sec:Background}
In this section, we recall background and fix notation. We present relevant definitions and results concerning topological groupoids in
Section~\ref{subsec:BackTopGpoids} and affine definable
spaces in Section~\ref{subsec:BackDefinSpaces}. In Section~\ref{subsec:BackOrbSpaceDefinGpoid}, we review the definition and
basic properties of orbit space definable groupoids and consider definable groupoids and corresponding $\G$-spaces in
Sections~\ref{subsec:BackDefinGpoids} and \ref{subsec:BackDefinGSpaces}.

\subsection{Topological Groupoids}
\label{subsec:BackTopGpoids}

We begin by recalling the definition of (proper) topological groupoids.
	
\begin{definition}[Topological groupoid and properties, orbit space, weak orbit type]
\label{def:topological groupoid}
A \emph{topological groupoid} consists of topological spaces $\G_0$ of objects and $\G_1$ of arrows, which in this paper are always assumed to be Hausdorff,
endowed with the following continuous structure maps:
\begin{itemize}
	\item Source: $s\colon\G_1\to\G_0$;
	\item Target: $t\colon\G_1\to\G_0$;
	\item Unit: $u\colon\G_0\to\G_1$, with $u(x)$ often written as $1_x$ for $x\in\G_0$;
	\item Inverse: $i\colon\G_1\to\G_1$, with $i(g)$ often written as $g^{-1}$ for $g\in\G_1$; and
	\item Multiplication: $m\colon\G_1\ftimes{s}{t} \G_1 = \{ (g,h)\in \G_1\times\G_1 : s(g) = t(h)\} \to\G_1$, with $m(g,h)$ often written as $gh$.
\end{itemize}
We assume that all topological groupoids are \emph{source-open}, meaning the source map $s$ is an open map. When the groupoid needs to be specified,
we will denote the structure maps $s_\G$, $t_\G$, etc.

For $g,h\in\G_1$ and $x\in\G_0$, $s(gh) = s(h)$, $t(gh) = t(g)$, multiplication is associative when it is defined,
$u(x)$ is a two-sided identity, i.e., $g 1_{s(g)} = 1_{t(g)} g = g$, $gg^{-1} = 1_{t(g)}$, and $g^{-1}g = 1_{s(g)}$.

For $x,y\in\G_0$, we denote $\G_x = s^{-1}(x)$, $\G^y = t^{-1}(y)$, and let $\G_x^y$ denote the set
\[
    \G_x^y = s^{-1}(x)\cap t^{-1}(y).
\]
When $x=y,$ $\G_x^x$ is called the \emph{isotropy group} of $x$.
If $S\subseteq\G_0$, we let $\G_{\vert S}$ denote the \emph{restriction of $\G$ to $S$}, the groupoid with object space $S$, arrow space $s^{-1}(S)\cap t^{-1}(S)$,
and structure maps given by the restrictions of those of $\G$.
The \emph{saturation} $\Sat S$ of $S$ is the set $t\circ s^{-1}(S)\subseteq\G_0$.

If $x\in\G_0$, the \emph{orbit} $\G x$ of $x$ is the set $\{ t(g) : g\in\G_x\}$.
The \emph{orbit space} $\lvert\G\rvert$ is the set of orbits of points in $\G_0$ with the quotient topology
with respect to the \emph{orbit map} $\pi\colon\G_0\to\lvert\G\rvert$ which sends $x$ to its orbit $\G x$.
The assumption that $\G$ is source-open implies that the orbit map is as well open; this assumption is sufficient for $\pi$ to be a quotient map
\cite[Prop.~2.11]{TuNonHausdorff}.

Following \cite[Def.~5.6]{PPTOrbitSpace}, we say that the \emph{weak orbit type} of $x\in\G_0$ is the isomorphism class
$[\G_x^x]$ of the isotropy group $\G_x^x$ up to isomorphism of topological groups; see Remark~\ref{rem:LieGroup}. We say that $x,y\in\G_0$ have the
\emph{same weak orbit type} if $\G_x^x$ and $\G_y^y$ are isomorphic as topological groups.
We let $\G_{[\G_x^x]}$ denote the set of points in $\G_0$ with weak orbit type $[\G_x^x]$ and let
$\lvert\G\rvert_{[\G_x^x]} = \pi(\G_{[\G_x^x]})$.

A topological groupoid is \emph{proper} if the map
\[
    (s,t): \G_1 \to \G_0 \times \G_0
\]
is proper, i.e., the preimage of compact sets is compact. We say $\G$ is \emph{source-proper} if the source map is proper.

A topological groupoid $\G$ is \emph{transitive} if $\G_0$ consists of a single orbit, i.e., for each $x,y\in\G_0$, there is a $g\in\G_1$
such that $s(g) = x$ and $t(g) = y$. Equivalently, $\lvert\G\rvert$ consists of a single point.
\end{definition}

\begin{definition}[Groupoid homomorphism]
\label{def:Groupoid-Hom}
If $\G$ and $\HH$ are topological groupoids, a \emph{homomorphism} $\Phi\colon\G\to\HH$ consists of two
continuous maps $\Phi_0\colon\G_0\to\HH_0$ and $\Phi_1\colon\G_1\to\HH_1$ that preserve each of the structure maps.
An \emph{isomorphism of topological groupoids} is a homomorphism $\Phi$ such that $\Phi_0$ and $\Phi_1$ are homeomorphisms.
\end{definition}

Note that groupoids are generalizations of groups; a topological group $G$ is a topological groupoid with a single object, and a group homomorphism
$\Phi = \Phi_1$
is a groupoid homomorphism with constant $\Phi_0$.

We now recall the definition of groupoid actions and the associated construction of translation groupoids. Both of these concepts are defined similarly to the case of a group action. We focus on left actions and note that right actions can be defined with the obvious modifications, mutatis mutandis.

\begin{definition}[Groupoid action, translation groupoid, equivariant map, intertwining map]
\label{def:crossed-prod-by-groupoid}
Let $\G$ be a topological groupoid, and let $X$ be a topological space.
\begin{enumerate}
\item   A \emph{left action}
        of $\G$ on $X$ is given by a continuous anchor map $\alpha_X\colon X\to\G_0$ and a continuous
		action map $\G_1\ftimes{s}{\alpha_X}X \to X$, written $(g,x)\mapsto g\ast x$, such that
        $\alpha_X(g\ast x) = t(g)$, $h\ast(g\ast x) = (hg)\ast x$, and $1_{\alpha_X(x)} \ast x = x$
        for all $x\in X$ and $g,h\in\G_1$ such that these expressions are defined. We then refer to $X$ as
		a \emph{(left) $\G$-space}. A left $\G$-space $X$ is \emph{proper} if the anchor map $\alpha$ is proper.
\item   If $X$ is a left $\G$-space and $x\in X$, then the \emph{$\G$-orbit} of $x$, denoted $\G x$, is the set
        $\G x = \{ g\ast x : g\in\G_1, s(g) = \alpha(x) \}$. The \emph{isotropy group} of $x$, denoted $\G_x^x$, is the
        set of $g\in\G_1$ such that $s(g) = \alpha(x)$ and $g\ast x = x$.
\item   A $\G$-space is \emph{transitive} if, for each $x,y\in X$, there is a $g\in\G_1$ such that $g\ast x = y$.
\item   If $(X,\alpha_X)$ and $(Y,\alpha_Y)$ are left $\G$-spaces with anchors $\alpha_X$ and $\alpha_Y$, a \emph{$\G$-equivariant map}
        is a map $\psi\colon X\to Y$ that commutes with the action maps, $\psi(g\ast x) = g\ast\psi(x)$, and satisfies $\alpha_X = \alpha_Y\circ\psi$.
        An \emph{isomorphism of $\G$-spaces} is a $\G$-equivariant map such that $\psi\colon X\to Y$ is a homeomorphism.
        See \cite[Def.~2.10]{FarsiScullWattsPAMS}.

\item   If $\Phi\colon\G\to\HH$ is a homomorphism of topological groupoids, $X$ is a $\G$-space, and $Y$ is an $\HH$-space,
        a \emph{$\G$-$\HH$-equivariant map intertwined by $\Phi$} is a continuous map $\psi\colon X\to Y$ such that $\Phi_0\circ\alpha_X = \alpha_Y\circ\psi$ and
        $\psi(g\ast x) = \Phi_1(g)\ast\psi(x)$.

\item   If $X$ is a $\G$-space with anchor map $\alpha_X$, the \emph{translation groupoid} $\G\ltimes X$ is defined in the following way:
		the space of objects is $X$, space of arrows is $\G_1 \ftimes{s}{\alpha_X} X$,
		$s_{\G\ltimes X}(g,x) = x$, $t_{\G\ltimes X}(g,x) = g\ast x$,
		$u_{\G\ltimes X}(x) = (1_{\alpha_X(x)},x)$, $i_{\G\ltimes X}(g,x) = (g^{-1},g\ast x)$, and
		the product is given by $(h, g\ast x)(g,x) = (hg, x)$. We refer to the orbit space $\lvert\G\ltimes X\rvert$ as the \emph{orbit space of $X$}
        and denote it by $\lvert X\rvert$.
\end{enumerate}
In particular, any topological groupoid $\G$ acts on its object space $\G_0$ on the left with anchor map the identity and action map given by
$g\ast x = t(g)$, and one checks that $\G\ltimes\G_0$ is isomorphic to $\G$.
In addition, if $X\subseteq\G_0$, then restricting the action of $\G$ on $\G_0$ to $X$, we have that $\G\ltimes X$ is isomorphic as a topological
groupoid to $\G_{\vert X}$. We will often identify $\G_{\vert X}$ with $\G\ltimes X$ via this isomorphism.
\end{definition}

To clarify our notation: when considering a topological groupoid $\G$ along with a $\G$-space $X$, we will typically use $x, y$, etc.\@
to denote elements of $X$, $p, q$, etc.\@ to denote elements of $\G_0$, and $g, h$, etc.\@ to denote elements of $\G_1$.
However, when we are only considering a groupoid $\G$ so that $\G_0$ plays the role of the space on which $\G$ acts, we will use $x, y$, etc.\@ to denote
elements of $\G_0$.

If $X$ is a $\G$-space, then a consequence of the definition of the translation groupoid $\G\ltimes X$ is that the $\G$-orbit of $x\in X$
coincides with the orbit of $x$ in the groupoid $\G\ltimes X$, and the isotropy group $\G_x^x$ of $x$ in $\G$ is isomorphic as a topological group to the isotropy group
$(\G\ltimes X)_x^x$ in $\G\ltimes X$ via $g\mapsto (g,x)$. If $S\subseteq X$, then the \emph{$\G$-saturation} of $S$ is the subset of $X$ given by
the saturation of $S$ in the groupoid $\G\ltimes X$. If $X$ is transitive, then $\G\ltimes X$ is easily seen to be transitive.

We will also at some points consider Lie groupoids, hence we recall the following.

\begin{definition}[Lie groupoid]
\label{def:proper-Lie-Groupoid}
A \emph{Lie groupoid} is a topological groupoid in which $\G_0$ and $\G_1$ are smooth manifolds (without boundary), all of the structure maps are smooth,
and the source map is a surjective submersion. A Lie groupoid is \emph{proper} if it is proper as a topological groupoid, \emph{cocompact} if
its orbit space $\lvert\G\rvert$ is compact, and \emph{\'{e}tale} if $s$ and $t$ are local diffeomorphisms.

If $\G$ and $\HH$ are Lie groupoids, a \emph{Lie groupoid homomorphism} $\Phi\colon\G\to\HH$ is a homomorphism $(\Phi_0,\Phi_1)$
of topological groupoids such that $\Phi_0$ and $\Phi_1$ are smooth.

If $\G$ is a Lie groupoid and $X$ is a smooth manifold, then a \emph{(left) Lie groupoid action} of $\G$ on $X$ is a groupoid action as
in Definition~\ref{def:crossed-prod-by-groupoid} where the anchor and action maps are smooth. In this case, $\G\ltimes X$ is a Lie groupoid.
\end{definition}

We will need the definition of the $\Gamma$-loop space, $\Gamma$-inertia groupoid, and $\Gamma$-inertia space from \cite[Sec.~4.1]{FarSeaJLMS};
see also \cite{FarsiPflaumSeaton2}.
	
\begin{definition}[$\Gamma$-loop space, $\Gamma$-inertia groupoid, $\Gamma$-inertia space]
\label{def:GammaInertia}
Let $\G$ be a topological groupoid and $\Gamma$ a finitely presented discrete group.
\begin{enumerate}
\item   Define  $\HOM(\Gamma,\G)$, the {\it $\Gamma$-loop space},  to be the set of continuous groupoid homomorphisms from $\Gamma$ to $\G$. If
        $\phi\in\HOM(\Gamma,\G)$,   let $\phi_0$ denote the map on objects and, by abuse of notation,
		also the image of the identity of $\Gamma$ in $\G_0$ (the unique value of $\phi_0$). We let $\phi_1$ denote the map on arrows so that
		$\phi_1\colon\Gamma\to\G_{\phi_0}^{\phi_0}$ is a group homomorphism.
		We endow $\HOM(\Gamma,\G)$  with the compact-open topology.
			
\item   The groupoid $\G$ acts continuously on the topological space $\HOM(\Gamma,\G)$ by conjugation.
        The anchor map of the $\G$-action is given by $A_{\Gamma,\G}\colon\phi\mapsto\phi_0$, sometimes denoted simply by $A$ when there is no risk of confusion.
        If $g\in\G_1$ and $\phi\in\HOM(\Gamma,\G)$ with $s(g) = \phi_0$, the action of $g$ on $\phi$ is given by
        $(g\phi g^{-1})_0 = t(g)$ and $(g\phi g^{-1})_1(\gamma) = g\phi_1(\gamma)g^{-1}$ for $\gamma\in\Gamma$.
        The {\it $\Gamma$-inertia groupoid} $\Lambda_\Gamma\G$ is the translation groupoid $\G \ltimes  \HOM(\Gamma,\G)$.
        The {\it $\Gamma$-inertia space} is the orbit space $\lvert\Lambda_\Gamma\G\rvert$.
			
\item   The anchor map $A_{\Gamma,\G}$ induces a continuous map on orbit spaces
		$\lvert A_{\Gamma,\G}\rvert\colon\lvert\Lambda_\Gamma\G\rvert\to\lvert\G\rvert$ given by
		$\lvert A_{\Gamma,\G}\rvert\colon\G\phi\mapsto\G\phi_0$.
\end{enumerate}
When $\Gamma = \Z$, we call $\HOM(\Z,\G)$ the \emph{loop space} and $\Lambda_\Z\G$ the \emph{inertia groupoid}. We let
$\Lambda_\ell\G = \Lambda_{\Z^\ell}\G$.
\end{definition}

The $\Gamma$-inertia groupoid $\Lambda_\Gamma\G$ of a Lie groupoid need not be a Lie groupoid
in general, although it is in some cases.
If $\G$ is a proper \'{e}tale Lie groupoid and hence presents an orbifold, then the $\Gamma$-loop space $\HOM(\Gamma,\G)$ is a smooth
manifold and $\Lambda_\Gamma\G$ is as well proper and \'{e}tale, presenting the \emph{orbifold of $\Gamma$-sectors}
or the \emph{inertia orbifold} when $\Gamma = \Z$; see \cite[Def.~2.49]{AdemLeidaRuan} and \cite[Def.~2.4]{FarsiSeaGenTwistSec}.
If $V$ is an open subset of $\G_0$ such that $\G_{\vert V} = \G\ltimes V$ is isomorphic as a Lie groupoid to $G\ltimes V$ for a finite
group $G$, i.e. $G\ltimes V$ is an \emph{orbifold chart}, then every $\phi\in\HOM(\Gamma,\G)$ with $\phi_0\in V$ defines
a homomorphism $\underline{\phi}=\pr_1\circ\phi\colon\Gamma\to G$ where $\pr_1$ denotes the projection $(\G\ltimes V)_1 = G\times V\to G$.
Then a neighborhood of $\phi$ in $\HOM(\Gamma,\G)$ is of the form $C_G(\underline{\phi})\ltimes V^{\langle\underline{\phi}\rangle}$
where $C_G(\underline{\phi})$ is the centralizer of $\underline{\phi}$ in $G$ and $V^{\langle\underline{\phi}\rangle}$ is
the subset of $V$ fixed by the image of $\underline{\phi}$. The map $\HOM(\Gamma,\G)\to\G_0$
given by $\phi\mapsto\phi_0$ identifies each connected component $S$ of $\HOM(\Gamma,\G)$
with a subset of $\G_0$ in such a way that $\G\ltimes S$, where $\G$ acts on $S$ by conjugation, is a subgroupoid
of $\G$. See \cite[Sec.~2.2B]{FarsiSeaGenTwistSec} for more details.

If $\G = G\ltimes M$ is a translation groupoid by a finite group $G$, i.e., $\G$ presents a \emph{global quotient orbifold}, then this yields
\begin{equation}
\label{eq:GlobQuotHom}
	\HOM(\Gamma,\G) = \bigsqcup\limits_{\underline{\phi}\in\HOM(\Gamma,G)} M^{\langle\underline{\phi}\rangle}.
\end{equation}
The restriction of the $G\ltimes M$-action on $\HOM(\Gamma,\G)$ is given by $C_G(\underline{\phi})\ltimes M^{\langle\underline{\phi}\rangle}$.

If $\G$ is a Lie groupoid that is not \'{e}tale, then $\HOM(\Gamma,\G)$ need not be a manifold so that $\Lambda_\Gamma\G$ is not a Lie groupoid.
However, when  $\Gamma=\Z$ and $\G$ is a proper Lie groupoid, $\Lambda\G$ was shown in \cite{FarsiPflaumSeaton2} to have the structure of a \emph{differentiable stratified groupoid}.

We now recall the definition of Morita equivalence for topological groupoids; see \cite{Metzler}.
Recall that a continuous function $f\colon X\to Y$ between topological spaces $X$ and $Y$ \emph{admits local sections}
if for every $y\in Y$, there is a neighborhood $U$ of $y$ in $Y$ and continuous function $s\colon U\to X$ such that $f\circ s$ is the identity on $U$.

\begin{definition}[Topological and Lie weak equivalence, Morita equivalence]
\label{def:WeakMoritaEquiv}
\mbox{}
\begin{enumerate}
\item   A homomorphism $\Phi\colon\G\to\HH$ between
        two topological groupoids $\G$ and $\HH$ is a \emph{topological weak equivalence}, also known as
        a \emph{topological essential equivalence}, if the map $t_\HH\circ\pr_1\colon\HH_1\ftimes{s_{\HH}}{\Phi_0}\G_0\to\HH_0$, where $\pr_1$
        is the projection onto the first factor, is a surjection that admits local sections and the diagram
\begin{equation}
\label{eq:WeakEquivFullyFaithful}
\xymatrix{
	\G_1
	\ar[d]_{(s_\G,t_\G)}
	\ar@{>}[r]^{\Phi_1}
	&\HH_1
	\ar[d]^{(s_\HH,t_\HH)}
	\\
	\G_0\times\G_0
	\ar[r]^{\Phi_0\times\Phi_0}
	&
	\HH_0\times\HH_0
}
\end{equation}
is a fibred product of spaces.
\item Two groupoids $\G$ and $\HH$ are \emph{topologically Morita equivalent} if there is a third topological groupoid
        $\K$ and topological weak equivalences $\G\leftarrow\K\rightarrow\HH$.
\end{enumerate}
If $\Phi\colon\G\to\HH$ is a homomorphism of Lie groupoids, the definition of \emph{Lie weak equivalence} is identical to (1) except that
the map $t_\HH\circ\pr_1$ is required to be a surjective submersion and \eqref{eq:WeakEquivFullyFaithful} is a fibred product of smooth manifolds.
The definition of \emph{Lie Morita equivalent} Lie groupoids is then identical to (2)
where ``topological" is replaced by ``Lie"; see \cite[Sec.~5.4]{MoerdijkMrcun}.
A Lie weak equivalence of Lie groupoids is of course a topological weak equivalence of the underlying topological groupoids, and similarly,
a Lie Morita equivalence of Lie groupoids is as well a topological Morita equivalence of topological groupoids.
\end{definition}

\begin{remark}
\label{rem:MoritaHomeo}
If $\G$ and $\HH$ are topologically Morita equivalent topological groupoids, then it is well-known that there is a homeomorphism
$\Psi\colon\lvert\G\rvert\to\lvert\HH\rvert$ that induces for each $x\in\G_0$ and $y\in\HH_0$ such that
$\Psi(\G x) = \HH y$ an isomorphism of the topological groups $\G_x^x$ and $\HH_y^y$. That is, $\Psi$
restricts to a homeomorphism $\lvert\G\rvert_{[G]}\to\lvert\HH\rvert_{[G]}$ for each weak orbit type $[G]$.
This is easily seen using
the formulation of Morita equivalence in terms of Hilsum-Skandalis morphisms/bibundles. In particular, see
\cite[Def.~62 and Prop.~64]{Metzler} for the equivalence of the bibundle definition of topological Morita equivalence with
Definition~\ref{def:WeakMoritaEquiv}. See \cite[Lem.~1]{CrainicMestre} for the homeomorphism of orbit spaces and corresponding
isomorphisms of isotropy groups in the case of Lie groupoids. Note that the proofs of Parts~1.\@ and 2.\@ of \cite[Lem.~1]{CrainicMestre}
apply verbatim to the case of topological groupoids, given \cite[Prop.~64]{Metzler}. Moreover, the isomorphism of the isotropy groups in
Part~2.\@ of \cite[Lem.~1]{CrainicMestre} is clearly continuous with continuous inverse and hence is a homeomorphism.
This can also be seen using the fact that Diagram~\eqref{eq:WeakEquivFullyFaithful} is a fibred product; see
the proof of \cite[Thrm.~4.3.1]{delHoyoOrbispace} or \cite[Prop.~60]{Metzler}.
\end{remark}


\subsection{Definable Sets and Affine Definable Spaces}
\label{subsec:BackDefinSpaces}

As in \cite{FarSeaJLMS}, we will consider topological groupoids whose orbit spaces have well-defined Euler characteristics.
Let us briefly recall this approach and refer the reader to \cite[Sec.~2.2, 3.1]{FarSeaJLMS} for more details.

An \emph{o-minimal structure} on $\R$ is a
sequence $(\mathcal{S}_n)_{n\in\N}$ such that each $\mathcal{S}_n$ is a collection
of subsets of $\R^n$ that is closed under finite unions and set differences satisfying the following:
$A\times\R,\R\times A\in \mathcal{S}_{n+1}$ for each $A\in \mathcal{S}_n$;
the diagonal $\{(x,\ldots,x):x\in\R\}\subset\R^n$ is an element of $\mathcal{S}_n$ for each $n$;
the projection of any $A\in \mathcal{S}_{n+1}$ to $\R^n$ as the first $n$ coordinates is an element of $\mathcal{S}_n$;
$\mathcal{S}_1$ is the set of finite unions of points and open intervals in $\R$; and $\mathcal{S}_2$ contains $\{(x,y)\in\R^2 : x < y\}$.
As an example, a subset of $\R^n$ is \emph{semialgebraic} if it is a finite union of solution sets to finite systems of
polynomial equations and inequalities, and the collection of semialgebraic sets is an o-minimal structure on $\R$.
See \cite[Ch.~1]{vandenDriesBook} for more details.

For the remainder of the paper, we fix an o-minimal structure on $\R$ that contains the semialgebraic sets.
We refer to elements of this structure as \emph{affine definable sets} or simply \emph{definable sets}.
The reader is welcome to use the o-minimal structure of semialgebraic
sets itself, in which case \emph{definable} is synonymous with \emph{semialgebraic} and no background on general
o-minimal structures is necessary.

A function between definable sets is \emph{definable} if its graph is a definable set.
An \emph{affine definable space} (\cite[Ch.~10, Sec.~1]{vandenDriesBook}, \cite[Def.~2.1]{FarSeaJLMS})
is a regular topological space $X$ along with a topological embedding (i.e., a continuous
function that is a homeomorphism onto its image) $\iota_X\colon X\to \R^n$ whose image $\iota_X(X)$ is a definable subset of $\R^n$.
If $(X, \iota_X)$ is an affine definable space, a subset $A\subseteq X$ is \emph{definable} if $\iota_X(A)$ is a definable subset of $\R^n$.
A \emph{morphism of affine definable spaces} is a function that induces a definable function on the images of the respective embeddings.
A \emph{definable homeomorphism} is a morphism of definable sets that is at the same time a homeomorphism. Note that the graph of a bijective function
and its inverse coincide up to permuting coordinates so that the inverse of a definable bijective function is definable.

We will make frequent use of the following.

\begin{theorem}[{Trivialization Theorem, \cite[Ch.~10, Sec.~1]{vandenDriesBook}}]
\label{thrm:Trivialization}
Let $X$ and $Y$ be definable sets and let $f\colon X\to Y$ be a continuous definable function.
Then $Y$ admits a partition $Y = Y_1\sqcup\cdots\sqcup Y_r$ into finitely
many definable subsets $Y_1,\ldots,Y_r$ such that for each $i=1,\ldots,r$,
there is a definable set $E_i$ and a continuous definable function $\lambda_i\colon f^{-1}(Y_i)\to E_i$ such that
$(f_{\vert f^{-1}(Y_i)},\lambda_i)\colon f^{-1}(Y_i)\to Y_i\times E_i$ is a definable homeomorphism.

If $X_1,\ldots, X_k$ are finitely many definable subsets of $X$, then the $Y_i$, $E_i$, and
$\lambda_i$ can be chosen so that for each $i$, there are definable subsets
$F_1,\ldots, F_k$ of $E_i$ such that $(f_{\vert f^{-1}(Y_i)}, \lambda_i)\big(X_j\cap f^{-1}(Y_i)\big) = Y_i\times F_j$.
\end{theorem}

Note that if $f$ is not surjective, then at least one instance of $E_i$ will be empty.

\begin{definition}[Definable trivialization, definably trivial]
\label{def:Trivialization}
A pair $(E_i,\lambda_i)$ as in Theorem~\ref{thrm:Trivialization} is called a \emph{definable trivialization of
$f_{\vert f^{-1}(Y_i)}$ respecting $X_1,\ldots, X_k$}, or simply a \emph{definable trivialization of
$f_{\vert f^{-1}(Y_i)}$} if there are no subsets $X_j$ under consideration. If such a definable trivialization exists,
we say that $f$ is \emph{definably trivial} on $Y_i$.
\end{definition}

\begin{remark}
\label{rem:Trivialization}
A definable trivialization identifies
$f^{-1}(Y_i)$ with $Y_i\times E_i$ in such a way that $f_{\vert f^{-1}(Y_i)}$ corresponds to the projection
onto $Y_i$. For each $y\in Y_i$, the map $(f_{\vert f^{-1}(Y_i)},\lambda_i)$
restricts to a definable homeomorphism of the fiber $f^{-1}(y)$ onto $E_i$ so that $E_i$ can be chosen to be
any such fiber. If $(X, \iota_X)$ and $(Y, \iota_Y)$ are affine definable spaces and $f\colon X\to Y$ is a continuous morphism
of affine definable spaces, then the theorem applies via application to the continuous definable function
$\iota_Y\circ f\circ\iota_X^{-1}\colon \iota_X(X)\to\iota_Y(Y)$.
\end{remark}

We will need the fact that affine definable spaces are closed under fibered products over morphisms, which we now establish.

\begin{lemma}
\label{lem:FibProdDefinable}
Let $X, Y, Z$ be affine definable spaces and $f\colon X\to Z$ and $g\colon Y\to Z$ be continuous morphisms of affine definable spaces. Then
$X\ftimes{f}{g}Y$ is an affine definable space.
\end{lemma}
\begin{proof}
Note that $X\ftimes{f}{g}Y$ is a subset of the affine definable space $X\times Y$.
As $f$ and $g$ are continuous morphisms of affine definable spaces, using Theorem~\ref{thrm:Trivialization} and Remark~\ref{rem:Trivialization},
we can decompose $Z$ into finitely many sets on which $f$
is definably trivial and finitely many sets on which $g$ is definably trivial. By intersecting these sets, we can decompose $Z$ into definable
subsets $Z = Z_1\sqcup\cdots\sqcup Z_r$ such that for each $i$, there are definable trivializations $(E_i,\lambda_i)$ of
$f_{\vert f^{-1}(Z_i)}$ and $(F_i,\mu_i)$ of $g_{\vert g^{-1}(Z_i)}$. That is,
$(f_{\vert f^{-1}(Z_i)},\lambda_i)\colon f^{-1}(Z_i)\to Z_i\times E_i$ and
$(g_{\vert g^{-1}(Z_i)},\mu_i)\colon g^{-1}(Z_i)\to Z_i\times F_i$ are definable homeomorphisms.
Let $P$ denote the continuous morphism of affine definable spaces
\[
    P = (f_{\vert f^{-1}(Z_i)}\circ\pr_1,g_{\vert g^{-1}(Z_i)}\circ\pr_2,\lambda_i\circ\pr_1,\mu_i\circ\pr_2)
        \colon  f^{-1}(Z_i)\times g^{-1}(Z_i)
                \to Z_i \times Z_i \times E_i \times F_i,
\]
and then the fibered product $f^{-1}(Z_i)\ftimes{f_{\vert f^{-1}(Z_i)}}{g_{\vert g^{-1}(Z_i)}}g^{-1}(Z_i)$
is the preimage under $P$ of the definable set $\{ (z, z, e, f) : z\in Z_i, e\in E_i, f\in F_i\}$.
Hence $f^{-1}(Z_i)\ftimes{f_{\vert f^{-1}(Z_i)}}{g_{\vert g^{-1}(Z_i)}}g^{-1}(Z_i)$ is a definable subset of $X\times Y$.
Then
\[
    X\ftimes{f}{g}Y = \bigsqcup\limits_{i=1}^r f^{-1}(Z_i)\ftimes{f_{\vert f^{-1}(Z_i)}}{g_{\vert g^{-1}(Z_i)}}g^{-1}(Z_i)
\]
is a union of definable sets and hence a definable subset of the affine definable space $X\times Y$.
\end{proof}

The proof of Lemma~\ref{lem:FibProdDefinable} also implies the following, which we will need in Section~\ref{subsec:BurnsideRing}.

\begin{corollary}
\label{cor:FibProductTrivial}
Let $X, Y, Z, E, F$ be affine definable spaces such that $X$ is definably homeomorphic to $Z\times E$ and $Y$ is definably homeomorphic to $Z\times F$.
Let $\rho_1\colon X\to Z$ denote the composition $X \to Z\times E \to Z$ where the second map is the projection, and similarly let
$\rho_2\colon Y\to Z$ denote the composition $Y\to Z\times F \to Z$.
Then $X\ftimes{\rho_1}{\rho_2}Y$ is definably homeomorphic to $Z\times E\times F$.
\end{corollary}

Definable sets have well-defined Euler characteristics that can be defined in terms of a cell decomposition \cite[Ch.~3, 4]{vandenDriesBook}.
We first recall the definition and notation for definable cells from \cite[Ch.~3, Sec.~2]{vandenDriesBook}. If $(i_1,\ldots,i_m)$ is a sequence of
$1$'s and $0$'s, then an \emph{$(i_1,\ldots,i_m)$-cell} is defined inductively as follows. A $(0)$-cell is a point in $\R$; a $(1)$-cell is an open
interval in $\R$; an $(i_1,\ldots,i_m,0)$-cell is the graph of a definable continuous function from an $(i_1,\ldots,i_m)$-cell $Y$ to $\R$; and an
$(i_1,\ldots,i_m,1)$-cell is the set of points $(p,x)$ where $p\in Y$ an $(i_1,\ldots,i_m)$-cell and $f(p) < y < g(p)$ for definable continuous
functions such that $f < g$ on $Y$; note that we allow $f$ to be the constant function $-\infty$ and $g$ to be $+\infty$.  The \emph{dimension} of
an $(i_1,\ldots,i_m)$-cell is $i_1 + i_2 + \cdots + i_m$. In general, the dimension of a definable set is the maximum dimension of a
$(i_1,\ldots,i_m)$-cell it contains. An \emph{open cell} is a $(1,\ldots,1)$-cell, and each $(i_1,\ldots,i_m)$-cell is (definably) homeomorphic to
an open cell of the same dimension via coordinate projection \cite[Ch.~3, (2.7)]{vandenDriesBook} (allowing $\R^0$, a one-point space, to be a
$(\,)$-cell, which is an open cell).

To define the Euler characteristic $\chi$ of an affine definable space $(X, \iota_X)$, identify $X$ with $\iota_X(X)\subseteq\R^n$.
Then $X$ can be written as a finite disjoint union of cells $C_1,\ldots,C_k$, and the Euler characteristic of $X$ is defined to be
\begin{equation}
\label{eq:EulerChar}
    \chi(X) = \sum_{i=1}^k (-1)^{\dim C_i};
\end{equation}
see \cite[Ch.~3, Thrm.~(2.11) and Ch.~4, (2.1)]{vandenDriesBook}. Note that $\chi$ is additive over finite unions, multiplicative over finite products,
and $\chi(\emptyset) = 0$.
This definition of the Euler characteristic agrees with other notions of $\chi$ in the case that $X$ is compact but
is not homotopy invariant in general. By \cite[Thrm.~2.2]{BekeInvarEC}, this Euler characteristic is independent of the choice of the o-minimal structure and
a homeomorphism invariant of the underlying topological space.
As $\chi$ is finitely additive, it can be used as a measure to integrate bounded constructible
functions on an affine definable space $X$; recall that $f\colon X\to\Z$ is a \emph{constructible function} if $f^{-1}(k)$ is a definable
subset of $X$ for all $k\in\Z$ (\cite[Sec.~1-2]{ViroIntegralEulerChar}, \cite[Sec.~4]{CurryGhristEAEulerCalc}).
Specifically, the integral $\int_X f(x)\,d\chi(x)$ of a bounded constructible function $f(x)$ over $X$ with respect to $\chi$
is given by $\sum_{k\in\Z}k\chi\big(f^{-1}(k)\big)$, the finite sum of the Euler characteristics of level sets of $f$ multiplied
by the corresponding value of $f$.

We now establish the following generalization of the multiplicativity of the Euler characteristic
for fibrations to affine definable spaces, which will be used in Section~\ref{subsec:BurnsideRing}.

\begin{proposition}
\label{prop:Multiplicative}
Let $X$ and $Y$ be affine definable spaces and let $f\colon X\to Y$ be a continuous surjective morphism of affine definable spaces such that
for each $y\in Y$, the fiber $f^{-1}(y)$ is definably homeomorphic to a fixed definable space $F$. Then
\[
    \chi(X) = \chi(F)\chi(Y).
\]
\end{proposition}
\begin{proof}
By Theorem~\ref{thrm:Trivialization} and Remark~\ref{rem:Trivialization}, we can decompose $Y = Y_1\sqcup Y_2\sqcup\cdots\sqcup Y_k$
such that for each $i$, $f^{-1}(Y_i)$ is definably
homeomorphic to $Y_i\times F_i$ where $F_i$ is the preimage of an arbitrary point $y_i\in Y_i$ and via this homeomorphism, $f_{\vert f^{-1}(Y_i)}$
corresponds to the projection $Y_i\times F_i\to Y_i$. Then by the additivity of the Euler characteristic,
\begin{align*}
    \chi(X)     &=      \sum\limits_{i=1}^k \chi\big(Y_i\times F_i\big)
                \\&=    \sum\limits_{i=1}^k \chi(Y_i)\chi(F_i).
\end{align*}
Because each $F_i$ is definably homeomorphic to $F$, this is equal to
\begin{align*}
    \sum\limits_{i=1}^k \chi(F)\chi(Y_i)
                &=      \chi(F) \sum\limits_{i=1}^k \chi(Y_i)
                \\&=    \chi(F)\chi(Y).
        \qedhere
\end{align*}
\end{proof}

Finally, we recall the following from \cite[Ch.~10, (2.13)]{vandenDriesBook}, which will be an important hypothesis in Section~\ref{sec:Burnside}.

\begin{definition}[Definably proper equivalence relation]
\label{def:DefinProperEquiv}
If $X$ is an affine definable space, an  equivalence relation $E\subseteq X\times X$ on $X$ is \emph{definably proper} if $E$ is definable and
the projection $E\to X$ onto the first factor is definably proper, i.e., the preimage of a compact definable set is compact;
see \cite[Ch.~6, Def.~(4.4)]{vandenDriesBook}.
\end{definition}


\subsection{Orbit space definable groupoids}
\label{subsec:BackOrbSpaceDefinGpoid}

An orbit space definable groupoid is, roughly, a groupoid whose orbit space and its partition into points with the same weak orbit type are sufficiently tame to allow for the
definition of the $\Gamma$-Euler characteristics. Recall that we have fixed an o-minimal structure on $\R$ that contains the semialgebraic sets;
\emph{definable} sets, spaces, etc.\@ always mean definable with respect to this o-minimal structure.

To begin, we observe that from Definition~\ref{def:topological groupoid}, if $\G$ is a topological groupoid, each isotropy group $\G_x^x$ is a topological
group with the multiplication and topology that it inherits as a subgroupoid of $\G$.
We now recollect the definition from \cite[Def.~3.1]{FarSeaJLMS}, which assumes that each isotropy group
$\G_x^x$ carries a compact Lie group structure. Recall that the weak orbit type of a point $x\in\G_0$ or an orbit $\G x\in\lvert\G\rvert$ is
the isomorphism class $[\G_x^x]$ of $\G_x^x$ as a topological group.

\begin{definition}[Orbit space definable groupoid]
\label{def:OrbitSpaceDefinable}
A topological groupoid $\G$ is \emph{orbit space definable} if there is an embedding $\iota_{\lvert\G\rvert}\colon\lvert\G\rvert\to\R^n$ with respect to which
$(\lvert\G\rvert,\iota_{\lvert\G\rvert})$ is an affine definable space, each isotropy group $\G_x^x$ is a compact Lie group with respect to the topology
it inherits as a subgroupoid of $\G$, and the partition of $\lvert\G\rvert$ into points with the same weak orbit type is a finite partition of
$\lvert\G\rvert$ by definable subsets.
\end{definition}

Because any continuous homomorphism of Lie groups is smooth \cite[Cor.~1.10.9]{DuistermaatKolk}, if $x,y\in\G_0$ have the same weak orbit
type in an orbit space definable groupoid $\G$, then $\G_x^x$ and $\G_y^y$ are isomorphic as Lie groups. In particular, if the topological group
$\G_x^x$ admits the structure of a Lie group, then this structure is unique.

\begin{remark}
\label{rem:LieGroup}
Our primary cases of interest are those where each $\G_x^x$ is a compact Lie group, in which case $\G_x^x$ admits a faithful finite-dimensional representation and hence a smooth isomorphism with a subgroup of $\GL(n,\R)$, so we have included this in the definition here and in \cite[Def.~3.1]{FarSeaJLMS}.
(Note that a compact topological group is a Lie group if and only if it admits such a representation; see \cite[Cor.~2.40 and Def.~2.41]{HofmannMorriss}).
In particular, the definition of the $\Gamma$-orbifold Euler characteristic makes sense whenever $\G_x^x\backslash\HOM(\Gamma,\G_x^x)$ has a well-defined Euler characteristic, which is satisfied when $\G_x^x$ is a compact Lie group; see Example~\ref{ex:GammaEC}.
This all said, the requirement in Definition~\ref{def:OrbitSpaceDefinable} that $\G_x^x$ is a compact Lie group can be relaxed in some applications.
In particular, Definition~\ref{def:UnivEC} of the universal Euler characteristic can easily be adapted,
and Theorem~\ref{thrm:ECAddUnique} holds, if the
isotropy groups are assumed merely to be topological groups.

Again by the fact that continuous homomorphisms of Lie groups are smooth \cite[Cor.~1.10.9]{DuistermaatKolk},
if a topological group admits the structure of a Lie group, then this structure is unique.
If $G$ is a definable group, i.e., a topological group that is as well an affine definable space such that the group operation and inversion are morphisms of affine
definable spaces, then by \cite[Prop.~2.5 and Rem.~2.6]{Pillay}, $G$ is a Lie group. Note that the topology constructed in \cite{Pillay} in this case coincides with
the original topology on $G$, as our definition of affine definable space assumes a topological embedding into $\R^n$.
Conversely, if $G$ is a compact Lie group, then $G$ admits a faithful finite-dimensional representation and hence a definable group structure that does not depend
on the choice of representation, see \cite[pp.~2349-50]{FarSeaJLMS}. Note that a compact Lie group $G$ may admit multiple distinct definable group structures,
as a compact definable group need not admit a faithful definable representation; see \cite[pp.~740-1]{ParkSuhLinEmbed}.
Hence, if $G$ is a compact definable group, then the underlying topological group $G$ is a Lie group, but the definable structure may not coincide with the definable
structure this Lie group inherits from a faithful representation.
\end{remark}

Examples of orbit space definable groupoids include translation groupoids by semialgebraic actions \cite[Cor.~3.5]{FarSeaJLMS}.
Further, by \cite[Lem.~3.4]{FarSeaJLMS}, cocompact proper Lie groupoids are orbit space definable, as are groupoids of the form
$\G_{\vert \pi^{-1}(U)} = \G\ltimes\pi^{-1}(U)$ where $\G$ is a cocompact proper Lie groupoid and $U$ is a definable subset of $\lvert\G\rvert$.

If $\G$ is an orbit space definable groupoid with embedding $\iota_{\lvert\G\rvert}\colon\lvert\G\rvert\to\R^n$ and $\HH$ is a topological groupoid that is
Morita equivalent to $\G$, then by Remark~\ref{rem:MoritaHomeo}, there is a homeomorphism $\Psi\colon\lvert\HH\rvert\to\lvert\G\lvert$ that restricts to
a homeomorphism $\lvert\HH\rvert_{[G]}\to\lvert\G\rvert_{[G]}$ for each weak orbit type $[G]$.
Then $\iota_{\lvert\G\rvert}\circ\Psi\colon\lvert\HH\rvert\to\R^n$ is a topological embedding with respect to which $\lvert\HH\rvert$ is an affine definable space
and each $\lvert\HH\rvert_{[G]}$ is a definable subset,
satisfying the requirements of Definition~\ref{def:OrbitSpaceDefinable}. That is, we have the following.

\begin{lemma}
\label{lem:OrbitSpaceDefinMorita}
If $\G$ is an orbit space definable groupoid with embedding $\iota_{\lvert\G\rvert}\colon\lvert\G\rvert\to\R^n$ and $\HH$ is a topological groupoid that is
topologically Morita equivalent to $\G$, then $\HH$ is as well orbit space definable with embedding $\iota_{\lvert\G\rvert}\circ\Psi\colon\lvert\HH\rvert\to\R^n$,
and for each topological group $G$ that occurs as an isotropy group of $\G$ or $\HH$, the definable subsets
$\lvert\G\rvert_{[G]}\subseteq\lvert\G\rvert$ and $\lvert\HH\rvert_{[G]}\subseteq\lvert\HH\rvert$ are definably homeomorphic.
\end{lemma}


\subsection{Definable groupoids}
\label{subsec:BackDefinGpoids}

In many examples of topological groupoids, the object space, arrow space, and structure maps are themselves definable. In this section,
we consider the class of \emph{definable groupoids} with this property. This is in contrast to the orbit space definable groupoids
recalled in the last section, for which only the orbit space is assumed to have a definable structure and hence sufficiently controlled topology 
to admit the definition of the $\Gamma$-Euler characteristics.
Recall again that we have fixed an o-minimal structure on $\R$ that contains the semialgebraic sets and that \emph{definable} means definable with respect to this structure.

\begin{definition}[Definable groupoid, definable groupoid homomorphism, definable weak and Morita equivalence]
\label{def:DefinGrpoid}
A \emph{definable groupoid} is a topological groupoid $\G$ such that the space of objects $\G_0$ and space of arrows $\G_1$ are affine definable
spaces and the structure maps $s, t\colon\G_1\to\G_0$, $i\colon\G_1\to\G_1$, $u\colon\G_0\to\G_1$, and $m\colon\G_1\ftimes{s}{t}\G_1\to\G_1$
are morphisms of affine definable spaces; note that $\G_1\ftimes{s}{t}\G_1$ is an affine definable space by Lemma~\ref{lem:FibProdDefinable}.
This in particular means that there are topological embeddings $\iota_{\G_1}\colon\G_1\to\R^n$
and $\iota_{\G_0}\colon\G_0\to\R^k$ whose images are definable sets with compatible definable structures given by the morphism $u$.

If $\G$ and $\HH$ are definable groupoids, a \emph{definable groupoid homomorphism} $\Phi\colon\G\to\HH$ is a continuous homomorphism $(\Phi_0,\Phi_1)$
of topological groupoids such that $\Phi_0$ and $\Phi_1$ are morphisms of affine definable spaces. Because the inverse of a definable function is definable,
we can define a \emph{definable groupoid isomorphism} to be a definable groupoid homomorphism that is an isomorphism of topological groupoids.
The inverse of a definable groupoid isomorphism is then a definable groupoid isomorphism.

A definable groupoid homomorphism $\Phi\colon\G\to\HH$ is a
\emph{definable weak equivalence} if it is a topological weak equivalence, see Definition~\ref{def:WeakMoritaEquiv}, satisfying the following additional requirements:
the surjective morphism $t_\HH\circ\pr_1\colon\HH_1\ftimes{s_{\HH}}{\Phi_0}\G_0\to\HH_0$ of affine definable spaces
admits local continuous definable sections,
and Diagram~\eqref{eq:WeakEquivFullyFaithful} is a fibred product of affine definable spaces; see Lemma~\ref{lem:FibProdDefinable}.
Two definable groupoids are \emph{definably Morita equivalent} if they are topologically Morita equivalent via definable weak equivalences, i.e.,
there is a third definable groupoid $\K$ and definable weak equivalences $\G\leftarrow\K\rightarrow\HH$.
\end{definition}

Given $\iota_{\G_1}$, we may define $\iota_{\G_0} = \iota_{\G_1}\circ u$ so that $\iota_{\G_0}(\G_0)\subset\iota_{\G_1}(\G_1)\subseteq\R^n$.
In particular, some authors identify $\G_0$ with $u(\G_0)\subseteq\G_1$.

Definable groupoids that are at the same time Lie groupoids were considered in \cite{Tanabe}; see Remark~\ref{rem:DefinLieFiniteOrbit}.
Other examples include translation groupoids by semialgebraic group actions or definable actions of compact Lie groups;
see Example~\ref{ex:GrpTranslation}, and see also Section~\ref{sec:Examples} for more examples.

Note that definable groupoids need not be orbit space definable. For example, the translation groupoid
$\GL(n,\R)\ltimes\R^n$ corresponding to the standard action of $\GL(n,\R)$ on $\R^n$ is definable in the o-minimal
structure of semialgebraic sets, yet the isotropy group of the origin is not compact.
Similarly, an orbit space definable groupoid $\G$ need not admit the structure of a definable groupoid.
As an example, if $X$ is any topological space, the \emph{pair groupoid of $X$} is the topological groupoid with
object space $X$, arrow space $X\times X$, source and target maps given by the projections, and multiplication
given by $(x,y)(y,z) = (x,z)$. Because a pair groupoid has trivial isotropy groups and is transitive so that $\lvert X\rvert$
is a point, the pair groupoid of any topological space $X$ is orbit space definable, though $X$ can be taken to be a topological space that does not admit the
structure of an affine definable space.

In \cite[Prop.~3.5]{FarSeaJLMS}, it is shown that if $\G$ is a proper, source-proper definable groupoid such that
the partition of $\lvert\G\rvert$ into points with the same weak orbit type
is a partition into finitely many definable sets, then $\G$ is orbit space definable.
Note that this implicitly used the result \cite[Prop.~2.5, Rem.~2.6]{Pillay} that definable
groups admit a Lie group structure; see Remark~\ref{rem:LieGroup}.

If $\G$ is a definable groupoid and $x\in\G_0$, then the orbit $\G x$ of $x$ can be expressed as $t(s^{-1}(x))$ and therefore is a definable
subset of $\G_0$. The restriction $\G_{\vert \G x} = \G\ltimes (\G x)$ of $\G$ to the orbit $\G x$ is a transitive definable groupoid.

\begin{definition}[Orbit type in a definable groupoid]
\label{def:OrbitTypeGrpoid}
Let $\G$ be a definable groupoid. The \emph{orbit type} of a point $x\in\G_0$ is the definable isomorphism class of the groupoid
$\G_{\vert \G x} = \G\ltimes (\G x)$ given by the restriction of $\G$ to the orbit $\G x$ of $x$. We also say that two points $x$ and $y$ have the \emph{same orbit type}
if the groupoids $\G_{\vert \G x} = \G\ltimes (\G x)$ and $\G_{\vert \G y} = \G\ltimes (\G y)$ are definably isomorphic. We will identify the orbit type of
$x$ with the orbit $\G x$ when $\G$ is understood.
\end{definition}

Recall that the weak orbit type of $x\in\G_0$ is the isomorphism class of $\G_x^x$ as a topological group.
In a transitive topological groupoid, the isotropy groups $\G_x^x$ and $\G_y^y$ of any two points $x,y\in\G_0$ are isomorphic as topological groups. Hence,
if $\G$ is definable and $x,y\in\G_0$ have the same orbit type, they have the same weak orbit type.
That is, the partition into points with the same orbit type is finer than the partition into points with the same weak orbit type.

If $G$ is a topological group and $X$ is a $G$-space, then in the context of group actions, points $x, y\in X$ are said to have the same orbit type
if their orbits are $G$-equivariantly homeomorphic. Assuming $G$ is compact and $X$ is Hausdorff, this is the case if and only if $x$ and $y$ have the same
\emph{isotropy type}, i.e., the isotropy groups of $x$ and $y$ are conjugate in $G$; see \cite[Ch.~1, Sec.~4]{Bredon}. In the translation groupoid $G\ltimes X$,
this is a stronger condition than $x$ and $y$ having the same
orbit type in the sense of Definition~\ref{def:OrbitTypeGrpoid}; see Example~\ref{ex:GrpTranslation}.

It is not clear whether a definable groupoid can have infinitely many orbit types in general. However, we have the following
restriction on the topology of the orbits and orbit types. See also Proposition~\ref{prop:TransitiveFiniteOrbitType}.

\begin{proposition}
\label{prop:FiniteHomeoTypes}
Let $\G$ be a proper, source-proper definable groupoid.
There are finitely many definable homeomorphism classes of $\G$-orbits in $\G_0$.
Among the groupoids $\{ \G_{\vert \G x} = \G\ltimes(\G x) : \G x\in \lvert\G\rvert\}$, there are finitely many definable homeomorphism classes of
arrow spaces. As well, there are finitely many homeomorphism classes of isotropy groups $\G_x^x$.
\end{proposition}
\begin{proof}
By \cite[Prop.~3.5]{FarSeaJLMS}, the orbit space $\lvert\G\rvert$ is an affine definable space and the orbit map
$\pi\colon\G_0\to\lvert\G\rvert$ is a continuous morphism of affine definable spaces. Then by Theorem~\ref{thrm:Trivialization},
$\lvert\G\rvert = \lvert\G\rvert_1\sqcup\cdots\sqcup\lvert\G\rvert_r$ where, for each $i$, $\pi^{-1}(\lvert\G\rvert_i)$
is definably homeomorphic to $\lvert\G\rvert_i\times F_i$ where $F_i$ is an affine definable space definably homeomorphic to
the orbit of any point in $\pi^{-1}(\lvert\G\rvert_i)$. In particular, the homeomorphism class of the fiber $\pi^{-1}(x)$
is constant on $\lvert\G\rvert_i$, implying that there are finitely many homeomorphism classes of orbits in $\G_0$.
Because $\pi$ corresponds to the projection, for any $x\in \lvert\G\rvert_i$ any arrow with source in a fiber $\{x\}\times F_i$
must have target in the same fiber as well, i.e., $\{x\}\times F_i$ is $\G$-invariant.
So without loss of generality, let us assume that $r = 1$, i.e., that $\G_0$ is definably homeomorphic to $\lvert\G\rvert\times F$.

Now, consider the map $(\pi,\pi)\circ(s,t)\colon\G_1\to\lvert\G\rvert\times\lvert\G\rvert$. Because the source and target of any
$g\in\G_1$ are in the same orbit, the image of this map is the diagonal $D$ in $\lvert\G\rvert\times\lvert\G\rvert$,
so we may compose it with the projection $\pr_1\colon\lvert\G\rvert\times\lvert\G\rvert\to\lvert\G\rvert$ to obtain a definable continuous morphism
$\pr_1\circ(\pi,\pi)\circ(s,t)\colon\G_1\to\lvert\G\rvert$ sending $g\mapsto \G s(g)$.
Again by Theorem~\ref{thrm:Trivialization}, we may decompose $\lvert\G\rvert$
into finitely many sets whose preimage under this map is definably homeomorphic to $\lvert\G\rvert_i\times E_i$ and such that
$\pr_1\circ(\pi,\pi)\circ(s,t)$ corresponds to the projection. Once again, we restrict to one of these $\G$-invariant sets
and assume that $\G_1$ is definably homeomorphic to $\lvert\G\rvert\times E$. Note that the fiber $\{\G x\}\times E$ of an orbit
$\G x \in\lvert\G\rvert$ is given by
$s^{-1}(\G x)\cap t^{-1}(\G x)$, the set of arrows restricted to the orbit of $x$. Hence, the spaces of arrows restricted to an orbit are
definably homeomorphic to one another.

Now, consider the loop space $\HOM(\Z,\G)$; see Definition~\ref{def:GammaInertia}.
Fixing a generator of $\Z$ identifies $\HOM(\Z,\G)$ with the definable subset of $\G_1$ consisting of $g$ such that $s(g) = t(g)$.
The anchor map $A\colon\HOM(\Z,\G)\to\G_0$, which coincides with the restriction of the source map, is continuous and definable,
so we may again use Theorem~\ref{thrm:Trivialization}
decompose $\G_0$ into subsets on which this map is definably trivial. It follows that there are finitely many homeomorphism classes of
isotropy groups $\G_x^x$ of $\G$, completing the proof.
\end{proof}

Let $\G$ be a proper, definable groupoid. Then the isotropy groups of $\G$ are compact definable groups and hence compact Lie groups; see Remark~\ref{rem:LieGroup}.
Thus by Proposition~\ref{prop:FiniteHomeoTypes}, there are finitely many homeomorphism types of $\G_x^x$.
By \cite[Thrm.~2]{Boekholt}, two connected compact Lie groups that are homeomorphic are locally isomorphic, and by
\cite[Cor.~3]{BaumLocalIso}, a local isomorphism class of connected compact Lie groups contains finitely many isomorphism classes.
It follows that there are finitely many isomorphism classes of identity components of isotropy groups $\G_x^x$.
Then by \cite[Thrm.~5.9.5]{tomDieckTGRepThry}, there are finitely many isomorphism classes of compact Lie groups with a fixed identity component and number of components.
We therefore have the following.

\begin{corollary}
\label{lem:FiniteIsotropyGroups}
Let $\G$ be a proper, source-proper definable groupoid.
There are finitely many isomorphism classes of isotropy groups $\G_x^x$ as topological groups.
That is, $\G$ has finitely many weak orbit types.
\end{corollary}

\begin{remark}
\label{rem:DefinLieFiniteOrbit}
If $\G$ is a $\mathcal{C}^r$ Lie groupoid, $1\leq r<\infty$ that is at the same time a definable groupoid, then Tanabe has shown that there is a Whitney stratification
of $\G_0$ into finitely many definable $\mathcal{C}^r$ submanifolds $M_i$ such that the restriction of $\G$ to each $M_i$ is a regular groupoid; see \cite[Thrm.~3.1]{Tanabe}.
In the o-minimal structure of semialgebraic sets, this holds as well for definable groupoids that are at the same time analytic Lie groupoids by \cite[Thrm.~1.1]{Tanabe}.
For regular proper $\mathcal{C}^\infty$ Lie groupoids, Weinstein's linearization theorem \cite[Thrm.~4.1]{WeinsteinLineariz} demonstrates that there is a neighborhood of
each orbit $\G x$ of finite type such that the restriction of $\G$ to $\G x$ is $\mathcal{C}^\infty$-isomorphic to $\G\ltimes N(\G x)$ where $N(\G x)$ denotes the normal bundle
to $\G x$ in $\G_0$. Using the fact that the representation of the isotropy group $\G_x^x$ of $x$ on the normal space has finitely many isotropy types, the groupoid
$\G\ltimes N(\G x)$ contains finitely many definable isomorphism classes of transitive Lie subgroupoids.
Alternatively in the $\mathcal{C}^\infty$ case, one can use the linearization theorem of Zung \cite[Thrm.~2.4]{Zung}, see also \cite{CrainicStruchiner,delHoyoDeMeloLinearization}, to describe the local structure of the groupoid.
See also \cite[Prop.~3.6]{PPTOrbitSpace}, where it is shown that every proper Lie groupoid is Lie Morita equivalent to a Lie groupoid such that each orbit has finite type.
Hence, if $\G$ is an analytic definable Lie groupoid in the o-minimal structure of semialgebraic sets such that
$\G_0$ is compact and each orbit has finite type, then $\G$ contains finitely many $\mathcal{C}^\infty$-isomorphism classes of transitive subgroupoids.
However, it is not clear if the identification of a neighborhood of $\G x$ with $\G\ltimes N(\G x)$ can be accomplished with a definable map,
and hence whether $\G$ has finitely many orbit types in the sense of Definition~\ref{def:OrbitTypeGrpoid}.
\end{remark}


\subsection{Definable $\G$-spaces}
\label{subsec:BackDefinGSpaces}

Let $\G$ be a definable groupoid. In this section, we consider the structure of a topological $\G$-space that has a compatible definable structure.
Recalling Definition~\ref{def:crossed-prod-by-groupoid} and Lemma~\ref{lem:FibProdDefinable}, we begin with the following.

\begin{definition}[Definable $\G$-space, equivariant map]
\label{def:DefinGSpace}
Let $\G$ be a definable groupoid. A \emph{definable (left) $\G$-space} is a topological $\G$-space $X$ that is an affine definable space
such that the anchor and action maps are morphisms of affine definable spaces. A \emph{definable $\G$-equivariant map} of $\G$-spaces
is a $\G$-equivariant map, see Definition~\ref{def:crossed-prod-by-groupoid}(4), that is as well a morphism of affine definable spaces.
If $\Phi\colon\G\to\HH$ is a homomorphism of
definable groupoids, then a \emph{definable $\G$-$\HH$-equivariant map intertwined by $\Phi$} is a $\G$-$\HH$-equivariant map $\psi\colon X\to Y$
intertwined by $\Phi$, see Definition~\ref{def:crossed-prod-by-groupoid}(5), that is as well a morphism of affine definable spaces.

Let $X$ be a definable $\G$-space. Recall that the \emph{orbit space} $\lvert X\rvert$ is the orbit space of the translation groupoid
$\G\ltimes X$. We say that $X$ has a \emph{definable $\G$-quotient} if $\lvert X\rvert$ admits the structure of an affine definable space
with respect to which the orbit map $\pi_X\colon X\to \lvert X\rvert$ is a morphism of affine definable spaces.
\end{definition}

If $X$ is a definable left $\G$-space, then the definition of the translation groupoid $\G\ltimes X$ implies that $\G\ltimes X$
is a definable groupoid. As well, if $x\in X$, then the $\G$-orbit $\G x$ is the image of the set $s^{-1}\big(\alpha(x)\big) \times \{x\}$
under the action map and therefore a definable subset of $X$.

If $\G$ is a definable groupoid, then $\G_0$ is a definable left $\G$-space where the anchor map is the identity and the action map is given by
$(g,x)\mapsto t(g)$. In particular, we say $\G$ has a \emph{definable quotient} if $\G_0$ has a definable
$\G$-quotient as a $\G$-space. This is the case when the equivalence relation on $\G_0$ of being in the same orbit is definably proper,
see Definition~\ref{def:DefinProperEquiv} and \cite[Ch.~10, (2.13) and Thrm.~(2.15)]{vandenDriesBook}. This in particular holds when $\G$ is proper and source-proper
by \cite[Proposition~3.5]{FarSeaJLMS}.

\begin{proposition}
\label{prop:TransitiveG-spaces}
Let $\G$ be a definable groupoid, and let $X$ be a transitive $\G$-space. Pick any point $x\in X$ so that $X = \G x$
and let $\HH = \G_{\vert \G \alpha(x)} = \G\ltimes (\G \alpha(x))$ denote the subgroupoid of $\G$ given by the restriction to the orbit $\G \alpha(x)$.
Then $X$ is a transitive $\HH$-space, and the $\G$-action is determined by the $\HH$-action.
\end{proposition}
\begin{proof}
Let $y \in X$. To see that $\alpha(y)\in \G \alpha(x)$, note that as $X$ is transitive, there is a $g\in\G_1$ such that $g\ast x = y$.
It follows that $s(g) = \alpha(x)$ and $t(g) = \alpha(g\ast x) = \alpha(y)$ so that $\alpha(y)\in \G \alpha(x)$. Therefore,
$\alpha(X) \subseteq \G \alpha(x)$. Hence, the only elements $h\in\G_1$ such that $g\ast z$ is defined for some point $z\in X$
are those such that $s(h)\in \G \alpha(x)$ and $t(h)\in \G \alpha(x)$, i.e., arrows in $\HH$.
\end{proof}

Using Proposition~\ref{prop:TransitiveG-spaces}, we define below the orbit type of a point $x$ in a definable $\G$-space $X$ to be
the equivalence class of the $\G_{\vert \G\alpha(x)}$-space $\G x$
via definable equivariant homeomorphisms intertwined by definable groupoid isomorphisms in the sense of Definition~\ref{def:DefinGSpace}.

\begin{definition}[Orbit type in a definable $\G$-space, definable orbit type partition]
\label{def:OrbitTypeG-Space}
Let $\G$ be a definable groupoid and let $X$ be a definable left $\G$-space with anchor map $\alpha$.
If $x,y\in X$, then $x$ and $y$ have the \emph{same orbit type} if the following hold:
\begin{itemize}
\item   There is a definable groupoid isomorphism $\varphi\colon\G\ltimes(\G \alpha(x))\to\G\ltimes(\G \alpha(y))$
        of the transitive definable groupoids acting on $\G x$ and $\G y$ described in Proposition~\ref{prop:TransitiveG-spaces},
        i.e., $\alpha(x)$ and $\alpha(y)$ have the same orbit type in $\G$ in the sense of Definition~\ref{def:OrbitTypeGrpoid}.
\item   There is a map $\psi\colon\G x\to\G y$ that is a $\big(\G\ltimes(\G \alpha(x))\big)$-$\big(\G\ltimes(\G \alpha(y))\big)$-equivariant
        definable homeomorphism intertwined by $\varphi$; see Definitions~\ref{def:DefinGSpace} and \ref{def:crossed-prod-by-groupoid}(5).
\end{itemize}
As $\varphi$ and $\psi$ are invertible, having the same orbit type is an equivalence relation on $X$. We refer to the set of equivalence classes
as the \emph{orbit type partition} of $X$. We say that $X$ has a \emph{definable orbit type partition} when the equivalence classes are definable
subsets of $X$. The \emph{orbit type partition of $\G$} is the orbit type partition of the $\G$-space $\G_0$;
see Proposition~\ref{prop:G0OrbitType}.
\end{definition}

Note that if $x,y\in X$ are in the same $\G$-orbit, then they trivially have the same orbit type, where the maps $\varphi$ and $\psi$ required
by Definition~\ref{def:OrbitTypeG-Space} are given by the appropriate identity maps. In particular, a transitive definable $\G$-space has a single orbit type.

\begin{remark}
\label{rem:OrbitTypeGroupoidVsSpace}
Let $\G$ be a definable groupoid and let $X$ be a definable $\G$-space. We remark that if $x, y\in X$ have the same orbit type in $X$ then they
have the same orbit type in the groupoid $\G\ltimes X$ in the sense of Definition~\ref{def:OrbitTypeGrpoid}.
Indeed, recall that the space of arrows of the translation groupoid $\G\ltimes(\G x)$ is given by $\G_1 \ftimes{s}{\alpha_X} (\G x)$ and hence
depends only on $\G\ltimes(\G\alpha(x))$; that is, $\G\ltimes(\G x) = \G_{\vert \G\alpha(x)}\ltimes(\G x)$. Hence, there is a definable isomorphism
$\G\ltimes(\G x)\to \G\ltimes(\G y)$, where the map on objects is given by the $\big(\G\ltimes(\G x)\big)$-$\big(\G\ltimes(\G y)\big)$-equivariant
definable homeomorphism $\psi$, and the map on arrows is given by the definable groupoid isomorphism $\varphi$,
indicated by Definition~\ref{def:OrbitTypeG-Space}.
\end{remark}

The converse of Remark~\ref{rem:OrbitTypeGroupoidVsSpace}, however, is false; points with the same orbit type in a translation groupoid $\G\ltimes X$ need not have the same
orbit type as points in the $\G$-space $X$, which we illustrate with the following.

\begin{example}
\label{ex:OrbitTypesDiffer}
Let $O = \{x,y\}$ be a two-point space, let $G = S_2$ be the symmetric group considered as a groupoid with a single object, and let $G$ act on $O$
by permutations. Let $\G = G\ltimes O$ denote the translation groupoid, which is isomorphic to the pair groupoid on $O$.
The transitive groupoid $\G$ acts on its object space $\G_0 = O$, and the translation groupoid $\G\ltimes O$ is isomorphic to the groupoid
$G\ltimes O$.

Now let $X = O_1\sqcup O_2$ where $O_i = \{x_i, y_i\}$ for $i=1,2$. Let $\HH = G \sqcup (G\ltimes O_2)$ be the groupoid with
$\HH_0 = \{p,x_2,y_2\}$ where $p$ is the single object of the groupoid $G$. Define an $\HH$-action on $X$ as follows.
The anchor map $\alpha$ sends $O_1$ to $p$ and is the identity on $O_2$.
Define an action of $G$ on $O_1$ by permutations and let $G\ltimes O_2$ act on $O_2$ as above.
The groupoid $\HH\ltimes X$ has two orbits, $O_1$ and $O_2$, and as $\HH\ltimes O_1$ is isomorphic to $\HH\ltimes O_2$,
they have the same orbit type in the groupoid $\HH\ltimes X$. However, there is no isomorphism from $G = \HH_{\vert \alpha(O_1)} = \HH\ltimes\{p\}$ to
$G\ltimes O_2 = \HH_{\vert \alpha(O_2)} = \HH\ltimes O_2$ as the object spaces have a different number of elements,
so that $O_1$ and $O_2$ have distinct orbit types in the $\HH$-space $X$.
\end{example}

More generally, for any Lie group $G$ that acts definably, transitively, and nontrivially on a definable space $O$, the groupoids
$G$ and $G\ltimes O$ both act on $O$ and are non-isomorphic, though the translation groupoids $G\ltimes O$ and $(G\ltimes O)\ltimes O$ are definably isomorphic.
This illustrates the general fact that if $\G$ and $\HH$ are transitive groupoids, translation groupoids $\G\ltimes (\G x)$ and $\HH\ltimes (\HH y)$
may be isomorphic even if there is no equivariant homeomorphism $\G x\to\HH y$ intertwined by a groupoid isomorphism.

In the case of the $\G$-space $\G_0$, however, the two notions of orbit type coincide. Hence, the orbit type partition of $\G$ coincides with the partition
of $\G_0$ into points with the same orbit type in $\G$, in the sense of Definition~\ref{def:OrbitTypeGrpoid}.

\begin{proposition}
\label{prop:G0OrbitType}
Let $\G$ be a definable groupoid. If $x, y\in\G_0$ have the same orbit type in the groupoid $\G$,
then they have the same orbit type in the $\G$-space $\G_0$.
\end{proposition}
\begin{proof}
Suppose $x$ and $y$ have the same orbit type in $\G$ so that, as the anchor map is the identity, there is a definable isomorphism
$\varphi\colon\G_{\vert \G x} = \G\ltimes (\G x)\to\G_{\vert \G y} = \G\ltimes (\G y)$. Then the map $\varphi_0$ on objects is a definable homeomorphism
$\G x\to \G y$ that is a $\big(\G\ltimes (\G x)\big)$-$\big(\G\ltimes (\G y)\big)$-equivariant map intertwined by $\varphi$.
\end{proof}

Let us recall the following important examples of $\G$-spaces from \cite[Ex~2.7(2), Rem~2.11, and Rem~3.6]{FarsiScullWattsPAMS}. Note that this reference
primarily uses right $\G$-spaces, so the results have been translated to the context of left $\G$-spaces. In addition, the reference considers topological
groupoids and $\G$-spaces, though in the context of a definable groupoid $\G$, the definable structures are straightforward to verify.

\begin{example}{\cite{FarsiScullWattsPAMS}}
\label{ex:PAMSsummary}
If $\G$ is a definable groupoid $\G$ and $p\in\G_0$, then $s^{-1}(p)$ is a definable left $\G$-space with anchor given by the target map and action given by
the groupoid product $g^\prime\ast g = g^\prime g$ for $g\in s^{-1}(p)$ and $g^\prime\in\G_1$ with $s(g^\prime) = t(g)$. Similarly, the group
$\G_p^p$ acts on $s^{-1}(p)$ on the right by groupoid multiplication. The left $\G$-action and right $\G_p^p$ action commute so that
$s^{-1}(p)/\G_p^p$ inherits a left $\G$-action. Specifically, letting $[g]$ denote the $\G_p^p$-orbit of $g\in s^{-1}(p)$, the action is
given by $g^\prime\ast[g] = [g^\prime g]$. If $s^{-1}(p)/\G_p^p$ admits a definable $\G$-quotient, e.g. if $\G_p^p$ is a compact Lie group,
then the $\G$-action on $s^{-1}(p)/\G_p^p$ is definable.

If $X$ is a definable left $\G$-space with anchor $\alpha$ and $p\in\G_0$ is a point in the image of $\alpha$, then the $\G$-action on $X$ restricts
to a definable left $\G_p^p$-action on $\alpha^{-1}(p)$. Conversely, starting with $\alpha^{-1}(p)$ as a left $\G_p^p$-space,
then $s^{-1}(p)\times \alpha^{-1}(p)$ is a right $\G_p^p$-space
with action given by $(g,x)\ast h = (gh, h^{-1}x)$; we denote the $\G_p^p$-orbits with $[g,x]$. The left $\G$-action on $s^{-1}(p)\times \alpha^{-1}(p)$
with anchor $\beta(g,x) = t(g)$ given by $f\ast(g,x) = (fg,x)$ commutes with the right $\G_p^p$-action so that
$s^{-1}(p)\times_{\G_p^p} \alpha^{-1}(p) = (s^{-1}(p)\times \alpha^{-1}(p))/\G_p^p$ is a left $\G$-space with anchor also denoted $\beta$ given by $\beta([g,x]) = t(g)$
and action given by $f\ast[g,x] = [fg,x]$.

By \cite[Prop.~3.7]{FarsiScullWattsPAMS}, if $\G$ is transitive and source-proper and
$X$ is a proper left $\G$-space, then $s^{-1}(p)\times_{\G_p^p} \alpha^{-1}(p)$ is $\G$-homeomorphic to $X$.
Once again, if $s^{-1}(p)\times_{\G_p^p} \alpha^{-1}(p)$ admits a definable $\G$-quotient, e.g. if $\G_p^p$ is a compact Lie group,
then the $\G$-action on $s^{-1}(p)\times_{\G_p^p} \alpha^{-1}(p)$ is definable, and the $\G$-homeomorphism from
$s^{-1}(p)\times_{\G_p^p} \alpha^{-1}(p)$ to $X$ is a definable $\G$-equivariant homeomorphism.
\end{example}

Using Example~\ref{ex:PAMSsummary}, we have the following.

\begin{proposition}
\label{prop:TransitiveFiniteOrbitType}
Let $\G$ be a transitive source-proper definable groupoid and let $X$ be a definable proper left $\G$-space.
If there is a point $p\in\G_0$ in the image of $\alpha$ such that $\alpha^{-1}(p)$ has finitely many $\G_p^p$-conjugacy classes of isotropy groups
as a $\G_p^p$-space, then $X$ has finitely many orbit types.
\end{proposition}
\begin{proof}
Note that the definable $\G$-space $X$ is by definition an affine definable space and hence regular, in particular Hausdorff.
Using \cite[Prop.~3.7]{FarsiScullWattsPAMS}, we identify $X$ with $s^{-1}(p)\times_{\G_p^p} \alpha^{-1}(p)$
as a left $\G$-space. Choose $\G_p^p$-orbits $[g_1,x_1],[g_2,x_2]\in s^{-1}(p)\times_{\G_p^p} \alpha^{-1}(p)$ with representatives $(g_1,x_1)$ and $(g_2,x_2)$
such that the $\G_p^p$-stabilizers of $x_1$ and $x_2$ are conjugate
in $\G_p^p$; this condition does not depend on the choice of representatives
$(g_i,x_i)$ of the $\G_p^p$-orbit $[g_i,x_i]$, $i=1,2$. Let $O_i$ denote the $\G$-orbit of $[g_i, x_i]$. We construct
a $\G$-equivariant definable homeomorphism $\psi\colon O_1\to O_2$. Note that as $\G$ is transitive, there is only one
orbit in $\G_0$ so the isomorphism $\varphi$ in Definition~\ref{def:OrbitTypeG-Space} can be taken to be the identity.

Let $H$ denote the $\G_p^p$-isotropy group of $x_1$. By changing the representative $(g_2,x_2)$ of the $\G_p^p$-orbit $[g_2,x_2]$ if necessary,
we may assume that $H$ is also the $\G_p^p$-isotropy group of $x_2$.
Define $\gamma = g_1 g_2^{-1}$, and then $\gamma$ is an arrow from $t(g_2)$ to $t(g_1)$.
We define $\psi$ as follows. Any element of $O_1$ can be written as $f\ast[g_1,x_1] = [f g_1, x_1]$ for some $f\in\G_1$ with $s(f) = t(g_1)$,
i.e., as the $\G_p^p$-orbit of a point of the form $(f g_1, x_1) \in s^{-1}(p)\times \alpha^{-1}(p)$.
We define the image under $\psi$ to be the $\G_p^p$-orbit of $(f \gamma g_2, x_2)$. That is,
\begin{equation}
\label{eq:TransitiveFiniteOrbitType}
    \psi\big([f g_1, x_1]\big)
        =   [f \gamma g_2, x_2].
\end{equation}

We first claim that $\psi$ is well-defined.
Suppose $[fg_1, x_1] = [kg_1, x_1]$ for $f, k\in\G_1$ with $s(f) = s(k) = t(g_1)$.
Then there is an $h\in\G_p^p$ such that $(fg_1 h, h^{-1} x_1) = (kg_1, x_1)$, implying that $h\in H$ and
\[
    k = fg_1 h g_1^{-1}.
\]
Applying $\psi$ and recalling that $\gamma = g_1 g_2^{-1}$ and $H$ fixes $x_2$,
\begin{align*}
    \psi\big([kg_1,x_1]\big)
        &=      [k\gamma g_2, x_2]
        \\&=    [k\gamma g_2 h^{-1}, h x_2]
        \\&=    \big[(fg_1 h g_1^{-1})(g_1 g_2^{-1}) g_2 h^{-1}, x_2\big]
        \\&=    [fg_1, x_2]
        \\&=    [f\gamma g_2, x_2]
        \\&=    \psi\big([fg_1,x_1]\big).
\end{align*}
Hence $\psi$ is well-defined. In addition, $\psi$ is invertible, because we can apply the same construction with $\gamma^{-1}$.

It follows from the description of $\psi$ in Equation~\eqref{eq:TransitiveFiniteOrbitType} and the definability of the groupoid operations that $\psi$ is definable.
We now confirm that $\psi$ is as well continuous.
Note that an arbitrary element $[fg_1, x_1]$ of $O_1$ can be specifically represented as $(fg_1 h, h^{-1} x_1)$ where the choice of
$f\in s^{-1}(t(g_1))$ and $h\in\G_p^p$ can be made unique locally by choosing a local section of the quotient map $\G_p^p\to\G_p^p/H$.
We can similarly define $\psi$ to be the
map induced by the $\G_p^p$-invariant map $\widetilde{\psi}\colon \widetilde{O_1}\to\widetilde{O_2}$ where $\widetilde{O_i}$
is the $\G$-saturation of the $\G_p^p$-orbit of $(g_i, x_i)$ in $s^{-1}(p)\times \alpha^{-1}(p)$ and $\widetilde{\psi}$
maps $(fg_1 h, h^{-1} x_1)$ to $(f\gamma g_2 h, h^{-1} x_2)$. Then as $\widetilde{\psi}$ is continuous and the map
$s^{-1}(p)\times \alpha^{-1}(p)\to s^{-1}(p)\times_{\G_p^p} \alpha^{-1}(p)$ is a quotient map by a group action and hence open,
$\psi$ is a definable continuous map.

It remains to show that $\psi$ is $\G$-equivariant. Recall that the anchor map of the $\G$-action on
$s^{-1}(p)\times_{\G_p^p} \alpha^{-1}(p)$ is given by $\beta\big([g,y]\big) = t(g)$. Then we have
\begin{align*}
    \beta\circ\psi\big([fg_1, x_1]\big)
        &=  \beta\big([f\gamma g_2, x_2]\big)
        \\&=    t(f\gamma g_2)
        \\&=    t(f)
        \\&=    t(fg_1)
        \\&=    \beta\big([fg_1, x_1]\big)
\end{align*}
so that $\beta_{\vert O_2}\circ\psi = \beta_{\vert O_1}$. Finally, for any $k\in\G_1$ such that $s(k) = t(f)$, we have
\begin{align*}
    \psi\big(k\ast[fg_1, x_1]\big)
        &=      \psi\big([kfg_1, x_1]\big)
        \\&=    [kf\gamma g_2, x_2]
        \\&=    k\ast[f\gamma g_2, x_2]
        \\&=    k\ast\psi\big([fg_1, x_1]\big),
\end{align*}
and $\psi$ commutes with the action map. This completes the proof.
\end{proof}


\section{Examples of Definable Groupoids}
\label{sec:Examples}

In this section, we give examples of definable groupoids to illustrate the range of their structures. Definable groupoids can be very singular
when compared, for instance, to Lie groupoids.

\begin{example}
\label{ex:GrpTranslation}
Let $G$ be a compact Lie group;
via a faithful representation, we assume $G$ is an algebraic subset of some $\R^{m^2}$ and use the corresponding definable structure on $G$.
Let $X$ be a definable subset of $\R^n$ and a $G$-space such that the action map $(g,x)\mapsto gx$ is definable. Considering $G$ as a groupoid
with a single object, the groupoid $G\ltimes X$ is a definable groupoid. Points in $X$ have the same weak orbit type if their isotropy groups
are isomorphic as topological groups, hence Lie groups, see Remark~\ref{rem:LieGroup}.
The orbit type in the groupoid $G\ltimes X$ of $x\in X$ is the isomorphism class of the definable groupoid
$G\ltimes (G x)$, which is definably isomorphic to $G\ltimes (G/G_x^x)$. Note
in particular that the partitions into points with the same weak orbit types and orbit types in the groupoid $G\ltimes X$ are generally coarser than the partition
into points with the same orbit type in $X$ as defined in the context
of group actions; two points can have the same orbit type in the groupoid $G\ltimes X$ if their isotropy groups are not conjugate in $G$.

Similarly, two points $x,y\in X$ have the same orbit type
in the $G$-space $X$ if there is an automorphism of $G$ that intertwines a $G$-equivariant definable homeomorphism of the orbit $Gx$ onto $Gy$.
Note in particular that $x$ and $y$ have the same orbit type in the $\G$-space $X$ if their isotropy groups are conjugate, but the converse is false in general;
hence, the orbit type partition in the sense of
Definition~\ref{def:OrbitTypeG-Space} is coarser than the partition into orbit types in the usual sense for group actions.
\end{example}

\begin{example}
\label{ex:GrpTranslationSpecific}
As a specific example, let the torus $G = \T^2 = \{(g_1,g_2)\in\C^2 : \lvert g_1\rvert = \lvert g_2\rvert = 1\}$ act on
$X = \C^2$ via $(g_1, g_2)(z_1, z_2) = (g_1 z_1, g_2 z_2)$ for $(z_1,z_2)\in\C^2$; we use complex coordinates for notational convenience.
Then the point $(1,0)$ has orbit $\G(1,0) = \{(z_1,0) : \lvert z_1\rvert = 1\}$,
and $\G\ltimes\big(\G(1,0)\big)$ is definably isomorphic to $\T^2\ltimes\Sp^1$ where $\T^2$ acts by multiplication with $\Sp^1$-kernel. Similarly,
$\G(0,1) = \{(0,z_2) : \lvert z_2\rvert = 1\}$, and $\G\ltimes\big(\G(0,1)\big)$ is definably isomorphic to $\G\ltimes\big(\G(1,0)\big)$
so that $(1,0)$ and $(0,1)$ have the same weak orbit type and orbit type in $\T^2\ltimes\C^2$
despite the fact that their isotropy groups are not conjugate in $G$. There are three orbit types in the groupoid $\T^2\ltimes\C^2$,
partitioning $\C^2$ into the origin, the set of points with one nonzero coordinate, and the set of points with both coordinates nonzero.
Considering $\T^2$ as a definable groupoid and $\C^2$ as a definable $\T^2$-space, the orbit type partition coincides with
the partition into points with the same orbit type in the groupoid $\T^2\ltimes\C^2$.
\end{example}

One class of examples of definable groupoids is given by the following.

\begin{definition}[Singular bundle of groups]
\label{def:SingBundGrps}
A \emph{singular bundle of groups} $\G$ is a groupoid given by the following construction.
Let $X$ be a definable subset of $\R^n$. We equip $X$ with a definable triangulation $C_1\sqcup\cdots\sqcup C_k$ in the sense of
\cite[Ch.~9]{vandenDriesBook} so that each $C_i$ is the interior of a simplex and the $C_i$ are pairwise disjoint; for simplicity, we identify
each simplex $C_i$ with its image in $X$ and consider $X = C_1\sqcup\cdots\sqcup C_k$. Let $G$ be a compact Lie group with definable structure as
in Example~\ref{ex:GrpTranslation}.
Let $H_0 \lneq H_1 \lneq\cdots\lneq H_d$ be a sequence of closed definable subgroups in $G$ where $d$ is the maximum of the dimensions of the $C_i$
and let $\G_1$ be the subspace of points
$(g,x)\in G\times X$ such that if $x$ is in the interior of a simplex of dimension $r$, then $g\in H_r$. Letting $\G_0 = X$, we define the structure
maps $s(g,x) = t(g,x) = x$, $u(x) = (e,x)$, $i(g,x) = (g^{-1},x)$, and $m\big((h,x),(g,x)\big) = (hg,x)$.
\end{definition}

Note that a singular bundle of groups $\G$ satisfies the requirements of Definition~\ref{def:topological groupoid} and hence is a topological groupoid.
Observe that $\G$ is source-open. Indeed, suppose $x$ is contained in the open
simplex $C_i$ of dimension $r$ and $U$ is an $\epsilon$-ball about $(g,x)$ in $\G_1$. If $C_j$ is a simplex whose closure contains
$C_i$, then $H_r\lneq H_{\dim C_j}$, implying that for $\epsilon$ small enough that $U$ does not intersect fibers over any simplices of dimension smaller than $r$,
the set of points $(g,y)$ such that $\lvert y - x\rvert < \epsilon$ in $X$
is contained in $\G_1$ and hence in $U$. Therefore, $s(U)$ is the $\epsilon$-ball about $x$ in $X$, implying that $s$ is an open map.
Moreover, the structure maps are clearly definable so that $\G$ is a definable groupoid satisfying Definition~\ref{def:DefinGrpoid}.

A singular bundle of groups is a definable groupoid $\G$
consisting only of isotropy arrows, i.e., $s(g,x) = t(g,x)$ for each $(g,x)\in\G_1$, and can be thought of as a bundle
of compact Lie groups whose fiber is not constant.
The orbit space $\lvert\G\rvert$ of $\G$ is homeomorphic to $X$, and the isotropy group of each point $x\in\G_0$ contained in the interior of a simplex
of dimension $r$ is $H_r$. Hence, the triangulation serves to index the points in $X$ with each of the possible fibers;
two points $x,y\in X$ have the same weak orbit type if and only if they are in (interiors of) simplices of the same dimension.
As each orbit is given by a point so that $\G\ltimes(\G x) = H_i\ltimes \{x\}$, the partition into points with the same weak orbit type
coincides with the partition into points with the same orbit type.

\begin{example}
\label{ex:SingBundGrpsR}
As an explicit example of a singular bundle of groups, let $X = \G_0 = \R$ with definable triangulation
\[
    (-\infty,0)\sqcup\{0\}\sqcup(0,\infty),
\]
let $G = \Sp^1$ be the circle, let $H_0 = \{e\} < G$, and let $H_1 = G$. Then the space of arrows $\G_1$ of the groupoid $\G$ is given by
\[
    \G_1 = \big(\Sp^1\times(\R\smallsetminus\{0\})\big)\sqcup\big(\{e\}\times\{0\}\big),
\]
the source and target maps are the projection
onto $\R$, and the unit, inverse, and multiplication maps are the group operations in the fibers.
There are two orbit types and weak orbit types in $\G$; each nonzero point has the same orbit type, and the orbit type of $0$ is distinct.

Now, consider the following $\G$-action on $Y = \Sp^1\times\R$. Let $\alpha\colon Y\to\G_0$ be the projection onto the $\R$-factor, and then
\[
    \G_1\ftimes{s}{\alpha}Y =   \{ \big((h,x),(g,x)\big) : h,g\in\Sp^1, x\in\R, (h,x)\in\G_1\}.
\]
Define an action map $\G_1\ftimes{s}{\alpha}Y\to Y$ by $\big((h,x),(g,x)\big) \mapsto (hg,x)$, and then this map along with the anchor
map $\alpha$ makes $Y$ into a definable $\G$-space. The orbits of $Y$ consist of the circles $\Sp^1\times\{x\}$ for $x\neq 0$ and the singletons
$\{(h,0)\}$ for $h\in\Sp^1$. Hence, the orbit space of $Y$ is $\R$ with a copy of $\Sp^1$ attached at the origin.
If $(g_1, x_1), (g_2, x_2)\in Y$ with $x_1, x_2\neq 0$, then
$\varphi\colon\G\ltimes\big(\G\alpha(g_1,x_1)\big) = \Sp^1\ltimes\{x_1\}\to \G\ltimes\big(\G\alpha(g_2,x_2)\big) = \Sp^1\ltimes\{x_2\}$
given by $\varphi_0(x_1) = x_2$ and $\varphi_1(h,x_1)=(h,x_2)$ is a definable groupoid isomorphism
that intertwines the $\big(\G\ltimes\big(\G\alpha(g_1,x_1)\big)\big)$-$\big(\G\ltimes\big(\G\alpha(g_2,x_2)\big)\big)$-equivariant
definable homeomorphism $\psi\colon\G(g_1, x_1)\to\G(g_2,x_2)$ given by
$\psi(g,x_1) = (g,x_2)$ for $g\in\Sp^1$. Hence, each point with nonzero $\R$-coordinate has the same orbit type in the $\G$-space $Y$.
Similarly, all points of the form $(g,0)\in Y$ have the same $\G$-space orbit type, as their orbits are singletons, and the map
sending one such point to another is trivially an equivariant definable homeomorphism intertwined by the definable groupoid isomorphism $\varphi_0^\prime(0) = 0$
and $\varphi_1^\prime(e)= e$.
\end{example}

As a variation on Example~\ref{ex:SingBundGrpsR} yielding a definable groupoid with arrows that are not isotropy arrows, we have the following.

\begin{example}
\label{ex:O(2)onR}
Let $\G_0 = \R$, again triangulated as $(-\infty,0)\sqcup\{0\}\sqcup(0,\infty)$,
and let $G = \OO_2(\R)$ be the $2\times 2$ orthogonal group over $\R$. Define the space of arrows of the groupoid $\G$ to be
\[
    \G_1 = \{ (g, x)\in G\times\R : g = e \text{ if } x = 0\}.
\]
Recall that $\OO_2(\R)$ consists of two connected components, $\SO_2(\R)\simeq\Sp^1$ consisting of elements with determinant $1$, and
$\OO_2(\R)\smallsetminus\SO_2(\R)$ consisting of elements with determinant $-1$. Define $s(g,x) = x$, $u(x) = (e,x)$, and
\[
    t(g,x)  =   \begin{cases}
                    x,  &   g\in\SO_2(\R),
                    \\
                    -x, &   g\notin\SO_2(\R).
                \end{cases}
\]
Similarly define
\[
    i(g,x)  =   \begin{cases}
                    (g^{-1},x),     &   g\in\SO_2(\R),
                    \\
                    (g^{-1},-x),    &   g\notin\SO_2(\R).
                \end{cases}
\]
Then the product of $(h,y)$ and $(g,x)$ is defined if and only if $s(h,y) = t(g,x)$, i.e., $g\in\SO_2(\R)$ and $y = x$ or $g\notin\SO_2(\R)$ and $y = -x$, and is given by
\[
    m\big((h,x),(g,x)\big)  =  (hg,x) \text{ if } g\in\SO_2(\R), \quad
    m\big((h,-x),(g,x)\big) =  (hg,x) \text{ if } g\notin\SO_2(\R),
\]
or more concisely,
\[
    m\big((h,\det(g)x),(g,x)\big)  =  (hg,x).
\]
Note that $\G\ltimes(\R\smallsetminus\{0\}) = \G_{\vert \R\smallsetminus\{0\}}$ is definably isomorphic to the translation groupoid
of $\OO_2(\R)$ acting on $\R\smallsetminus\{0\}$ where the action of $g\in\OO_2(\R)$ is given by $gx = \det(g)x$,
though $\G$ itself is not isomorphic as a topological groupoid to the translation groupoid of this action on $\R$. In particular, the orbit space $\lvert\G\rvert$
of $\G$ is homeomorphic to $[0,\infty)$,
where $0$ has trivial isotropy and each other point has $\SO_2(\R)\simeq\Sp^1$-isotropy.

Once again, the partition into points with the same weak orbit type in $\G$ coincides with the partition into points with the same orbit type.
Each nonzero point $x\in\R$ has orbit $\{\pm x\}$, and $\G\ltimes(\G x)$
is definably isomorphic to the translation groupoid of the $\OO_2(\R)$-action on $\{\pm x\}$ given by $hx = (\det h)x$.
The orbit of the origin is the singleton $\{0\}$ with trivial isotropy, hence $0$ is the only point with the second orbit type.

As a variation of this example with singular object space, one may replace $\G_0$ with the union of the origin and positive $x$-
and $y$-axes in $\R^2$, identifying the positive $y$-axis with $(-\infty,0)$ to define the action. Similarly, one may
replace $\G_0$ with the quadratic cone $x^2 + y^2 = z^2$ in $\R^3$ and define the action of $g\notin\SO_2(\R)$ as reflection
through the $xy$-plane in $\R^3$, yielding a groupoid whose orbit space is a single node of the cone with trivial isotropy
at the origin and $\SO_2(\R)\simeq\Sp^1$-isotropy elsewhere.
\end{example}

Note that the cases considered in Definition~\ref{def:SingBundGrps} and Examples~\ref{ex:SingBundGrpsR} and \ref{ex:O(2)onR} are unlike translation groupoids
$G\ltimes M$ for the smooth action of a compact Lie group $G$ on a smooth connected manifold $M$, or more generally, proper Lie groupoids. In the case of
$G\ltimes M$, there is an open dense subset of $M$ called the \emph{principal stratum}. The isotropy groups of points in the principal stratum are all conjugate and
are minimal, in the sense that they are contained in every other isotropy group up to conjugation; see \cite[Ch.~IV, Thrm.~3.1]{Bredon}.
Hence, in the case of $G\ltimes M$, the isotropy
groups are smallest on an open dense set and become larger on lower-dimensional subsets. This is also the case locally for proper Lie groupoids using the slice theorem
of \cite{Zung}, see \cite{CrainicStruchiner,delHoyoDeMeloLinearization,PPTOrbitSpace}. However, in the above examples the isotropy groups are largest on an open dense set.

Finally, we give an example to illustrate that the partition of $\lvert\G\rvert$ into points with the same orbit type in a definable groupoid $\G$ need not be
preserved by definable Morita equivalence; see Definition~\ref{def:DefinGrpoid} and compare Example~\ref{ex:OrbitTypesDiffer}.
Recall that for orbit space definable groupoids, the partition of $\lvert\G\rvert$ into points with the same weak orbit type is preserved by topological
Morita equivalence; see Lemma~\ref{lem:OrbitSpaceDefinMorita}.

\begin{example}
\label{ex:OrbitTypesMorita}
As in Example~\ref{ex:OrbitTypesDiffer}, let $O = \{x,y\}$ be a two-point space, let $G = S_2$ be the symmetric group considered as a groupoid with a single object,
and let $G$ act on $O$ by permutations. Then $G\ltimes O$ is isomorphic to the pair groupoid on $O$. Let $\mathbf{T}$ denote the trivial
groupoid with one object $z$ and one arrow $u(z)$. The embedding of $\mathbf{T}$ into $G\ltimes O$ sending $z$ to $x$ is easily seen to be a definable weak
equivalence.

Let $\G = (G_1\ltimes O_1)\sqcup(G_2\ltimes O_2)$ and $\G^\prime = (G_3\ltimes O_3)\sqcup\mathbf{T}$
where each $G_i$ and $O_i = \{x_i,y_i\}$ for $i=1,2,3$ are copies of $G$ and $O$, respectively.
Then the groupoid homomorphism $\G^\prime\to\G$ that is the identity $\G_3\ltimes O_3\to\G_1\ltimes O_1$ on $\G_3\ltimes O_3$ and $z\mapsto x_2$ on $\mathbf{T}$
is a definable weak equivalence. However, the two orbits $\G x_1$ and $\G x_2$ in $\G$ clearly
have the same orbit type, while the two orbits $\G^\prime x_3$ and $\G^\prime z$ in $\G^\prime$ have different cardinalities and hence distinct orbit types.
\end{example}


\section{The universal Euler characteristic of orbit space definable groupoids}
\label{sec:UniversalEC}

We now examine Euler characteristics of orbit space definable groupoids.  We begin by considering additive and multiplicative topological invariants on general affine definable spaces and on orbit space definable groupoids.  In an extension of \cite[Sec.~5]{GZLMH-Universal}, we introduce the universal Euler characteristic $\chi^{\un}$ of orbit space definable groupoids and show that, among orbit space definable groupoids with sufficiently controlled local topology, see Definition~\ref{def:IsoLocTriv}, any additive and multiplicative topological invariant of orbit space definable groupoids can be realized as the composition of $\chi^{\un}$ with a homomorphism of rings.  We conclude this section with Examples~\ref{ex:GammaEC} and \ref{ex:GammaESC}, which include the $\Gamma$-Euler characteristic for orbit space definable groupoids introduced in \cite{FarSeaJLMS} as well as the Euler-Satake analogue of counting connected components \cite{SatakeGB}.

We begin with the following definitions for general affine definable spaces.

\begin{definition}[Additive and multiplicative topological invariant]
\label{def:AffineTopInvar}
Let $R$ be a set.
\begin{enumerate}
\item A \textit{topological invariant of affine definable spaces} $I$ assigns to each affine definable space $X$ an element $I(X)$ of $R$ such that if $X$ and $Y$ are homeomorphic affine definable spaces, then $I(X) = I(Y)$.
\item If $R$ has the structure of an additive group, we say that a topological invariant of affine definable spaces is \textit{additive} if for every affine definable space $X$ and definable subset $Y$, we have
\begin{equation}
\label{eq:AdditiveTop}
    I(X)    =   I(Y) + I(X\smallsetminus Y).
\end{equation}
\item If $R$ has the additional multiplicative structure of a ring, we say that a topological invariant of affine definable spaces is \textit{multiplicative} if for every pair of affine definable spaces $X$ and $Y$, we have
\begin{equation}
\label{eq:Multiplicative}
I(X\times Y) = I(X)\cdot I(Y).
\end{equation}
\end{enumerate}
\end{definition}

We extend Definition~\ref{def:AffineTopInvar} in a natural way to orbit space definable groupoids, whose orbit spaces are affine definable spaces. First note the following.

\begin{remark}
\label{rem:OrbSpDefGpdInvar}
\mbox{}
\begin{enumerate}
\item   If $\G$ is an orbit space definable groupoid and $X\subseteq\lvert\G\rvert$ is a definable subset, then $\G_{\vert\pi^{-1}(X)} = \G\ltimes\pi^{-1}(X)$
        admits the structure of an orbit space definable groupoid by restricting the embedding of $\lvert\G\rvert$ to $X$. See \cite[Lem.~3.2(ii)]{FarsiSeaGenOrbEuler}.
\item   Similarly, if $\G$ and $\HH$ are orbit space definable groupoids, then $\G\times\HH$ admits the structure of an orbit space
        definable groupoid via the product embedding. See \cite[Lem.~3.2(i)]{FarsiSeaGenOrbEuler}.
\end{enumerate}
\end{remark}

\begin{definition}[Additive and multiplicative topological invariant of orbit space definable groupoids]
\label{def:OrbSpDefGpdInvar}
Let $R$ be a set.
\begin{enumerate}
\item   A \textit{topological invariant of orbit space definable groupoids} $I$ assigns to each orbit space definable groupoid $\G$ an element $I(\G)$ of $R$ such that if $\G$ and $\HH$ are orbit space definable groupoids that are topologically Morita equivalent, then $I(\G) = I(\HH)$.
\item   If $R$ has the structure of an additive group, we say that a topological invariant of orbit space definable groupoids is \textit{additive} if for every orbit space definable groupoid $\G$ and disjoint definable subsets $X_1$ and $X_2$ of $\lvert\G\rvert$ such that $\lvert\G\rvert = X_1\sqcup X_2$, setting
    $\G_i = \G_{\vert\pi^{-1}(X_i)} = \G\ltimes\pi^{-1}(X_i)$ for $i = 1,2$, we have
\begin{equation}
\label{eq:AdditiveGpd}
    I(\G)    =   I(\G_1) + I(\G_2).
\end{equation}
\item   If $R$ has the additional multiplicative structure of a ring, we say that a topological invariant of orbit space definable groupoids is \textit{multiplicative} if for every pair of orbit space definable groupoids $\G$ and $\HH$, we have
\begin{equation}
\label{def:Multiplicative}
I(\G\times \HH) = I(\G)\cdot I(\HH).
\end{equation}
\end{enumerate}
\end{definition}

We observe that the Euler characteristic $\chi$ of Equation~\eqref{eq:EulerChar} is a well-defined additive and multiplicative topological invariant of affine definable spaces that takes values in the ring $\mathbb{Z}$.  Now, suppose that $I$ is any additive topological invariant of affine definable spaces.  In a generalization of \cite[Prop.~1]{GZLMH-Universal}, we make the following observations about the relationship between an additive topological invariant of affine definable spaces and $\chi$.
See Section~\ref{subsec:BackDefinSpaces} for a review of the notation of definable cells and \cite[Ch.~3, Sec.~2]{vandenDriesBook} for more details.

\begin{lemma}
\label{lem:ECAddUnique}
Let $I$ be an additive topological invariant of affine definable spaces. If $C$ is an
$(i_1,\ldots,i_m)$-cell, then $I(C) = (-1)^{\dim C}I(\{p\})$ where $p$ is a point.
\end{lemma}
\begin{proof}
By \cite[Ch.~3, (2.7)]{vandenDriesBook} and topological invariance, it is sufficient to show the result for
$(\overset{m}{\overbrace{1,\ldots,1}})$-cells $C$ in $\R^m$. The proof is by induction on $m$. First assume $m = 0$ so that $C$ is a $(\,)$-cell. Then
$C$ is a point so that $I(C) = I(\{p\})$ by topological invariance. If $m = 1$ so that $C$ is a $(1)$-cell,
then $C$ is an interval $(a,b)$. Choosing $p$ such that $a < p < b$, we can express $C= (a,p) \sqcup \{p\}\sqcup (p,b)$ and then by additivity and
topological invariance,
\begin{align*}
    I\big((a,b)\big)
        &=      I\big((a,p)\big) + I(\{p\}) + I\big((p,b)\big)
        \\&=    2I\big((a,b)\big) + I(\{p\}).
\end{align*}
It follows that $I(C) = - I(\{p\})$.

Now assume the result holds for $(\overset{m}{\overbrace{1,\ldots,1}})$-cells. If $C$ is an
$(\overset{m}{\overbrace{1,\ldots,1}},1)$-cell in $\R^{m+1}$, then there is a $(\overset{m}{\overbrace{1,\ldots,1}})$-cell $C^{\prime}$
and definable continuous (possibly infinite) functions $f,g\colon C^{\prime}\to\R$ such that $f < g$ and $C = \{ (p,x) : p\in C^{\prime}, f(p)<x<g(p)\}$.
Choose a definable continuous function $h\colon C^{\prime}\to\R$ such that $f < h < g$. We could take $h = (f+g)/2$ if both $f$ and $g$ are finite;
we take $h = 0$ if both are infinite; $h = g-1$ if only $f$ is infinite; and $h = f+1$ if only $g$ is infinite. Then $C$ is the disjoint
union of the $(\overset{m}{\overbrace{1,\ldots,1}},1)$-cells $C_1$ defined by $f$ and $h$ and $C_2$ defined by $h$ and $g$, and the
$(\overset{m}{\overbrace{1,\ldots,1}},0)$-cell $C_3$ given by the graph of $h$. Note that $C_3$ is definably homeomorphic to an
$(\overset{m}{\overbrace{1,\ldots,1}})$-cell so that $I(C_3) = (-1)^m I(\{p\})$ by the inductive hypothesis. Note further
that $C_1$ and $C_2$ are obviously definably homeomorphic to $C$. Hence we have
\begin{align*}
    I(C)
        &=      I(C_1) + I(C_2) + I(C_3)
        \\&=    2I(C) + (-1)^m I(\{p\}),
\end{align*}
implying that $I(C) = (-1)^{m+1}I(\{p\})$ and completing the proof.
\end{proof}

\begin{theorem}[Every additive topological invariant is a multiple of $\chi$]
\label{thrm:ECAddUniqueTop}
Let $I$ be a topological invariant of affine definable spaces. Then $I$ is additive if and only if $I(X) = \chi(X)I(\{p\})$
for each affine definable space $(X,\iota)$, where $p$ is a point.
\end{theorem}
\begin{proof}
 Assume $I$ is additive and let $(X,\iota)$ be an affine definable space. Since $X$ is definable, we may identify $X$ with $\iota(X)\subseteq\R^n$.
 Then $X$ can be written as a finite disjoint union of cells $C_1,\ldots,C_k$.  By additivity of $I$ and Lemma~\ref{lem:ECAddUnique},
\begin{align*}
    I(X)
        &=      \sum\limits_{i=1}^k I(C_i)
        \\&=    \sum\limits_{i=1}^k (-1)^{\dim C_i}I(\{p\})
        \\&=    \chi(X)I(\{p\}).
\end{align*}
Conversely, the additivity of $\chi$ clearly implies that $\chi(X)I(\{p\})$ is additive, completing the proof.
\end{proof}

In order to extend Lemma~\ref{lem:ECAddUnique}  and Theorem~\ref{thrm:ECAddUniqueTop} to orbit space definable groupoids, we need to introduce
a hypothesis to control the topology of the sets $\G_0$ and $\G_1$.

\begin{definition}
\label{def:IsoLocTriv}
Let $\G$ be an orbit space definable groupoid. We say that $\G$ is \emph{isotropy locally trivial} if there is a finite definable cell decomposition of
$\lvert\G\rvert$ subordinate to the partition into points with the same weak orbit type such that for each cell $C$ contained in
$\lvert\G\rvert_{[H]}$, where $H$ is an isotropy group of $\G$, we have $\G_{\vert \pi^{-1}(C)} = \G\ltimes\pi^{-1}(C)$ is topologically Morita equivalent to $H\ltimes C$
with trivial $H$-action.
We as well call a cell decomposition with this property \emph{isotropy locally trivial}.
\end{definition}

Let us indicate that isotropy local triviality is fairly common among orbit space definable groupoids.

\begin{proposition}
Let $\G$ be a cocompact proper Lie groupoid. Then $\G$ is isotropy locally trivial.
\end{proposition}
\begin{proof}
By \cite[Cor.~3.11]{PPTOrbitSpace}, there is for each point $x\in\G_0$ an open neighborhood $U_x \subseteq\G_0$ such that
$\G_{\vert U_x} = \G\ltimes U_x$ is isomorphic as a Lie groupoid
to $(\G_x^x\ltimes V_x)\times(O_x\times O_x)$ where $V_x$ is a $\G_x^x$-invariant open neighborhood of the origin in a linear representation $W_x$ of $\G_x^x$,
$O_x$ is a neighborhood of $x$ in its orbit, and $O_x\times O_x$ denotes the pair groupoid. As $\lvert\G\rvert$ is compact, there is
a finite collection $x_1,\ldots,x_m$ such that the $\pi(U_{x_i})$ cover $\lvert\G\rvert$. By \cite[Thrm.~5.7 and Cor.~7.2]{PPTOrbitSpace},
$\lvert\G\rvert$ admits a triangulation subordinate to the partition of $\lvert\G\rvert$ into orbits of points with the same weak orbit type
and such that the representations $W_x$ are isomorphic. We may by barycentric subdivision ensure
that each simplex is contained in some $\pi(U_{x_i})$. Using this triangulation to define an embedding $\iota_{\lvert\G\rvert}\colon\lvert\G\rvert\to\R^n$
as in the proof of \cite[Prop.~3.4]{FarSeaJLMS} gives $\G$ the structure of an orbit space definable groupoid.
Fix an open simplex $C$ contained in $\pi(U_{x_i})$ and identify $V_{x_i}$ with the subspace $V_{x_i}\times\{x_i\}$ in $U_{x_i}\simeq V_{x_i}\times O_{x_i}$.
Let $H$ be the isotropy group of a point in $\pi^{-1}(C)\cap V_{x_i}$ and note that each point of
$\pi^{-1}(C)\cap V_{x_i}$ has isotropy group conjugate to $H$ in $\G_{x_i}^{x_i}$; see \cite[Prop.~5.11]{PPTOrbitSpace}.
Letting $N$ denote the normalizer of $H$ in $\G_{x_i}^{x_i}$,
the set $(V_{x_i})_{H}$ of points in $U_{x_i}$ with $\G_{x_i}^{x_i}$-isotropy group equal to $H$ is
by \cite[Thrm.~4.3.10]{PflaumBook} a $N/H$-opposite principal bundle over $\pi\big((V_{x_i})_{H}\big)$. Restricting this bundle to $C$,
$\pi^{-1}(C)\cap(V_{x_i})_{H} \to C$ is a $N/H$-opposite principal bundle over $C$ so that as $C$ is a simplex, the bundle is trivial.
Then $\pi^{-1}(C)\cap(V_{x_i})_{H}$ is $N$-equivariantly diffeomorphic to $N/H\times C$
with the restriction of $\pi$ to $\pi^{-1}(C)\cap(V_{x_i})_{H}$ corresponding to the projection to $C$. Let $D$ be the submanifold of $\pi^{-1}(C)\cap(V_{x_i})_{H}$
corresponding to $\{eH\}\times C$ under this diffeomorphism, and then $\pi$ restricts to a diffeomorphism $D\to C$. Moreover
$\G_{\vert \pi^{-1}(C)} = \G\ltimes\pi^{-1}(C)$ and $\G_{\vert D} = \G\ltimes D$ are Lie Morita equivalent by
\cite[Prop.~3.7]{PPTOrbitSpace}. By construction, $gD \subseteq D$ for $g\in\G_{x_i}^{x_i}$ if and only if $g\in H$ so that $\G_{\vert D}$ is Lie groupoid
isomorphic to $H\ltimes C$ with trivial $H$-action, completing the proof.
\end{proof}

It follows that if $\G$ is a cocompact proper Lie groupoid and $U$ is a definable subset of $\lvert\G\rvert$, then the (not necessarily cocompact)
groupoid $\G_{\vert \pi^{-1}(U)} = \G\ltimes\pi^{-1}(U)$ is as well isotropy locally trivial.

The proofs of Lemma~\ref{lem:ECAddUnique}  and Theorem~\ref{thrm:ECAddUniqueTop}
extend to characterize additive invariants of orbit space definable groupoids when restricted to isotropy locally trivial groupoids.

\begin{lemma}
\label{lem:ECGpdAddUnique}
Let $I$ be an additive topological invariant of orbit space definable groupoids and let $\G$ be an isotropy locally trivial orbit space definable groupoid.
Suppose that $C$ is an $(i_1,\ldots,i_m)$-cell in an isotropy locally trivial cell decomposition of the orbit space $\lvert\G\rvert$ so that
$\G_{\vert \pi^{-1}(C)} = \G\ltimes\pi^{-1}(C)$
and $H\ltimes C$ are topologically Morita equivalent.
Then $I(\G_{\vert\pi^{-1}(C)}) = (-1)^{\dim C}I(H\ltimes\{p\})$ where $\pi:\G_0\to\lvert\G\rvert$ is the orbit map and $H\ltimes \{p\}$ is the point groupoid
with object space $\{p\}$ and arrow space $H$.
\end{lemma}
\begin{proof}
First observe that when $C$ is a point in the orbit space $\lvert\G\rvert$, $\G_{\vert\pi^{-1}(C)}$ is topologically Morita equivalent
to $H\ltimes \{p\} = \G_p^p\ltimes\{p\}$ by Definition~\ref{def:IsoLocTriv}. This follows as well from the fact that $\G_{\vert\pi^{-1}(C)}$ is transitive;
see \cite[Sec.~3.2]{WilliamsHaar}. If $\dim C > 0$, decomposing $C$ into cells $C_i$, $i=1,2,3$, as in the proof of Lemma~\ref{lem:ECAddUnique},
the Morita equivalence of $\G_{\vert \pi^{-1}(C)}$ with $H\ltimes C$ implies that for each $i$,
$\G_{\vert\pi^{-1}(C_i)}$ is topologically Morita equivalent to $H\ltimes C_i$.
As $C_i$, $i=1,2$, are both homeomorphic to $C$, it follows that
$\G_{\vert\pi^{-1}(C_i)}$ and $H\ltimes C$ are topologically Morita equivalent. With these observations, the proof of the lemma follows the
same reasoning as the proof of Lemma~\ref{lem:ECAddUnique}.
\end{proof}

\begin{proposition}
\label{prop:ECGpdAddUnique}
Let $I$ be a topological invariant of orbit space definable groupoids. Then $I$ is additive on isotropy locally trivial groupoids if and only if
$I$ is additive over the decomposition of the orbit space of a groupoid by weak orbit types and
$I(\G) = \chi(\lvert\G\rvert)I(G\ltimes\{p\})$ for each isotropy locally trivial orbit space definable groupoid $\G$ such that every isotropy group
is isomorphic as a topological group to a fixed compact Lie group $G$, where $p$ is a point.
\end{proposition}
\begin{proof}
 Assume $I$ is additive and let $\G$ be an isotropy locally trivial orbit space definable groupoid.
 Then using an isotropy locally trivial cell decomposition of $\lvert\G\rvert$,
 and Lemma~\ref{lem:ECGpdAddUnique}, the proof follows exactly as the proof of Theorem~\ref{thrm:ECAddUniqueTop}.
 Conversely, the additivity of $I$ over the decomposition by weak orbit types reduces to the case of a single weak orbit type,
 and then the additivity of $\chi$ ensures that $\chi(\lvert\G\rvert)I(G\ltimes\{p\})$ is an additive invariant of
 isotropy locally trivial orbit space definable groupoids with a single weak orbit type. 
 \end{proof}

We now turn to defining the universal Euler characteristic $\chi^{\un}$ for orbit space definable groupoids, a generalization of
\cite[Sec.~5]{GZLMH-Universal}.  We begin by defining the ring $\mathscr{R}$ which will serve as the codomain of $\chi^{\un}$.

Let $\mathscr{G}$ be the set of isomorphism classes, as topological groups, of compact Lie groups.  Let $T^{[G]}$ denote a formal variable corresponding to the isomorphism class $[G]\in\mathscr{G}$ of the group $G$.  Let $\mathscr{P}$ denote the ring of polynomials with integer coefficients over the set $\{T^{[G]} : [G]\in\mathscr{G}\}$ where we identify $T^{[e]}$ with the constant polynomial $1$. Let $\mathscr{R}=\mathscr{P}/\mathscr{I}$ where $\mathscr{I}$ is the ideal generated by the following polynomials,
\begin{equation}
\mathscr{I} = \langle T^{[G]}T^{[H]}-T^{[G\times H]} : [G], [H] \in \mathscr{G}\rangle.
\end{equation}
Thus, for example, $T^{[e]}+\mathscr{I}$ is the unity of $\mathscr{R}$ and, after simplifications of the form $T^{[G]}T^{[H]}=T^{[G\times H]}$, elements of $\mathscr{R}$ can be expressed as sums $\big(\sum_{[G]\in\mathscr{G}}a_{[G]}T^{[G]}\big)+\mathscr{I}$ where $a_{[G]}\in\mathbb{Z}$ and
$a_{[G]} = 0$ for all but finitely many $[G]$. In a slight abuse of notation, in the following we refer to a coset of $\mathscr{I}$ in $\mathscr{R}$
by a representative element $\sum_{[G]\in\mathscr{G} }a_{[G]}T^{[G]}$.

As in the case of orbifolds in \cite{GZLMH-Universal}, we can also represent an element of $\mathscr{R}$ as a polynomial in the $T^{[G]}$ corresponding
to \emph{indecomposable} $G$, i.e., $G$ such that if $G \simeq H_1\times H_2$ as compact Lie groups, then one of the $H_i$ is trivial. However,
let us observe that this expression is not in general unique, unlike the case of finite groups treated in \cite{GZLMH-Universal}. Indeed, there is no
Krull-Schmidt theorem for compact Lie groups, as was demonstrated to us by Jason DeVito. See \cite{DeVito} for a counterexample as well as a proof of
a Krull-Schmidt theorem for compact connected Lie groups whose center has dimension at most $1$.

We are now prepared to define the universal Euler characteristic of an orbit space definable groupoid $\G$. Recall that
$\lvert\G\rvert_{[G]}$ denotes the set of orbits in $\lvert\G\rvert$ of points with isotropy group isomorphic as topological groups to $G$.

\begin{definition}[Universal Euler characteristic]
\label{def:UnivEC}
Let $\G$ be an orbit space definable groupoid with orbit space $\lvert\G\rvert$. Using the notation introduced above,
the \textit{universal Euler characteristic of $\G$} is defined by
\begin{equation}\label{eqn:univEC}
\chi^{\un}(\G)=\sum_{[G]\in \mathscr{G}} \chi(\lvert\G\rvert_{[G]})T^{[G]}\in\mathscr{R}.
\end{equation}
Observe that by the definition of an orbit space definable groupoid, the sum in Equation~\eqref{eqn:univEC} is finite.
\end{definition}

We have the following.

\begin{lemma}
\label{lem:UnivECAddMult}
The universal Euler characteristic is an additive and multiplicative topological invariant of orbit space definable groupoids.
In particular, if $\G$ is an orbit space definable groupoid and $\HH$ is a topological groupoid that is topologically Morita equivalent to
$\G$, then $\chi^{\un}(\HH)$ is defined, and $\chi^{\un}(\G) = \chi^{\un}(\HH)$.
\end{lemma}
\begin{proof}
Let $\G$ be an orbit space definable groupoid and suppose $\HH$ is a topological groupoid that is topologically Morita equivalent to $\G$.
Then $\HH$ is orbit space definable and the sets $\lvert\G\rvert_{[G]}$ and $\lvert\HH\rvert_{[G]}$ are definably homeomorphic
by Lemma~\ref{lem:OrbitSpaceDefinMorita}. It follows that $\chi^{\un}(\HH)$ is defined and $\chi^{\un}(\G) = \chi^{\un}(\HH)$.

If $\G$ is an orbit space definable groupoid and
$\lvert\G\rvert = X_1\sqcup X_2$ for disjoint definable subsets $X_1, X_2\subseteq\lvert\G\rvert$, setting $\G_i = \G_{\vert\pi^{-1}(X_i)} = \G\ltimes\pi^{-1}(X_i)$
as in Definition~\ref{def:OrbSpDefGpdInvar}~(2),
clearly $\lvert\G_i\rvert_{[G]} = \lvert\G\rvert_{[G]}\cap \lvert\G_i\rvert$ for each $[G]\in\mathscr{G}$ and $i=1,2$
so that additivity of $\chi^{\un}$ follows from the additivity of $\chi$. Similarly, if $\G$ and $\HH$
are orbit space definable groupoids, then the identification of the $(\G\times\HH)$-orbit of $(x,y)\in\G_0\times\HH_0$
with $\G x\times \HH y \in \lvert\G\rvert\times\lvert\HH\rvert$ induces a homeomorphism
$\lvert\G\times\HH\rvert \to \lvert\G\rvert\times\lvert\HH\rvert$, and the isotropy group of
$(x,y)\in\G_0\times\HH_0$ is equal to $\G_x^x\times\HH_y^y$, from which multiplicativity follows.
\end{proof}

By virtue of this fact, the universal Euler characteristic can be used to define other additive and multiplicative topological
invariants of orbit space definable groupoids. The following is an immediate consequence of Lemma~\ref{lem:UnivECAddMult}.
Note that we do not require rings to be unital and hence do not require ring homomorphisms to preserve the multiplicative identity.

\begin{corollary}
\label{cor:UnivECcircRAddMult}
Let $R$ be a ring and let $r:\mathscr{R}\to R$ be a ring homomorphism. Then $r\circ\chi^{\un}(\G)$ defines
an additive and multiplicative topological invariant of orbit space definable groupoids.
\end{corollary}

Restricting to isotropy locally trivial orbit space definable groupoids, and noting that isotropy local triviality is preserved by finite products
and restrictions to preimages of definable subsets of the orbit space, the invariant $\chi^{\un}$ is universal in the sense of the following, which
is Theorem~\ref{main:ECAddUnique}.

\begin{theorem}[Every additive, multiplicative invariant of isotropy locally trivial orbit space definable groupoids is a composition with $\chi^{\un}$]
\label{thrm:ECAddUnique}
Let $I$ be a topological invariant of orbit space definable groupoids taking values in a ring $R$.
Then $I$ is additive and multiplicative on isotropy locally trivial orbit space definable groupoids if and only if there is a unique ring homomorphism $r:\mathscr{R}\to R$
such that $I(\G)=r\circ\chi^{\un}(\G)$ for each isotropy locally trivial orbit space definable groupoid $\G$.
\end{theorem}
\begin{proof}
Assume $I$ is additive and multiplicative on isotropy locally trivial groupoids and $\G$ is an isotropy locally trivial orbit space definable groupoid.
Then $\lvert\G\rvert$ can be partitioned into a finite number of affine definable spaces $\lvert\G\rvert_{[G]}$ such that each $G$ is compact Lie,
see Remark~\ref{rem:LieGroup}, and $\G_{\vert\pi^{-1}(\lvert\G\rvert_{[G]})} = \G\ltimes\pi^{-1}(\lvert\G\rvert_{[G]})$ is a subgroupoid of $\G$ having orbit
space $\lvert\G\rvert_{[G]}$.  Thus, by additivity of $I$ and Proposition~\ref{prop:ECGpdAddUnique},
\begin{equation}
\label{eq:ECAddSum}
    I(\G)=\sum_{[G]\in\mathscr{G}}\chi(\lvert\G\rvert_{[G]})I(G\ltimes\{p\}),
\end{equation}
where $p$ is a point and we recall that $\mathscr{G}$ denotes the set of isomorphism classes of compact Lie groups as topological groups.
For each $[G]\in\mathscr{G}$, let $r(T^{[G]})=I(G\ltimes\{p\})$ and, noting that $I(\{p\})$ is idempotent and acts as the identity on each $I(G\ltimes\{p\})$,
extend $r$ by additivity and multiplicativity to a ring homomorphism $\mathscr{R}\to R$.  Observe that since $(G\ltimes\{p\})\times (H\ltimes \{p\})$ is isomorphic as a topological groupoid to $(G\times H)\ltimes \{p\}$ for any compact Lie groups $G$ and $H$, we have that $(G\ltimes\{p\})\times (H\ltimes \{p\})$ and $(G\times H)\ltimes \{p\}$ are topologically Morita equivalent.  Thus, since $I$ is a multiplicative invariant of orbit space definable groupoids, $r$ is well-defined as a map $\mathscr{R}\to R$.  From the definitions of $I$, $\chi^{\un}$, and $r$, it is immediate that $I(\G)=r\circ\chi^{\un}(\G)$.

To see that $r$ is unique, suppose that $r^{\prime}:\mathscr{R}\to R$ is any ring homomorphism satisfying $I(\G)=r^\prime\circ\chi^{\un}(\G)$
for isotropy locally trivial orbit space definable groupoids $\G$.
Let $[G]\in\mathscr{G}$ and apply both sides of this equation to the point groupoid $G\ltimes \{p\}$.   Since the Euler characteristic of a point equals $1$, we have $I(G\ltimes\{p\})=r^{\prime}\circ\chi^{\un}(G\ltimes\{p\})=r^{\prime}(T^{[G]})$.  Thus $r^{\prime}=r$.

Conversely, if $r:\mathscr{R}\to R$ is a ring homomorphism, then $I(\G)=r\circ\chi^{\un}(\G)$ defines a topological invariant
of orbit space definable groupoids that is clearly additive and multiplicative on isotropy locally trivial orbit space definable groupoids.
\end{proof}

Let us consider a slight generalization of the notion of constructible function,
see Section~\ref{subsec:BackDefinSpaces}, \cite[Sec.~1-2]{ViroIntegralEulerChar}, and \cite[Sec.~4]{CurryGhristEAEulerCalc},
by allowing functions that take values in a ring other than $\Z$.
Given an affine definable space $X$ and a ring $R$, we say that a function $f:X\to R$ is \textit{constructible} if $f^{-1}(k)$ is a definable subset of
$X$ for each $k\in R$. If $f$ takes only finitely many values in $R$, then we may define the \textit{integral of $f$ over $X$ with respect to the Euler characteristic} by
\[
    \int_X f(x)\,d\chi(x) = \sum_{s\in R}\chi(f^{-1}(s))s.
\]
Observe that if $\G$ is an orbit space definable groupoid, and thus $\lvert\G\rvert$ is finitely partitioned according to weak orbit type, we can express
\[
    \chi^{\un}(\G)=\int_{\lvert\G\rvert} T^{[\G_x^x]}\, d\chi(\G x)
\]
where the integrand $T^{[\G_x^x]}$ is the $\mathscr{R}$-valued function that assigns to the orbit $\G x\in\lvert\G\rvert$ the variable $T^{[\G_x^x]}$ associated to the isotropy group $\G_x^x$, whose isomorphism class as a topological group depends only on the orbit $\G x$. With this, we may use Equation~\eqref{eq:ECAddSum} to reinterpret
Theorem~\ref{thrm:ECAddUnique} as follows.

\begin{corollary}
\label{cor:ECAddUniqueIntegral}
Let $I$ be a topological invariant of orbit space definable groupoids taking values in a ring $R$.
Then $I$ is additive and multiplicative on isotropy locally trivial groupoids if and only if there is a unique ring homomorphism $r:\mathscr{R}\to R$
such that
\begin{equation}
\label{eq:ECAddUniqueIntegral}
    I(\G)   =   \int_{\lvert\G\rvert} r(T^{[\G_x^x]})\, d\chi(\G x)
\end{equation}
\end{corollary}

We conclude this section with two examples.

\begin{example}
\label{ex:GammaEC}
Recall from \cite[Def.~3.8]{FarSeaJLMS} that for a finitely presented group $\Gamma$ and an orbit space definable groupoid $\G$, the \textit{$\Gamma$-Euler characteristic of $\G$} is given by
\begin{equation}
\label{eq:GammaECDef}
    \chi_{\Gamma}(\G)=\int_{\lvert\G\rvert}\chi\big(\G_x^x\backslash\HOM(\Gamma,\G_x^x)\big)\, d\chi(\G x),
\end{equation}
where $\G_x^x$ acts on $\HOM(\Gamma,\G_x^x)$ by pointwise conjugation. It is shown in \cite[Thrm.~4.13]{FarSeaJLMS} that $\chi_{\Gamma}(\G)$ is the usual
Euler characteristic of the $\Gamma$-inertia space of $\G$; see Definition~\ref{def:GammaInertia}.
In the case $\Gamma = \Z$, fixing a generator of $\Z$ identifies $\HOM(\Z,\G_x^x)$ with $\G_x^x$ so that Equation~\eqref{eq:GammaECDef} becomes
\begin{equation}
\label{eq:ZECDef}
    \chi_{\Z}(\G)=\int_{\lvert\G\rvert}\chi\big(\Ad_{\G_x^x}\backslash\G_x^x\big)\, d\chi(\G x).
\end{equation}
Observe that $\chi_{\Gamma}$ is an invariant of orbit space definable groupoids taking values in $\mathbb{Z}$ that is both additive and multiplicative.
Indeed, additivity of $\chi_{\Gamma}$ follows from the additivity of $\chi$ applied to the orbit space $\lvert\G\rvert$; see \cite[Cor.~4.16]{FarSeaJLMS}.
To see that multiplicativity holds, consider
\begin{align*}
\chi_{\Gamma}(\G\times\HH) &=\int_{\lvert\G\times\HH\rvert}\chi\big((\G\times\HH)_{(x,y)}^{(x,y)}\backslash \HOM(\Gamma, (\G\times\HH)_{(x,y)}^{(x,y)}\big)\big)\, d\chi\big((\G\times\HH)(x,y)\big)\\
&=\int_{\lvert\G\times\HH\rvert}\chi\big((\G_x^x\times\HH_y^y)\backslash\HOM(\Gamma,\G_x^x\times\HH_y^y)\big)\, d\chi\big((\G\times\HH)(x,y)\big)\\
&=\int_{\lvert\G\times\HH\rvert}\chi\big((\G_x^x\backslash\HOM(\Gamma,\G_x^x))\times(\HH_y^y\backslash\HOM(\Gamma,\HH_y^y))\big)\, d\chi\big((\G\times\HH)(x,y)\big)\\
&=\int_{\lvert\G\times\HH\rvert}\chi\big(\G_x^x\backslash\HOM(\Gamma,\G_x^x)\big)\chi\big(\HH_y^y\backslash\HOM(\Gamma,\HH_y^y)\big)\, d\chi\big((\G\times\HH)(x,y)\big)\\
&=\sum_{([\G_x^x],[\HH_y^y])\in\mathscr{G}\times\mathscr{G}} \chi\big(\G_x^x\backslash\HOM(\Gamma,\G_x^x)\big)\chi\big(\HH_y^y\backslash\HOM(\Gamma,\HH_y^y)\big)\chi\big(\lvert\G\rvert_{[\G_x^x]}\times\lvert\HH\rvert_{[\HH_y^y]}\big)\\
&=\sum_{[\G_x^x]\in\mathscr{G}}\sum_{[\HH_y^y]\in\mathscr{G}}\chi\big(\G_x^x\backslash\HOM(\Gamma,\G_x^x)\big)
    \chi\big(\HH_y^y\backslash\HOM(\Gamma,\HH_y^y)\big)\chi\big(\lvert\G\rvert_{[\G_x^x]}\times\lvert\HH\rvert_{[\HH_y^y]}\big)\\
&=\sum_{[\G_x^x]\in\mathscr{G}}\sum_{[\HH_y^y]\in\mathscr{G}}\chi\big(\G_x^x\backslash\HOM(\Gamma,\G_x^x)\big)
    \chi\big(\HH_y^y\backslash\HOM(\Gamma,\HH_y^y)\big)\chi\big(\lvert\G\rvert_{[\G_x^x]}\big)\chi\big(\lvert\HH\rvert_{[\HH_y^y]}\big)\\
&=\chi_{\Gamma}(\G)\chi_{\Gamma}(\HH);
\end{align*}
see also \cite[Lem.~4.17]{FarSeaJLMS}.

In this case, we see that the homomorphism $r:\mathscr{R}\to\mathbb{Z}$ guaranteed by Theorem~\ref{thrm:ECAddUnique} is given by
\[
    r(T^{[G]}) = \chi\big(G\backslash\HOM(\Gamma,G)\big),
\]
which corresponds to $\chi_{\Gamma}(G\ltimes\{p\})$ for a point $p$, because the Euler characteristic of a point is 1.
\end{example}

\begin{example}
\label{ex:GammaESC}
Similarly, one may consider the homomorphism $r:\mathscr{R}\to\Q$ such that
\[
    r(T^{[G]}) = \frac{1}{\chi(G^\circ\backslash G)}
\]
where $G^\circ$ denotes the connected component of the identity of $G$
so that $\chi(G^\circ\backslash G)$ is the number of connected components of $G$.
In the case $\G$ presents an orbifold so that each isotropy group is finite, this yields the \emph{Euler-Satake characteristic} $\chi^{ES}(\G)$ of $\G$;
see \cite{LeinsterEulerCharCategory,SatakeGB,Thurston} and note that, in the current framework, the hypothesis that the orbifold is closed is unnecessary.
For a general orbit space definable groupoid, the resulting additive and multiplicative invariant given by
Theorem~\ref{thrm:ECAddUnique} can be expressed
\[
    \chi^{ES}(\G)   =  \int_{\lvert\G\rvert} \frac{1}{\chi(G^\circ\backslash G)}\, d\chi(\G x).
\]
We can similarly fix a finitely presented discrete group $\Gamma$ and let
\[
    r(T^{[G]}) = \int_{G\backslash\HOM(\Gamma,G)} \frac{1}{\chi\big(C_G(\phi)^\circ\backslash C_G(\phi)\big)} \, d\chi(G\phi),
\]
where $C_G(\phi)$ denotes the centralizer of $\phi$ in $G$, $\chi\big(C_G(\phi)^\circ\backslash C_G(\phi)\big)$ is the number of connected components
of $C_G(\phi)$, and $G\phi$ denotes the $G$-orbit of $\HOM(\Gamma,G)$ with $G$ acting by conjugation.
To see how this integral is defined, note that the $G$-space $\HOM(\Gamma,G)$ can be identified with an algebraic subset of $G^\ell$ where the $G$-action is by simultaneous conjugation; see \cite[Not.~3.7]{FarSeaJLMS}.
As this $G$-action on $G^\ell$ has finitely many isotropy types,
the integrand $1/\chi\big(C_G(\phi)^\circ\backslash C_G(\phi)\big)$ is constructible.
This defines a generalization of the
$\Gamma$-Euler-Satake characteristic \cite{FarsiSeaGenOrbEuler,TamanoiMorava,TamanoiCovering} from orbifolds to orbit space definable groupoid given by
\begin{equation}
\label{eq:DefGammaESC}
    \chi_\Gamma^{ES}(\G)   =   \int_{\lvert\G\rvert}
                                \int_{\G_x^x\backslash\HOM(\Gamma,\G_x^x)}
                                    \frac{1}{\chi\big(C_{\G_x^x}(\phi)^\circ\backslash C_{\G_x^x}(\phi)\big)} \, d\chi(\G_x^x\phi) \, d\chi(\G x).
\end{equation}
See \cite[Rem.~2.11, Rem.~3.11, Ex.~3.12]{FarSeaJLMS} for a discussion of these topological invariants of orbit space definable groupoids.
\end{example}

\section{The Burnside group of a definable groupoid}
\label{sec:Burnside}

Let $\G$ be a definable groupoid.
In this section, we define the \emph{Burnside group} $B(\G)$ of $\G$,
which when $\G$ is transitive carries as well a partial product, analogous to the Burnside ring of a compact Lie group $G$, see \cite{tomDieckBurnsideI} and
\cite[Ch.~5]{tomDieckTGRepThry}. We will as well define an \emph{equivariant Euler characteristic} that takes values in $B(\G)$.
See \cite{KaoSpinCatfied,KaoSpinBurnsideTheory} for a well-developed theory of Burnside rings in the case of a finite groupoid.

Each element of the set $B(\G)$ will be represented by a $\G$-space $X$ that has a definable $\G$-quotient $\lvert X\rvert$,
a definable orbit type partition, and satisfies two finiteness hypotheses; see Definitions~\ref{def:OrbitTypeG-Space} and \ref{def:BurnsideSet}.
Therefore, we henceforth assume that $\G$ has finitely many orbit types, a definable quotient $\lvert\G\rvert$, and a definable orbit type partition;
recall that $\G$ is said to have a definable quotient when the $\G$-space $\G_0$ has a definable $\G$-quotient.
This is the case when $\G$ is proper and source-proper by \cite[Prop.~3.5]{FarSeaJLMS},
or more generally when the equivalence relation on $\G_0$ of being in the same orbit is definably proper;
see Definition~\ref{def:DefinProperEquiv} and \cite[Ch.~10, (2.13) and Thrm.~(2.15)]{vandenDriesBook}.
We further restrict our attention to definable $\G$-spaces $X$
that have finitely many orbit types, a definable $\G$-quotient $\lvert X\rvert$, and a definable orbit type partition; see
Definitions~\ref{def:OrbitTypeGrpoid}, \ref{def:DefinGSpace}, and \ref{def:OrbitTypeG-Space} as well as Remark~\ref{rem:DefinLieFiniteOrbit}.
Note that each set in the orbit type partition of $X$ is a definable $\G$-space with definable $\G$-quotient so that as
the orbit map $\pi_X\colon X\to\lvert X\rvert$ is definable, the orbit type partition of $X$ induces a
partition of $\lvert X\rvert$ into definable sets.
We do not know of an example of a definable groupoid $\G$ or $\G$-space $X$ where the orbit type partition is not definable.

In the case of a Lie group $G$, the Burnside ring is generated by transitive $G$-spaces, which arise as orbits of a $G$-action;
see \cite[Prop.~5.5.1]{tomDieckTGRepThry}. We will see that there are parallels in the definable groupoid case.

\subsection{The underlying set of the Burnside group}
\label{subsec:BurnsideSet}

As explained above, the elements of the Burnside group $B(\G)$ will be equivalence classes of $\G$-spaces.
In order to define the appropriate equivalence relation, we will need the following two definitions. First,
we adapt the notions of a bisection and corresponding adjoint isomorphisms to the context of definable groupoids;
see \cite[Sec.~1.4]{MackenzieGT}.

\begin{definition}[Bisection, adjoint isomorphism]
\label{def:Bisection}
Let $\G$ be a definable groupoid and $S, T$ be definable subsets of $\G_0$. A \emph{definable bisection of $\G$ over $S$ with target $T$} is a definable
continuous function $\sigma\colon S\to G_1$ such that $s\circ\sigma = \id_S$ and such that the map $t\circ\sigma$,
denoted $\Ad(\sigma)_0$, is a definable homeomorphism of $S$ onto $T$.

Given a bisection $\sigma$ of $\G$ over $S\subseteq\G_0$ with target $T = \Ad(\sigma)_0(S)\subseteq\G_0$, the \emph{adjoint isomorphism} associated to $\sigma$
is the definable isomorphism $\Ad(\sigma)\colon \G_{\vert S} = \G\ltimes S \to \G_{\vert T} = \G\ltimes T$ whose map on objects is given
by $\Ad(\sigma)_0$ and whose map on arrows is given by $\Ad(\sigma)_1(h) = (\sigma\circ t(h))h(\sigma\circ s(h))^{-1}$.
Note that $\sigma\circ s(h)$ has source $s(h)$ and $\sigma\circ t(h)$ has source $t(h)$ so that $\Ad(\sigma)_1(h)$ is an arrow from
$\Ad(\sigma)_0\big(s(h)\big)\in T$ to $\Ad(\sigma)_0\big(t(h)\big)\in T$.
\end{definition}

Note that in \cite{MackenzieGT}, Mackenzie is working with Lie groupoids and only considers the case that $S$ is open. He refers to the adjoint
isomorphism $\Ad(\sigma)$ as the \emph{local inner automorphism defined by $\sigma$}; see \cite[Def.~1.4.8]{MackenzieGT}.

\begin{definition}[Admissible subgroupoid]
\label{def:AdmissableSubgroupoid}
We say that a source-open definable subgroupoid $\HH$ of $\G$ is \emph{admissible} if
\begin{enumerate}
\item[i.]   $\HH$ is wide, i.e., $\HH_0 = \G_0$;
\item[ii.]  $\HH$ is closed, i.e., $\HH_1$ is a closed subset of $\G_1$;
\item[iii.] whenever $p_1,p_2\in\G_0$ are in distinct $\G$-orbits and have the same orbit type in the groupoid $\G$ so that there is a definable isomorphism
            $\varphi\colon\G_{\vert \G p_1} = \G\ltimes(\G p_1)\to\G_{\vert \G p_2} = \G\ltimes(\G p_2)$, there is for any such $\varphi$ a definable bisection $\sigma$ of
            $\G$ over $\G p_2$ with target $\Ad(\sigma)_0(\G p_2) = \G p_2$
            such that the definable isomorphism $\Ad(\sigma)\circ\varphi\colon\G_{\vert \G p_1} = \G\ltimes(\G p_1)\to\G_{\vert \G p_2} = \G\ltimes(\G p_2)$
            restricts to a definable isomorphism $\HH_{\vert \G p_1} = \HH\ltimes(\G p_1)\to\HH_{\vert \G p_2} = \HH\ltimes(\G p_2)$; and
\item[iv.]  whenever $p_1,p_2\in\G_0$ are in the same $\G$-orbit, there is a definable bisection $\sigma$ of
            $\G$ over $\HH p_1$ with target $\Ad(\sigma)_0(\HH p_1) = \HH p_2$ such that $\Ad(\sigma)$ restricts to a definable isomorphism
            $\HH_{\vert \HH p_1} = \HH\ltimes(\HH p_1)\to\HH_{\vert \HH p_2} = \HH\ltimes(\HH p_2)$.
\end{enumerate}
\end{definition}

We remark that if $p_1$ and $p_2$ are in the same $\HH$-orbit,
then Condition~iv.\@ of Definition~\ref{def:AdmissableSubgroupoid} is trivial because $\sigma$ can be taken to be the unit map restricted to $\HH p_1 = \HH p_2$.
Furthermore, observe that Conditions~iii.\@ and iv.\@ imply that the orbit type partition of the groupoid $\HH$ is coarser than
the orbit type partition of the groupoid $\G$.

\begin{example}
\label{ex:AdmissibleGUnitTransitive}
Clearly, $\G$ is always an admissible subgroupoid of $\G$; the bisections in Conditions~iii.\@ and iv.\@ can be taken as above to be the restrictions of the unit map.
Similarly, the subgroupoid $\HH$ with $\HH_1 = u(\G_0)$, consisting only of units, is an admissible subgroupoid, again taking the bisection in Condition~iii.\@
to be the restriction of the unit map and the bisection in Condition~iv.\@ to map $p_1$ to an arrow from $p_1$ to $p_2$.

If $\G$ is transitive, then a source-open, closed, wide, definable subgroupoid $\HH$ of $\G$ is admissible provided
Condition~iv.\@ holds, i.e., for any $p_1, p_2\in\G_0$, there
is a definable bisection $\sigma$ of $\G$ over $\HH p_1$ with target $\HH p_2$ such that $\Ad(\sigma)$ is a definable isomorphism $\HH\ltimes(\HH p_1)\to\HH\ltimes(\HH p_2)$.
In this case, Condition~iii.\@ is trivial.
\end{example}

\begin{example}
\label{ex:AdmissibleBundle}
If $\G$ is a singular bundle of groups over $X$, see Definition~\ref{def:SingBundGrps}, then it is possible that a closed, wide subgroupoid
$\HH$ of $\G$ is not source-open. For instance, let $\G$ be as in Example~\ref{ex:SingBundGrpsR}, and let $\HH$ be the subgroupoid with
$\HH_0 = \R$ and $\HH_1 = \{ (h,x) \in\G_1 : h\in\Sp^1 \text{ if } x = 1, h = e \text{ if } x\neq 1 \}$. Then $\HH$ is clearly closed and wide, but if $U$
is a neighborhood of a nontrivial arrow $(h, 1)$ small enough so that $U$ does not contain $(e, x)$ for any $x\neq 1$,
then $s(U) = \{1\}$ is not open. Hence, our restriction to source-open topological groupoids implies that such an $\HH$ is not considered
an admissible subgroupoid of $\G$.

Now, following the notation of Definition~\ref{def:SingBundGrps},
let $\HH$ be a source-open, closed, wide subgroupoid of the singular bundle of groups $\G$. Condition~iii.\@ in Definition~\ref{def:AdmissableSubgroupoid}
implies that if $x, y\in \G_0$, $x\neq y$, have isomorphic isotropy, i.e., are contained in the interiors of simplices of the same dimension, then for
any definable isomorphism $\varphi\colon \G_x^x\to\G_y^y$, there is a $g\in\G_y^y$ such that $g\varphi(\HH_x^x) g^{-1} = \HH_y^y$.
This is in this case a strong condition, equivalently, the image of $\HH_y^y$ under any definable automorphism of $\G_y^y$ is conjugate to
$\HH_y^y$ in $\G_y^y$. Condition~iv.\@ is in this case trivial, as we must have $p_1 = p_2$ so that the bisection can again be taken
to be the restriction of the unit map.

Finally, let $\HH$ be the singular bundle of groups over $X$, with the same definable triangulation that is used to define $\G$, associated to a choice
$H_0^\prime \lneq H_1^\prime \lneq \cdots\lneq H_d^\prime$ where each $H_i^\prime$ is a closed subgroup of $H_i$ whose image under any automorphism of $H_i$
coincides with its image under an inner automorphism of $H_i$. It follows from the above observations that $\HH$ is an admissible subgroupoid of $\G$.
\end{example}

\begin{example}
\label{ex:AdmissibleO2}
Let $\G$ be the groupoid in Example~\ref{ex:O(2)onR}, and let $\HH$ be the wide subgroupoid of $\G$ such that $\HH_1$ is given by arrows $(g, x)$
with $g\in\SO_2(\R)\simeq\Sp^1$ when $x\neq 0$ as well as $(e,0)$. Then $\HH$ is clearly closed and source-open.
If $x\neq 0$, then the $\G$-orbit $\{\pm x\}$ of $x$ is the union of
two $\HH$-orbits, which are singletons. To verify Condition~iii.\@ of Definition~\ref{def:AdmissableSubgroupoid}, suppose $x,y\in\R$ are nonzero such that $x\neq \pm y$.
Any isomorphism $\varphi\colon\G\ltimes(\G x)\to\G\ltimes(\G y)$ must map units to units. If $\varphi_0(x) = y$, then $\varphi(-x) = -y$, and $\varphi$ must restrict to
an isomorphism mapping $\SO_2(\R)\times\{x\}$ onto $\SO_2(\R)\times\{y\}$
by continuity; hence, $\sigma(y) = u(y)$ satisfies the condition. If $\varphi_0(x) = -y$, then $\sigma(-y)$ can be taken to be
any arrow from $-y$ to $y$, i.e., $(h, -y)$ for any $h\in\OO_2(\R)\smallsetminus\SO_2(\R)$, and we can set $\sigma(y) = (h,-y)^{-1} = (h^{-1}, y)$.
Then $\Ad(\sigma)\circ\varphi$ restricts to a definable groupoid isomorphism $\SO_2(\R)\times\{x\}\to\SO_2(\R)\times\{y\}$.
Similarly, Condition~iv.\@ of Definition~\ref{def:AdmissableSubgroupoid} is satisfied for $\HH$.
If $x\in\R$ is nonzero, then a definable isomorphism $\HH\ltimes(\HH x)\to\HH\ltimes(\HH (-x))$ is given by the bisection
$\sigma$ that maps $x$ to $(h,x)$ where $h$ is any choice of element of $\OO_2(\R)\smallsetminus\SO_2(\R)$.
In particular, $t\circ\sigma$ maps $x$ to $-x$, and $\Ad(\sigma)$ is a definable isomorphism of $\HH\ltimes\{x\}$ onto
$\HH\ltimes\{-x\}$. Hence $\HH$ is admissible.
\end{example}

\begin{example}
\label{ex:AdmissibleLieGrp}
Let $G$ be a compact Lie group considered as a groupoid with a single object. Then Condition~i.\@ of Definition~\ref{def:AdmissableSubgroupoid} is trivial,
Condition~iii.\@ is vacuous, and in Condition~iv.\@, $p_1 = p_2$ so that $\sigma$ can be taken to be the restriction of the unit map; each subgroupoid
of $G$ is source-open so that the admissible subgroupoids are the closed subgroups $H$ of $G$.
\end{example}

Suppose $X$ is a $\G$-space such that the equivalence relation of being in the same $\G$-orbit is definably proper and let
$\HH$ be an admissible subgroupoid of $\G$. We now show that $X$ has a definable $\HH$-quotient provided $\G$ is source-proper.

\begin{proposition}
\label{prop:HOrbitMapDefinProper}
Suppose $\G$ is source-proper and the equivalence relation on $X$ of being in the same $\G$-orbit is definably proper.
If $\HH$ is an admissible subgroupoid of $\G$, then the equivalence relation on $X$ of points being in the same
$\HH$-orbit is definably proper. Hence, the $\HH$-space $X$ has a definable $\HH$-quotient.
\end{proposition}
\begin{proof}
Let $E_\G = \{ (x,y)\in X\times X : \G x = \G y\}$ and $E_\HH = \{ (x,y)\in X\times X : \HH x = \HH y\}$,
and let $\pr_i^\G\colon E_\G\to X$ and $\pr_i^{\HH}\colon E_\HH\to X$ denote the projections onto the $i$th factor, $i=1,2$.
Let $K$ be a definable compact subset of $X$, and then $(\pr_1^\G)^{-1}(K)$ is definable and compact by hypothesis. We claim that
$(\pr_1^\HH)^{-1}(K)$ is a closed subset of $(\pr_1^\G)^{-1}(K)$ and hence compact. Suppose $(x_i, y_i)$ is a convergent sequence of elements of $(\pr_1^\HH)^{-1}(K)$ so that for each $i$, there is an $h_i\in\HH$ such that $h_i \ast x_i = y_i$.
Let $x = \lim_{i\to\infty} x_i$ and $y = \lim_{i\to\infty} y_i$. Because $K$ is compact and each $x_i\in K$, we have $x\in K$.
Because $\G$ is source-proper, $s_\G^{-1}\big(\alpha(K)\big)$ is compact. By restricting to a subsequence, we can assume the $h_i$ converge to some $h\in\G$,
and then as $\HH$ is a closed subgroupoid of $\G$, we have $h\in\HH$. Then by continuity, $hx = y$ so that
$(x,y)\in(\pr_1^\HH)^{-1}(K)$, and $(\pr_1^\HH)^{-1}(K)$ is compact. It follows that $E_\HH$ is a definably proper equivalence relation and hence by
\cite[Ch.~10, Thrm.~(2.15)]{vandenDriesBook} that $X$ has a definable $\HH$-quotient.
\end{proof}

The elements of the Burnside group $B(\G)$ are defined in a manner analogous to the case of a compact Lie group in \cite[Sec.~5.5]{tomDieckTGRepThry}. Elements of $B(\G)$
are represented by definable $\G$-spaces $X$, where we replace tom Dieck's hypotheses controlling the topology of $X$, called $G$-ENR's in that reference, with
Conditions~i.\@ through iv.\@ in Definition~\ref{def:BurnsideSet} below.
Two such $\G$-spaces represent the same class in $B(\G)$ if their fixed-point sets with respect
to the action of any admissible subgroupoid have the same Euler characteristic. Recall that if $G$ is a topological group and $X$ is a $G$-space,
the \emph{isotropy type} of a point $x\in X$ is the $G$-conjugacy class of the isotropy group of $x$.

\begin{definition}[Underlying set of the definable Burnside group]
\label{def:BurnsideSet}
Let $\G$ be a definable groupoid with finitely many orbit types, a definable quotient, and a definable orbit type partition
and let $(X, \alpha)$ be a definable $\G$-space with anchor map $\alpha$.
If $\HH$ is a subgroupoid of $\G$, then $X^{\HH}$ denotes the set of fixed points of $\HH$ in $X$, i.e., the points $x\in X$ such that
$h\ast x = x$ for each $h\in s^{-1}(\alpha(x))\cap\HH_1$. Note that $X^{\HH}$ is a definable subset of $X$.

A definable $\G$-space $X$ is \emph{admissible} if
\begin{enumerate}
\item[i.]   $X$ has finitely many orbit types as a $\G$-space;
\item[ii.]  $X$ has a definable $\G$-quotient;
\item[iii.] $X$ has a definable orbit type partition, equivalently, the corresponding partition of $\lvert X\rvert$ is definable; and
\item[iv.]  if the image $\alpha(X)$ of $\alpha$ contains a point $p$ such that every element of $\G_0$ with the same orbit type as $p$ is in the orbit $\G p$,
            then $\G_p^p$ is compact and the $\G_p^p$-action on $\alpha^{-1}(p)$ has finitely many isotropy types.
\end{enumerate}
Consider the following equivalence relation on admissible $\G$-spaces. We say $X\equiv X^\prime$ if
\[
    \chi(X^\HH) = \chi\big((X^\prime)^\HH\big)
\]
for any admissible subgroupoid $\HH$ of $\G$.
We let $[X]$ denote the $\equiv$-class of the admissible $\G$-space $X$ and define the set $B(\G)$ to be the set of $\equiv$-classes $[X]$.
\end{definition}

Note that if a definable $\G$-space $(X,\alpha)$ is admissible, Condition~iv.\@ requires that $\alpha(X)$ does not contain points $p$
such that $\G_p^p$ is not compact and every point in $\G_0$ with the same orbit type as $p$ is in the orbit $\G p$. That is, a noncompact
isotropy group $\G_p^p$ can act on an admissible space $X$ only if there is a $q\in\G_0$ that has the same orbit type as $p$ in $\G$
with $q\notin\G p$.

In the case of a $G$-space $X$ where $G$ is a group, the class represented by $X$ in the Burnside group is sometimes referred to as the equivariant
Euler characteristic of $X$; see \cite[Sec.~3]{GZ-EquivariantAnalogues}. We therefore make the following definition.

\begin{definition}[Equivariant Euler characteristic]
\label{def:EqEC}
Let $\G$ be a definable groupoid with finitely many orbit types, a definable quotient, and a definable orbit type partition
and let $(X,\alpha)$ be an admissible $\G$-space.
The \emph{equivariant Euler characteristic} $\chi_{\G}^{\eq}(X)$ of $X$ is the class $[X] \in B(\G)$.
If the isotropy groups of $\G$ are all compact, the \emph{equivariant Euler characteristic} of $\G$ is defined to be $\chi^{\eq}(\G) = \chi_{\G}^{\eq}(\G_0)$.
\end{definition}

Suppose $G$ is a compact Lie group considered as a groupoid with one object. If $X$ is a definable $G$-space, then
Condition iv.\@ of Definition~\ref{def:BurnsideSet} implies Condition i. Moreover, $X$ in this case always has a definable $G$-quotient
by \cite[Ch.~10, Cor.~(2.18)]{vandenDriesBook}. Hence,
the Burnside group $B(G)$ is generated by equivalence classes of definable $G$-spaces $X$ with finitely many
isotropy types and definable orbit type partitions under the equivalence $[X] = [Y]$ if $\chi(X^H) = \chi(Y^H)$ for any closed subgroup $H\leq G$;
see Example~\ref{ex:AdmissibleLieGrp}. Therefore, $B(G)$ is in this case an analog of
the Burnside ring of $G$ defined in \cite[Sec.~5.5]{tomDieckTGRepThry} using definable $G$-spaces having finitely many isotropy types
and a definable orbit type partition in place of the $G$-ENR defined in \cite[Sec.~5.2]{tomDieckTGRepThry}.

\subsection{The operation of the Burnside group}
\label{subsec:BurnsideGroup}

To give $B(\G)$ the structure of an abelian group, we define $[X] + [Y] = [X \sqcup Y]$ where $X\sqcup Y$ has the obvious
$\G$-action: If $\alpha\colon X\to\G_0$ and $\beta\colon Y\to\G_0$ are the anchor maps of $X$ and $Y$, respectively, then $X\sqcup Y$
has anchor map $\alpha\sqcup\beta$, and the action of $g\in\G_1$ on $x\in X\sqcup Y$ is the action on $X$ if
$x\in X$ and $Y$ if $x\in Y$. Clearly, this operation is well-defined; as $(X \sqcup Y)^{\HH} = X^{\HH} \sqcup Y^{\HH}$ for any admissible
subgroupoid $\HH$ of $\G$, if $[X] = [X^\prime]$ and $[Y] = [Y^\prime]$, then $[X \sqcup Y] = [X^\prime \sqcup Y^\prime]$. As well,
$+$ is a commutative and associative operation with identity given by the class $[\emptyset]$ of the empty set $\emptyset$.
Let us demonstrate the existence of additive inverses following \cite{tomDieckBurnsideI}.

\begin{proposition}
\label{prop:BurnsideAdditiveInverse}
Let $\G$ be a definable groupoid with finitely many orbit types, a definable quotient, and a definable orbit type partition and
let $(X,\alpha)$ be an admissible $\G$-space.
Then $[X]$ has an additive inverse in $B(\G)$.
\end{proposition}
\begin{proof}
Let $S$ be a definable set with $\chi(S) = -1$, e.g., we can take $S$ to be the open interval $(0, 1)$.
Define $X^- = S \times X$, which we make a $\G$-space as follows. The anchor map $\alpha_-$ is given by
\[
    \alpha_- (s, x) = \alpha(x),
\]
and the action map is $g\ast(s, x) = (s, g\ast x)$. Conditions~i.\@ through iv.\@ of Definition~\ref{def:BurnsideSet}
are satisfied for $X^-$ as each follows from the corresponding condition for $X$.
For each admissible subgroupoid $\HH$ of $\G$, we have
$(X^-)^\HH = S \times X^\HH$ so that
\begin{align*}
    \chi\big( (X \sqcup X^-)^\HH \big)
    &=      \chi\big( X^\HH \sqcup (S \times X^\HH) \big)
    \\&=    \chi\big( X^\HH \big) + \chi\big( S \big)\chi\big( X^\HH \big)
    \\&=    \chi\big( X^\HH \big) - \chi\big( X^\HH \big) = 0.
\end{align*}
Hence $[X] + [X^-] = [X\sqcup X^-]$ is equivalent to the additive identity $[\emptyset]$, implying $[X^-] = -[X]$.
\end{proof}

Let us consider a slight generalization of the construction in the proof of Proposition~\ref{prop:BurnsideAdditiveInverse}.

\begin{definition}[Trivial product $\G$-space]
\label{def:TrivialProduct}
Let $\G$ be a definable groupoid with finitely many orbit types, a definable quotient, and a definable orbit type partition and
let $(X,\alpha)$ be an admissible $\G$-space.
Let $Y$ be any definable space. Define the \emph{trivial product $\G$-space of $Y$ and $X$} to be the $\G$-space $Y\ttimes X = (Y\times X, \alpha_{\triv})$
with anchor map
\[
    \alpha_{\triv} (y, x) = \alpha(x)
\]
and action map $g\ast(y, x) = (y, g\ast x)$. We will sometimes use the notation $Y\ttimes X$ for the space $Y\times X$ to emphasize that the anchor and $\G$-action
are those given above.
\end{definition}

As in the case of $X^-$, $Y\ttimes X$ is easily seen to represent a class in $B(\G)$.

We now detail a sequence of lemmas which will lead to the proof of Theorem~\ref{main:BurnsideGenTransitive}, that $B(\G)$ is generated as a group
by the classes of transitive $\G$-spaces.
We start by noticing that for each admissible subgroupoid $\HH$ of $\G$, we have
$(Y\ttimes X)^\HH = Y\ttimes X^\HH$ so that
\begin{align*}
    \chi\big( (Y\ttimes X)^\HH \big)
    &=      \chi\big( Y\ttimes X^\HH \big)
    \\&=    \chi\big( Y \big)\chi\big( X^\HH \big).
\end{align*}
In particular, assume $\chi(Y) > 0$ and let $F$ be a finite set consisting of $\chi(Y)$ points. Then
\[
    [Y\ttimes X] =   [F\ttimes X] = \big[ \overset{\chi(Y)}{\overbrace{X\sqcup\cdots\sqcup X}} \big]
                =   \overset{\chi(Y)}{\overbrace{[X] + \cdots + [X]}}.
\]
If $\chi(Y) < 0$, then we similarly have
\[
    [Y \ttimes X] =   \overset{-\chi(Y)}{\overbrace{-[X] - \cdots - [X]}}
\]
and if $\chi(Y) = 0$, then $[Y\ttimes X] = [\emptyset]$.
Summarizing the above discussion, we thus obtain the first of the lemmas which will be used to prove Theorem~\ref{main:BurnsideGenTransitive}.

\begin{lemma}
\label{lem:TrivialProduct}
Let $\G$ be a definable groupoid with finitely many orbit types, a definable quotient, and a definable orbit type partition;
let $(X,\alpha)$ be an admissible $\G$-space; and let $Y$ be any definable space. Then
\[
    \textup{$[Y\ttimes X] =   \chi(Y)[X].$}
\]
\end{lemma}

\begin{remark}
\label{rem:BurnsideInfiniteOrbitTypes}
Up to this point, the hypotheses that $\G$ has finitely many orbit types, a definable quotient, and a definable orbit type partition,
and that $X$ satisfies Conditions~i.\@ through iv.\@ of Definition~\ref{def:BurnsideSet}, were not necessary,
and an analogous definition of $B(\G)$ could be given consisting of classes $[X]$ of arbitrary definable $\G$-spaces $X$ with no further restrictions.
The proof of Proposition~\ref{prop:BurnsideAdditiveInverse} made no use of the hypotheses of finite orbit types, a definable quotient, a definable orbit type partition
for $\G$ and $X$, nor of Definition~\ref{def:BurnsideSet}~iv., so this would as well yield the definition of a group; similarly,
Lemma~\ref{lem:TrivialProduct} also did not use these hypotheses. However, it is unclear whether a definable groupoid $\G$ or definable $\G$-space $X$
can have infinitely many orbit types, or whether the definable $\G_p^p$-space $\alpha^{-1}(p)$ for $p\in\G_0$ can have infinitely many
isotropy types. Moreover, the more detailed description of $B(\G)$ that we now develop requires these restrictions.
Note that the hypothesis that $\G$ has finitely many orbit types, a definable quotient, and a definable orbit type partition
ensures that $\G_0$ will always represent a class in $B(\G)$ if the isotropy groups of $\G$ are compact. In particular, as the $\G_p^p$-orbit of $p\in\G_0$ is a singleton,
Definition~\ref{def:BurnsideSet}~iv.\@ holds as long as $\G_p^p$ is compact.
\end{remark}

Now, let $X$ and $\G$ be as above and let $X = A_1\sqcup A_2 \sqcup\cdots\sqcup A_r$ denote the orbit type partition.
Then, since the orbit type partition is definable and each $A_i$ is clearly admissible, $[X] = [A_1] + \cdots + [A_r]$ in $B(\G)$. That is, we have the following.

\begin{proposition}
\label{prop:BurnsideSingleOrbitType}
Let $\G$ be a definable groupoid with finitely many orbit types, a definable quotient, and a definable orbit type partition.
Then $B(\G)$ is generated as an abelian group by the classes
of admissible $\G$-spaces that have a single orbit type. Specifically, if $X$ is an admissible $\G$-space and
$X = \bigsqcup_{i=1}^r A_i$ is the orbit type partition of $X$, then $[X] = \sum_{i=1}^r [A_i]$.
\end{proposition}

We next show that $B(\G)$ is actually generated by the classes of transitive $\G$-spaces; see Lemma~\ref{lem:BurnsideSingleOrbit} and
Theorem~\ref{thrm:BurnsideGenTransitive}. Specifically, the class $[X]$ of an admissible $\G$-space $(X,\alpha)$
can be expressed as a linear combination of classes of orbits in $X$.
Proposition~\ref{prop:BurnsideSingleOrbitType} allows us to reduce to the case that
$X$ has a single orbit type.
The further decomposition into classes of orbits depends on whether $\alpha$ maps distinct orbits in $X$ to distinct orbits in $\G_0$ and, if not,
whether for any point $x\in X$, the orbit $\G\alpha(x)$ contains every point of the same orbit type as $\alpha(x)$ in $\G_0$.
In the former case, Definition~\ref{def:AdmissableSubgroupoid}~iii.\@ limits the possible admissible subgroupoids of $\G$ enough that
these distinct orbits represent the same class in $B(\G)$; the same happens if $\alpha$ maps distinct orbits in $X$ to the same orbit in $\G_0$
but there are additional orbits in $\G_0$ of the same orbit type.
If, however, $\alpha$ maps distinct orbits in $X$ to the same orbit in $\G_0$ and this is the only orbit of its orbit type in $\G_0$,
then we are in a situation very similar to the case of a compact Lie group, and the situation resembles that of \cite[Sec.~5.5]{tomDieckTGRepThry}.
The technical Lemmas~\ref{lem:BurnsideHFixIndependentDiffAlpha}, \ref{lem:BurnsideHFixIndependentSameAlpha}, and \ref{lem:BurnsideHFixIndependentOneGpOrbitType}
will address these three cases, and Lemma~\ref{lem:BurnsideSingleOrbit} will combine these to express the class $[X]$ of such a $\G$-space as a finite linear
combination of classes of $\G$-orbits. Note that the hypotheses of Lemmas~\ref{lem:BurnsideHFixIndependentDiffAlpha} and
\ref{lem:BurnsideHFixIndependentSameAlpha} each imply that $X$ is admissible.

Assume now that $X$ has a single orbit type. We first consider the case that distinct orbits in $X$ are mapped by $\alpha$ to distinct orbits in $\G_0$.

\begin{lemma}
\label{lem:BurnsideHFixIndependentDiffAlpha}
Let $\G$ be a definable groupoid with finitely many orbit types, a definable quotient, and a definable orbit type partition and
let $(X,\alpha)$ be a $\G$-space with one orbit type and a definable $\G$-quotient.
Let $x,y\in X$ such that $\G\alpha(x) \neq \G\alpha(y)$.
If $\HH$ is an admissible subgroupoid of $\G$, then $(\G x)^{\HH}$ is definably homeomorphic to $(\G y)^{\HH}$.
In particular, $\chi\big((\G x)^{\HH}\big) = \chi\big((\G y)^{\HH}\big)$.
\end{lemma}
\begin{proof}
Let $x, y\in X$ and assume $\G\alpha(x) \neq \G\alpha(y)$.
As $x$ and $y$ have the same orbit type in the $\G$-space $X$, we have by Definition~\ref{def:OrbitTypeG-Space},
a definable groupoid isomorphism $\varphi\colon\G_{\vert \G \alpha(x)} = \G\ltimes(\G \alpha(x))\to\G_{\vert \G \alpha(y)} = \G\ltimes(\G \alpha(y))$ as well as
a $\big(\G\ltimes(\G \alpha(x)))$-$\big(\G\ltimes(\G \alpha(y))\big)$-equivariant definable homeomorphism $\psi\colon\G x\to\G y$ intertwined by $\varphi$.
Let $\HH$ be an admissible subgroupoid of $\G$. By Definition~\ref{def:AdmissableSubgroupoid}~iii.,
there is a definable bisection $\sigma$ of $\G$ over $\G\alpha(y)$ with target $\Ad(\sigma)_0\big(\G \alpha(y)\big) = \G\alpha(y)$ such that the definable isomorphism
$\Ad(\sigma)\circ\varphi\colon\G\ltimes\big(\G \alpha(x)\big)\to\G\ltimes\big(\G \alpha(y)\big)$ restricts to a definable isomorphism
$\HH\ltimes\big(\G\alpha(x)\big)\to\HH\ltimes\big(\G\alpha(y)\big)$.

Let $\widehat{\psi}\colon\G x\to\G y$ be defined by $\widehat{\psi}(z) = \big(\sigma\circ\alpha\circ\psi(z)\big)\ast\psi(z)$.
Then $\widehat{\psi}$ is a continuous definable map. To see that $\widehat{\psi}$ is a homeomorphism, we claim that the
definable continuous map $w\mapsto \psi^{-1}\big[ \sigma\big( \Ad(\sigma)_0^{-1}\big(\alpha(w)\big)\big)^{-1} \ast w \big]$ is an inverse
of $\widehat{\psi}$. Let $w\in\G y$.
Then $\alpha(w)\in\G\alpha(y)$ so that as $\Ad(\sigma)_0$ is a homeomorphism of $\G\alpha(y)$ onto itself, $\Ad(\sigma)_0^{-1}\big(\alpha(w)\big)\in\G_0$
is defined, and $\sigma\big[\Ad(\sigma)_0^{-1}\big(\alpha(w)\big)\big]$ is an arrow from $p = \Ad(\sigma)_0^{-1}\big(\alpha(w)\big)$ to
$\Ad(\sigma)\circ\Ad(\sigma)_0^{-1}\big(\alpha(w)\big) = \alpha(w)$. Let
\[
    z = \psi^{-1}\big( \sigma(p)^{-1}\ast w\big),
\]
and then
\begin{align*}
    \widehat{\psi}(z)
    &=      \big(\sigma\circ\alpha\circ\psi(z)\big)\ast\psi(z)
    \\&=    \big(\sigma\circ\alpha(\sigma(p)^{-1}\ast w)\big)\ast(\sigma(p)^{-1}\ast w)
    \\&=    \big(\sigma\circ t(\sigma(p)^{-1})\big)\ast(\sigma(p)^{-1}\ast w)
    \\&=    \big(\sigma\circ s\circ\sigma(p)\big)\ast(\sigma(p)^{-1}\ast w)
    \\&=    \sigma(p)\ast(\sigma(p)^{-1}\ast w)
    \\&=    w,
\end{align*}
so that the map $w\mapsto \psi^{-1}\big[ \sigma\big( \Ad(\sigma)_0^{-1}\big(\alpha(w)\big)\big)^{-1} \ast w \big]$ is a right inverse of $\widehat{\psi}$.
Similarly,
\begin{align*}
    &\psi^{-1}\big[ \sigma\big( \Ad(\sigma)_0^{-1}\big(\alpha(\widehat{\psi}(z))\big)\big)^{-1} \ast \widehat{\psi}(z) \big]
    \\&\qquad=  \psi^{-1}\Big[ \sigma\Big( \Ad(\sigma)_0^{-1}\big(\alpha\big[\big(\sigma\circ\alpha\circ\psi(z)\big)\ast\psi(z)\big]\big)\Big)^{-1} \ast
                \big(\sigma\circ\alpha\circ\psi(z)\big)\ast\psi(z) \Big]
    \\&\qquad=  \psi^{-1}\Big[ \sigma\Big( (t\circ\sigma)^{-1}\big[t\circ\sigma\circ\alpha\circ\psi(z)\big]\Big)^{-1} \ast
                \big(\sigma\circ\alpha\circ\psi(z)\big)\ast\psi(z) \Big]
    \\&\qquad=  \psi^{-1}\Big[ \sigma \big(\alpha\circ\psi(z)\big)^{-1} \ast
                \big(\sigma\circ\alpha\circ\psi(z)\big)\ast\psi(z) \Big]
    \\&\qquad=  \psi^{-1}\circ\psi(z)
    \\&\qquad=  z,
\end{align*}
so that $w\mapsto \psi^{-1}\big[ \sigma\big( \Ad(\sigma)_0^{-1}\big(\alpha(w)\big)\big)^{-1} \ast w \big]$ is as well a left inverse of $\widehat{\psi}$,
and $\widehat{\psi}$ is a definable homeomorphism.

In addition, $\widehat{\psi}$ is $\G\ltimes\big(\G \alpha(x)\big)$-$\G\ltimes\big(\G \alpha(y)\big)$-equivariant intertwined by
$\Ad(\sigma)\circ\varphi$. To see this, let $z\in \G x$. We compute
\begin{align*}
    \alpha\circ\widehat{\psi}(z)
    &=      \alpha \big[\big(\sigma\circ\alpha\circ\psi(z)\big)\ast\psi(z)\big]
    \\&=    t\circ\sigma\circ\alpha\circ\psi(z)
    \\&=    \Ad(\sigma)_0\circ\alpha\circ\psi(z)
    \\&=    \Ad(\sigma)_0\circ\varphi_0\circ\alpha(z)
    \\&=    \big(\Ad(\sigma)\circ\varphi\big)_0\circ\alpha(z),
\end{align*}
where the second to last step follows from the fact that $\psi$ is
$\big(\G\ltimes(\G \alpha(x)))$-$\big(\G\ltimes(\G \alpha(y))\big)$-equivariant intertwined by $\varphi$.
As well, for $g\in\G_1$ with $s(g) = \alpha(z)$, we again use the equivariance of $\psi$ and note that
$s\circ\varphi_1(g) = \varphi_0\circ\alpha(z)$ and compute
\begin{align*}
    \widehat{\psi}(g\ast z)
    &=      \big(\sigma\circ\alpha\circ\psi(g\ast z)\big)\ast\psi(g\ast z)
    \\&=    \big[\sigma\circ\alpha\big(\varphi_1(g)\ast\psi(z)\big)\big]\ast\varphi_1(g)\ast\psi(z)
    \\&=    \big(\sigma\circ t\circ\varphi_1(g)\big)\ast \varphi_1(g)\ast\psi(z)
    \\&=    \big(\sigma\circ t\circ\varphi_1(g)\big)\ast \varphi_1(g)\ast \big(\sigma\circ s\circ\varphi_1 (g)\big)^{-1}\ast
                \big(\sigma\circ s\circ\varphi_1(g)\big)\ast\psi(z)
    \\&=    \big[ \big(\sigma\circ t\circ\varphi_1(g)\big)\varphi_1(g)\big(\sigma\circ s\circ\varphi_1 (g)\big)^{-1} \big]\ast
                \big(\sigma\circ\varphi_0\circ\alpha(z)\big)\ast\psi(z)
    \\&=    \big[\big(\Ad(\sigma)\circ\varphi\big)_1(g)\big]\ast\widehat{\psi}(z).
\end{align*}
This completes the verification that $\widehat{\psi}$ is $\G\ltimes\big(\G \alpha(x)\big)$-$\G\ltimes\big(\G \alpha(y)\big)$-equivariant intertwined by
$\Ad(\sigma)\circ\varphi$.

Now, recall that $\Ad(\sigma)\circ\varphi$ restricts to a definable isomorphism $\HH\ltimes\big(\G\alpha(x)\big)\to\HH\ltimes\big(\G\alpha(y)\big)$.
Hence, if $z\in\G x$ is fixed by $\HH$, it then follows that for each $h\in\HH$ with source $\alpha\circ\widehat{\psi}(z)$, we have
\[
    h\ast\widehat{\psi}(z)
        =   \widehat{\psi}\big[ \big(\Ad(\sigma)\circ\varphi\big)_1^{-1}(h) \ast z\big]
        =   \widehat{\psi}(z)
\]
so that $\widehat{\psi}(z)$ is fixed by $\HH$ as well. As $\Ad(\sigma)\circ\varphi$ is a definable isomorphism and $\widehat{\psi}$ a definable homeomorphism,
the converse holds by the same argument, switching roles. Therefore, $\widehat{\psi}$ restricts to a definable homeomorphism from $(\G x)^{\HH}$
to $(\G y)^{\HH}$, completing the proof.
\end{proof}

The hypothesis that $\G\alpha(x) \neq \G\alpha(y)$ is not necessary for the conclusion of Lemma~\ref{lem:BurnsideHFixIndependentDiffAlpha}
as long as there are points in $\G_0$ with the same orbit type as $\alpha(x)$ that are not contained in the $\G$-orbit of $\alpha(x)$. To address this case,
we have the following.

\begin{lemma}
\label{lem:BurnsideHFixIndependentSameAlpha}
Let $\G$ be a definable groupoid with finitely many orbit types, a definable quotient, and a definable orbit type partition and
let $(X,\alpha)$ be a $\G$-space with one orbit type and a definable $\G$-quotient.
Let $x,y\in X$ such that $\G\alpha(x) = \G\alpha(y) = \G p\subseteq\G_0$. Assume there exists $q\in\G_0$ such that $p$ and $q$ have the
same orbit type in $\G$ and $\G p \neq \G q$.
If $\HH$ is an admissible subgroupoid of $\G$, then $(\G x)^{\HH}$ is definably homeomorphic to $(\G y)^{\HH}$.
In particular, $\chi\big((\G x)^{\HH}\big) = \chi\big((\G y)^{\HH}\big)$.
\end{lemma}
\begin{proof}
Let $x\in X$, assume without loss of generality that $\alpha(x) = p$, and let $q\in\G_0$ such that $p$ and $q$ have the same orbit type in $\G$ and $\G p \neq \G q$.
Then there is a definable isomorphism $\varphi\colon \G_{\vert \G p} = \G\ltimes (\G p)\to\G_{\vert \G q} = \G\ltimes (\G q)$.
We define a definable $\G$-space $Z$ as follows. As a definable space, $Z$ is given by a copy of the orbit
$\G x$ of $x$ so that there is a definable homeomorphism $\psi\colon \G x \to Z$.
The anchor map $\beta \colon Z\to\G_0$ is given by $\varphi_0\circ\alpha\circ\psi^{-1}$, and the $\G$-action is defined by
$g\ast z = \psi\big(\varphi_1^{-1}(g)\ast \psi^{-1}(z)\big)$ for $g\in\G_1$ and $z\in Z$ such that $s(g) = \beta(z)$. That is, $Z$ is a copy of $\G x$
with the action of $\G\ltimes (\G p)$ translated to an action of $\G\ltimes (\G q)$ via
$\varphi$. By definition, the map $\psi$ is a $\big(\G\ltimes (\G p)\big)$-$\big(\G\ltimes (\G q)\big)$-equivariant homeomorphism
intertwined by $\varphi$. Hence, $x$ and $\psi(x)$ have the same orbit type in the $\G$-space $X\sqcup Z$. Note that $X\sqcup Z$
clearly has a definable $\G$-quotient, as the orbit space of $Z$ is a point.

If $y\in X$ such that $\G\alpha(x) = \G\alpha(y)$, then as $x$ and $y$ have the same orbit type by hypothesis, there is a
$\big(\G\ltimes(\G p)\big)$-$\big(\G\ltimes(\G p)\big)$-equivariant definable homeomorphism
$\psi^\prime\colon\G y\to \G x$ intertwined by a definable automorphism $\varphi^\prime\colon\G\ltimes (\G p)\to\G\ltimes (\G p)$.
Then the composition $\psi\circ\psi^\prime$ is a $\big(\G\ltimes (\G p)\big)$-$\big(\G\ltimes (\G q)\big)$-equivariant homeomorphism
intertwined by $\varphi\circ\varphi^\prime$. Therefore, $x$, $y$, and $\psi(x)$, as well as the elements of their $\G$-orbits,
have the same orbit type in the $\G$-space $X\sqcup Z$. It follows that $X\sqcup Z$ has a single orbit type. The claim then follows
by applying Lemma~\ref{lem:BurnsideHFixIndependentDiffAlpha} to $x$ and $\psi(x)$ as well as $y$ and $\psi(x)$.
\end{proof}

If there is a point $p\in\G_0$ such that $\G p$ contains every element in $\G_0$ with the same orbit type as $p$,
then Definition~\ref{def:AdmissableSubgroupoid}~iii.\@ does not apply, allowing more possibilities for $\HH_{\vert\G p} = \HH\ltimes(\G p)$
where $\HH$ is an admissible subgroupoid of $\G$. In particular, $\HH\ltimes(\G p)$ need not be preserved by a definable automorphism of
$\G\ltimes(\G p)$ so that if $x,y\in X$ are in distinct $\G$-orbits with $\alpha(x),\alpha(y)\in\G p$, it may no longer be the case that $(\G x)^\HH$ and $(\G y)^\HH$
are definably homeomorphic as in Lemmas~\ref{lem:BurnsideHFixIndependentDiffAlpha} and \ref{lem:BurnsideHFixIndependentSameAlpha}.
We address this final case with the following, which is similar to the case of a compact Lie group; see \cite[Prop.~5.5.1]{tomDieckTGRepThry}.
If $G$ is a group and $H$ is a subgroup, $N_G(H)$ denotes the normalizer of $H$ in $G$.

\begin{lemma}
\label{lem:BurnsideHFixIndependentOneGpOrbitType}
Let $\G$ be a definable groupoid with finitely many orbit types, a definable quotient, and a definable orbit type partition.
Assume $p\in\G_0$ is such that $\G_p^p$ is compact and the only points in $\G_0$ with the same orbit type as $p$ are contained in $\G p$.
Let $(X,\alpha)$ be an admissible $\G$-space with one orbit type, and suppose $x,y\in X$ such that
$\alpha(x) = \alpha(y) = p$. If $N_{\G_p^p}(\G_x^x)/\G_x^x$ is infinite, then $[\G x] = [\emptyset]$.
If $N_{\G_p^p}(\G_x^x)/\G_x^x$ and $N_{\G_p^p}(\G_y^y)/\G_y^y$ are both finite and $\G_x^x$ and $\G_y^y$ are conjugate in $\G_p^p$,
then $[\G x] = [\G y]$.
\end{lemma}

Note that the $\G$-isotropy group of each point in $\G x$ is conjugate in $\G$ so that $\G x$ is clearly admissible.

\begin{proof}
Let $\HH$ be an admissible subgroupoid of $\G$ and assume there exists $h\in\HH_1$ with $s(h) = p$ and $t(h)\neq p$. Then $\alpha^{-1}(p)^{\HH} = \emptyset$
and, as $h^{-1}\in\HH_1$, $\alpha^{-1}\big(t(h)\big)^{\HH} = \emptyset$. It follows that $\alpha^{-1}(q)^{\HH} = \emptyset$ for each $q\in\HH p$.
If $q\in\G p$ with $q\notin\HH p$, then Definition~\ref{def:AdmissableSubgroupoid}~iv.\@ yields an isomorphism
$\HH_{\vert\HH p} = \HH\ltimes (\HH p) \to\HH_{\vert\HH q} = \HH\ltimes (\HH q)$ induced by a definable bisection of $\G$ over $\HH p$,
implying that there is a $k\in\HH_1$ with $s(k) = q$ and $t(k)\neq q$. Therefore, $\alpha^{-1}(q)^{\HH} = \emptyset$ for each $q\in\G p$
so that $(\G x)^{\HH} = (\G y)^{\HH} = \emptyset$.

Now assume there is no $h\in\HH_1$ with $s(h) = p$ and $t(h)\neq p$. By Definition~\ref{def:AdmissableSubgroupoid}~iv., there is for each
$q\in\G p$ a bisection $\sigma$ of $\G$ over $\HH p$ such that $\Ad(\sigma)$ is a definable isomorphism
$\HH_{\vert \HH p} = \HH\ltimes(\HH p)\to\HH_{\vert \HH q} = \HH\ltimes(\HH q)$. As $\HH p$ is a singleton, $\sigma$ is given by a single
arrow $g = \sigma(p)$ such that $s(g) = p$, $t(g) = q$, and $g(\HH_p^p)g^{-1} = \HH_q^q$. In particular, $\HH q$ is a singleton.
Moreover, $\alpha(g\ast x) = q$, and the action of $g$ defines a definable homeomorphism of $\alpha^{-1}(p)\cap(\G x)$ onto $\alpha^{-1}(q)\cap(\G x)$
that restricts to a definable homeomorphism of $\alpha^{-1}(p)\cap(\G x)^{\HH}$ onto $\alpha^{-1}(q)\cap(\G x)^{\HH}$.
If $(\G x)^{\HH} \neq \emptyset$, then the restricted anchor map $\alpha_{\vert(\G x)^{\HH}}\colon(\G x)^{\HH}\to \G p$ is surjective, and
Proposition~\ref{prop:Multiplicative} implies
\begin{equation}
\label{eq:BurnsideHFixIndependentSameAlphaAnchorFiber}
    \chi\big((\G x)^{\HH}\big) = \chi\big(\alpha^{-1}(p)\cap(\G x)^{\HH}\big) \chi\big(\G p\big).
\end{equation}
If $(\G x)^{\HH} = \emptyset$, then this equation holds trivially. Applying the same observations to $\G y$, it follows that
if $\chi\big(\alpha^{-1}(p)\cap(\G x)^{\HH}\big) = \chi\big(\alpha^{-1}(p)\cap(\G y)^{\HH}\big)$, then
$\chi\big((\G x)^{\HH}\big) = \chi\big((\G y)^{\HH}\big)$.

The remainder of the proof follows that of \cite[Prop.~5.5.1]{tomDieckTGRepThry}. Note that the compact Lie group $\G_p^p$ acts on $\alpha^{-1}(p)\cap(\G x)$
and $\alpha^{-1}(p)\cap(\G x)^{\HH} = \alpha^{-1}(p)\cap(\G x)^{\HH_p^p}$.
Assume $\alpha^{-1}(p)\cap(\G x)^{\HH} \neq\emptyset$ and let $z\in\alpha^{-1}(p)\cap(\G x)^{\HH}$. Then $\HH_p^p\leq\G_z^z\leq\G_p^p$
and $N_{\G_p^p}(\G_z^z)/\G_z^z$ is continuously and definably isomorphic to $N_{\G_p^p}(\G_x^x)/\G_x^x$ via conjugation by an $h\in \G_p^p$ such that
$h\ast x = z$. Then $N_{\G_p^p}(\G_z^z)/\G_z^z$ acts freely and definably on $\alpha^{-1}(p)\cap(\G x)^{\HH}$, so that as 
$N_{\G_p^p}(\G_z^z)/\G_z^z$ is compact Lie, the $N_{\G_p^p}(\G_z^z)/\G_z^z$-action on $\alpha^{-1}(p)\cap(\G x)^{\HH}$ has a definable quotient $Q$
by \cite[Ch.~10, Cor.~(2.18)]{vandenDriesBook}. Applying Proposition~\ref{prop:Multiplicative} to the quotient map $\alpha^{-1}(p)\cap(\G x)^{\HH}\to Q$,
$\chi\big(\alpha^{-1}(p)\cap(\G x)^{\HH}\big) = \chi(Q)\chi\big(N_{\G_p^p}(\G_z^z)/\G_z^z\big)$.

Suppose $N_{\G_p^p}(\G_x^x)/\G_x^x$ is infinite, in which case $N_{\G_p^p}(\G_z^z)/\G_z^z$ is also infinite and hence a positive-dimensional compact Lie group.
Then as $\chi\big(N_{\G_p^p}(\G_z^z)/\G_z^z\big) = 0$,
we have $\chi\big(\alpha^{-1}(p)\cap(\G x)^{\HH}\big) = 0$. Hence either $\alpha^{-1}(p)\cap(\G x)^{\HH} = \emptyset$ or $\alpha^{-1}(p)\cap(\G x)^{\HH}$ is nonempty
with vanishing Euler characteristic. Along with Equation~\eqref{eq:BurnsideHFixIndependentSameAlphaAnchorFiber}, and noting that this holds for each
admissible $\HH$, it follows that $[\G x] = [\emptyset]$.

Now suppose $N_{\G_p^p}(\G_x^x)/\G_x^x$ and $N_{\G_p^p}(\G_y^y)/\G_y^y$ are both finite. Assume $\G_x^x$ is conjugate to $\G_y^y$ in $\G_p^p$.
If $\alpha^{-1}(p)\cap(\G x)^{\HH} \neq\emptyset$, then as above, there exists $h\in\G_p^p$ such that $h\ast x = z$ with $\HH_p^p\leq\G_z^z$,
and then $h\G_x^x h^{-1} = \G_z^z$. As $\G_z^z$ is conjugate to $\G_y^y$, there is an $h^\prime\in\G_p^p$ such that $h^\prime\ast y = w$ with $\G_w^w = \G_z^z$.
Then $(\G_p^p x)^{\HH} = (\G_p^p z)^{\HH}$ and $(\G_p^p y)^{\HH} = (\G_p^p w)^{\HH}$ are both definably homeomorphic to $(\G_p^p/\G_z^z)^{\HH_p^p}$.
That is, $\alpha^{-1}(p)\cap(\G x)^{\HH}$ is definably homeomorphic to $\alpha^{-1}(p)\cap(\G y)^{\HH}$ which, along with
Equation~\eqref{eq:BurnsideHFixIndependentSameAlphaAnchorFiber}, implies that $[\G x] = [\G y]$.
\end{proof}

Now, let us combine Lemmas~\ref{lem:BurnsideHFixIndependentDiffAlpha}, \ref{lem:BurnsideHFixIndependentSameAlpha}, and \ref{lem:BurnsideHFixIndependentOneGpOrbitType}
to describe how the class of an admissible $\G$-space with one orbit type decomposes into a linear combination of classes of transitive $\G$-spaces.
Let $(X,\alpha)$ be an admissible $\G$-space, and continue to assume that $X$ has one orbit type as a $\G$-space.
Note that $\alpha(X)$ contains points in $\G_0$ with the same orbit type in $\G$ by Definition~\ref{def:OrbitTypeG-Space}.
Pick a point $p\in\alpha(X)$. We define a decomposition $X = X_1\sqcup\cdots\sqcup X_s$ of $X$ as follows. If $\G_0$ contains at least one
point with the same orbit type in $\G$ as $p$ that is not an element of $\G p$, we let $s = 1$ and $X_1 = X$.
Otherwise, $\alpha(X) = \G p$, and for each $x\in X$, $\G x$ intersects $\alpha^{-1}(p)$. Let $(K_1),\ldots,(K_s)$
denote the $\G_p^p$-conjugacy classes of isotropy groups for the $\G_p^p$-action on $\alpha^{-1}(p)$ and recall that there are finitely many
by Definition~\ref{def:BurnsideSet}~iv. Note that if $x\in\alpha^{-1}(p)$
has $\G_p^p$-isotropy group $K_i$, then each element of $(\G x)\cap\alpha^{-1}(p)$ has $\G_p^p$-isotropy group contained in $(K_i)$.
Let $X_i$ be the set of points $y\in X$ such that the $\G_p^p$-isotropy group of each element of $(\G y)\cap\alpha^{-1}(p)$ is in the conjugacy class $(K_i)$.
Then each $X_i$ is a $\G$-invariant subspace of $X$ and the $X_i$ partition $X$. As well, each $\alpha^{-1}(p)^{K_i}$ is a definable subset of $X$.
As the orbit map $\pi_X\colon X\to \lvert X\rvert$ is a morphism of affine definable spaces,
each $\pi_X\big(\alpha^{-1}(p)^{K_i}\big)$ is a definable subset of $\lvert X\rvert$.
The $(K_i)$ are partially ordered by $(K_i)\leq(K_j)$ if there is a $g\in\G_p^p$ such that $K_i\leq gK_j g^{-1}$, and
letting $\alpha^{-1}(p)_{(K_i)}$ denote the set of points whose $\G_p^p$-isotropy group is in the conjugacy class $(K_i)$,
each $\pi_X\big(\alpha^{-1}(p)^{K_j}\big)$ is a disjoint union of the $\pi_X\big(\alpha^{-1}(p)_{(K_i)}\big)$ with $(K_j)\leq(K_i)$.
Hence, the $\pi_X\big(\alpha^{-1}(p)_{(K_i)}\big)$ can be described recursively along this partial order as finite set differences of the $\pi_X\big(\alpha^{-1}(p)^{K_i}\big)$.
This implies that the $X_i$ are preimages under $\pi_X$ of definable subsets of $\lvert X\rvert$, hence definable subsets of $X$.
Moreover, each $X_i$ has a definable $\G$-quotient given by the definable subset $\lvert X_i \rvert = \pi_X(X_i)$ of $\lvert X\rvert$.
Finally, as $X$ has one orbit type as a $\G$-space and each $X_i$ is a union of $\G$-orbits in $X$, each $X_i$ has one orbit type as a $\G$-space.

Let $\HH$ be an admissible subgroupoid of $\G$, fix an $i$ with $1\leq i\leq s$, and let $\rho = (\pi_X)_{\vert X_i}$ denote the restriction of the orbit map of $X$.
By Theorem~\ref{thrm:Trivialization}, we can partition $\lvert X_i \rvert$ into definable subsets
$\lvert X_i \rvert = \lvert X_i\rvert_1\sqcup\cdots\sqcup \lvert X_i\rvert_r$ such that each $\rho_{\vert \rho^{-1}(\lvert X_i\rvert_j)}$ admits a definable
trivialization respecting $\big(\rho^{-1}(\lvert X_i\rvert_j)\big)^{\HH}$. Hence, each
$\rho^{-1}(\lvert X_i\rvert_j)$ is definably homeomorphic to $\lvert X_i\rvert_j\times E_{ij}$ for a definable space $E_{ij}$,
and $\rho_{\vert \rho^{-1}(\lvert X_i\rvert_j)}$ corresponds to the projection via this definable homeomorphism.
However, note that each $E_{ij}$ is given by the $\G$-orbit $\G x_{ij}$ of a point $x_{ij}\in\rho^{-1}(\lvert X_i\rvert_j)\subseteq X_i$,
and recall that $X_i$ has a single orbit type as a $\G$-space. Choose any point $x_i\in X_i$, and then for any $j$,
$E_{ij}$ is definably homeomorphic to $\G x_i$. Then by Proposition~\ref{prop:Multiplicative}, we have
\[
    \chi(X_i) = \chi\big(\lvert X_i \rvert\big) \chi(\G x_i).
\]
Similarly, we have by Lemmas~\ref{lem:BurnsideHFixIndependentDiffAlpha}, \ref{lem:BurnsideHFixIndependentSameAlpha}
and \ref{lem:BurnsideHFixIndependentOneGpOrbitType} that for each $j$, $\chi\big((\G x_{ij})^{\HH}\big) = \chi\big((\G x_i)^{\HH}\big)$.
As the trivializations of $\rho_{\vert \rho^{-1}(\lvert X_i\rvert_j)}$ respect the subsets $\big(\rho^{-1}(\lvert X_i\rvert_j)\big)^{\HH}$,
Proposition~\ref{prop:Multiplicative} and Lemma~\ref{lem:TrivialProduct} yield
\begin{align*}
    \chi(X_i^{\HH})
        &=      \chi\big(\lvert X_i \rvert\big) \chi\big((\G x_i)^{\HH}\big)
        \\&=    \chi\big(\big(\lvert X_i\rvert\ttimes (\G x_i)\big)^{\HH}\big).
\end{align*}
It follows that $[X_i] = [\lvert X_i \rvert\ttimes (\G x_i)]$, i.e., that the class of $X_i$ in $B(\G)$ is equal to the class of the trivial product
of $\lvert X_i\rvert$ with the orbit $\G x_i$ of any $x_i\in X_i$. Then by Lemma~\ref{lem:TrivialProduct},
$[X] = \sum_{i=1}^s \chi(\lvert X_i\rvert)[\G x_i]$. By Lemma~\ref{lem:BurnsideHFixIndependentOneGpOrbitType}, for each $i$ such that $K_i$ has infinite index
in $N_{\G_p^p}(K_i)$, $[\G x_i] = [\emptyset]$. Then we have the following.

\begin{lemma}
\label{lem:BurnsideSingleOrbit}
Let $\G$ be a definable groupoid with finitely many orbit types, a definable quotient, and a definable orbit type partition, and
let $(X,\alpha)$ be an admissible $\G$-space such that $X$ has one orbit type as a $\G$-space. Then there are finitely
many $\G$-invariant definable subsets $X_1,\ldots,X_t$ of $X$ such that for any choice of $x_j\in X_j$,
\[
    \textup{$[X] =   \sum\limits_{j=1}^t [\lvert X_j\rvert\ttimes\G x_j]
        =   \sum\limits_{j=1}^t \chi(\lvert X_j\rvert)[\G x_j].$}
\]
If $\alpha(X)$ contains more than one orbit in $\G_0$ or $\alpha(X) = \G p$ for some $p\in\G_0$ such that $\G p$ does not contain every point with the same orbit type
as $p$ in the groupoid $\G$, then $t = 1$ and $X_1 = X$.
Otherwise, the $X_j$ are given by the $\G$-saturations of $\alpha^{-1}(p)_{(K_j)}$, where the $(K_j)$ are the $\G_p^p$-conjugacy classes of isotropy groups for the
$\G_p^p$-action on $\alpha^{-1}(p)$ such that $K_j$ has finite index in its normalizer in $\G_p^p$.
\end{lemma}

Now, let $(X,\alpha)$ be an admissible $\G$-space, which we no longer assume has a single orbit type.
By Proposition~\ref{prop:BurnsideSingleOrbitType}, $[X] = \sum_{i=1}^r [A_i]$ where each $A_i\subseteq X$ is a $\G$-space with a single orbit type.
Applying Lemma~\ref{lem:BurnsideSingleOrbit} to each $A_i$, and re-indexing the resulting subsets of all of the $A_i$ with the single index $\nu$,
we obtain the following, which implies Theorem~\ref{main:BurnsideGenTransitive}.

\begin{theorem}[$B(\G)$ is generated by transitive $\G$-spaces]
\label{thrm:BurnsideGenTransitive}
Let $\G$ be a definable groupoid with finitely many orbit types, a definable quotient, and a definable orbit type partition.
Then $B(\G)$ is generated as an abelian group by the classes of transitive $\G$-spaces.
Specifically, if $X$ is an admissible $\G$-space, there are finitely
many $\G$-invariant definable subsets $X_1,\ldots,X_N$ of $X$ such that for any choice of $x_\nu\in X_\nu$,
\begin{equation}
\label{eq:BurnsideGen}
    [X] =   \sum\limits_{\nu=1}^N \chi(\lvert X_\nu\rvert)[\G x_\nu].
\end{equation}
If $X = A_1\sqcup\cdots\sqcup A_r$ denotes the orbit type partition of $X$, then the $X_\nu$ are defined as in Lemma~\ref{lem:BurnsideSingleOrbit}
applied to each $A_i$.
\end{theorem}

Note that if $[X] = [X^\prime]$ and $X$ is a transitive definable $\G$-space, it need not be the case that $X^\prime$ is transitive.
In particular, if $Y$ is any affine definable space with $\chi(Y) = 1$ (e.g., a closed interval), then $[X] = [Y\ttimes X]$
by Lemma~\ref{lem:TrivialProduct}. Note further that it is possible in Equation~\eqref{eq:BurnsideGen} that $[\G x_\nu] = [\G x_\eta]$
for $\nu\neq \eta$.

By Theorem~\ref{thrm:BurnsideGenTransitive}, we can express
\[
    \chi_{\G}^{\eq}(X) = \sum\limits_{\nu=1}^N \chi(\lvert X_\nu\rvert)[\G x_\nu]
\]
where the $X_\nu$ are as defined in the theorem and $x_\nu$ is a choice of point in $X_\nu$ for each $\nu$; compare \cite[Def.~2]{GZ-EquivariantAnalogues}.

Alternatively, if $C_1\sqcup\cdots\sqcup C_k$ is a definable triangulation of $\lvert X\rvert$ in the sense of \cite[Ch.~9]{vandenDriesBook}, i.e.,
each $C_j$ is the interior of a simplex, and each $C_j$ is either contained in $\lvert X_\nu\rvert$ for some $\nu$ or is disjoint from the $\lvert X_\nu\rvert$,
then choosing an $x_j$ with $\G x_j\in X_j$ for each $j$,
\[
    \chi_{\G}^{\eq}(X) = \sum\limits_{j=1}^k (-1)^{\dim C_j} [\G x_j],
\]
analogous to \cite[Def.~1]{GZ-EquivariantAnalogues}. This follows directly from the fact that each $\lvert X_\nu\rvert$
is a disjoint union of $C_j$ and the additivity of $\chi$.
Note that if $C_j$ is disjoint from the $\lvert X_\nu\rvert$, then for any $x_j\in C_j$, $[\G x_j] = [\emptyset]$; see Lemma~\ref{lem:BurnsideHFixIndependentOneGpOrbitType}.

Note that the Burnside group $B(\G)$, and hence the equivariant Euler characteristic $\chi^{\eq}(\G)$, is an invariant of the definable isomorphism
class of a definable groupoid $\G$. The Burnside group, however, is not invariant under definable Morita equivalence as the following example demonstrates.

\begin{example}
\label{ex:BurnsideMorita}
Consider the definably Morita equivalent definable groupoids $\G$ and $\G^\prime$ described in Example~\ref{ex:OrbitTypesMorita}.
For the groupoid $\G$, as there are no nontrivial isotropy arrows, any transitive $\G$-space can be identified with one of the subsets $O_1$ or $O_2$ of $\G_0$ via its anchor map.
As $G_1\ltimes O_1$ and $G_2\ltimes O_2$ are distinct $\G$-orbits with the same orbit type in $\G$, the only admissible subgroupoids $\HH$ of $\G$ are
$\G$ itself and the subgroupoid consisting of only units. In either case, $\chi\big(O_1^{\HH}\big) = \chi\big(O_2^{\HH}\big)$ so that $[O_1] = [O_2]$
in $B(\G)$. That is, $B(\G)\simeq\Z$ is generated by the equivalence class of a single transitive $\G$-set.

In the case of the definable groupoid $\G^\prime$, transitive $\G^\prime$-sets again can be identified with orbits in $\G_0^\prime$ and hence either
$O_3$ or $\{z\}$. However, for the admissible subgroupoid $\HH$ of $\G^\prime$ consisting only of units,
\[
    \chi\big(O_3^{\HH}\big)
    =
    \chi(O_3)
    =
    2\neq 1
    =
    \chi(\{z\})
    =
    \chi\big(\{z\}^{\HH}\big),
\]
so that $[O_3]\neq[\{z\}]$ in $B(\G^\prime)$. Because $\chi\big(O_3^{\G^\prime}\big) = \chi(\emptyset) = 0$ and
$\chi\big(\{z\}^{\G^\prime}\big) = \chi(\{z\}) = 1$, $[O_3]$ and $[\{z\}]$ are linearly independent over $\Z$
and generate $B(\G^\prime)\simeq\Z^2$. Therefore, $B(\G)\not\simeq B(\G^\prime)$.
\end{example}

Even in the case of transitive definable groupoids, it is possible that $B(\G)$ is trivial while $B(\G^\prime)$ is nontrivial for definably Morita equivalent
$\G$ and $\G^\prime$.

\begin{example}
\label{ex:BurnsideMoritaTransitive}
Let $\G = \Sp^1\times\Sp^1$ be the pair groupoid on the circle $\Sp^1$ and let $\mathbf{T}$ denote the trivial groupoid with single object
$z$ as in Example~\ref{ex:OrbitTypesMorita}.
Then $\G$ and $\mathbf{T}$ are transitive, and the embedding of $\mathbf{T}$ into $\G$ as any point and its unit is a definable weak equivalence.
If $X$ is a nonempty transitive definable $\G$-space, then for any choice of $x\in X$, the map from $s^{-1}\big(\alpha(x)\big)$ to $X$ given by
$h\mapsto h\ast x$ is a definable homeomorphism; it follows that $X$ is $\G$-equivariantly definably homeomorphic to $\G_0$ with identity anchor.
Let $\HH$ be an admissible subgroupoid of $\G$. If $\HH$ contains a nonunit arrow with source $p_1$, then $\HH p_1 \cap \G_0^{\HH} = \emptyset$.
However, by Definition~\ref{def:AdmissableSubgroupoid}~iv., for any $p_2\in\G_0$, $\HH\ltimes(\HH p_1)$ is definably isomorphic to $\HH\ltimes(\HH p_2)$ so that
$\G_0^{\HH} = \emptyset$. If $\HH$ consists of only units, then $\chi(\G_0^{\HH}) = \chi(\G_0) = 0$. Hence $[\G_0] = [\emptyset]$ in $B(\G)$ so that
$B(\G)$ is trivial. On the other hand, $\chi(\mathbf{T}_0^{\mathbf{T}}) = \chi(\mathbf{T}_0) = 1$ so that $B(\mathbf{T})$ contains the nontrivial element
$[\mathbf{T}_0]$, and in fact can easily be seen to be isomorphic to $\Z$.
\end{example}

The situation described in Example~\ref{ex:BurnsideMorita} differs from the Burnside ring of a finite groupoid treated by El Kaoutit-Spinosa.
Their classical and categorified Burnside rings are invariant under Morita equivalence by
\cite[Thrm.~6.5]{KaoSpinBurnsideTheory} and \cite[Thrm.~7.9]{KaoSpinCatfied}, respectively.
Applying their constructions to Example~\ref{ex:BurnsideMorita} above, the $\G$-sets $O_1$ and $O_2$ represent different elements in the Burnside rings
of the finite groupoid $\G$.
In fact, by \cite[Cor.~6.10]{KaoSpinBurnsideTheory}, their classical Burnside ring decomposes into a product of the Burnside rings
of isotropy groups, and by \cite[Thrm.~7.10]{KaoSpinCatfied}, their categorified Burnside ring is a product of the
categorified Burnside rings of each restriction $\G_{\vert\G p} = \G\ltimes(\G p)$ of $\G$ to an orbit, with both decompositions indexed by the orbits of $\G$.
For the topological and definable groupoids
we consider, it is important to identify the $\G$-sets $O_1$ and $O_2$ in the Burnside group, as otherwise $B(\G)$ would have uncountably many generators
whenever $\lvert\G\rvert$ is not discrete.

Moreover, it is not surprising that the equivariant Euler characteristic is not Morita invariant, as it is by Theorem~\ref{thrm:BurnsideGenTransitive}
closely related to the equivariant Euler characteristic for
global quotient orbifolds described in \cite[Sec.~3]{GZ-EquivariantAnalogues}. For a finite group $G$, Gusein-Zade's equivariant Euler characteristic
is defined for a sufficiently tame $G$-space $X$ taking values in the classical Burnside ring $A(G)$ of $G$. The behavior of the resulting
invariant under Morita equivalence of the
translation groupoid $G\ltimes X$ is subtle; for topologically Morita equivalent groupoids, it can even reside in the Burnside rings of different groups.
For example, assume that $X$ is a smooth manifold and $K$ is a normal subgroup of $G$ such that $K$ acts freely on $X$. Then the groupoid homomorphism
$G\ltimes X \to (G/K) \ltimes (X/K)$ given by the quotient maps is a Lie weak equivalence of Lie groupoids; see \cite[Prop.~3.5]{PronkScull}.
The equivariant Euler characteristic of $G\ltimes X$ is an element of $A(G)$, while the equivariant Euler characteristic of $(G/K) \ltimes (X/K)$
is an element of $A(G/K)$, ruling out Morita invariance. Note that there is an injective homomorphism $A(G/K)\to A(G)$ called \emph{inflation}, see \cite{BoucBurnside},
but it is not an isomorphism in general.
It is easy to see that the inflation homomorphism maps the equivariant Euler characteristic of $(G/K) \ltimes (X/K)$ to the equivariant Euler characteristic of
$G\ltimes X$.

\subsection{A partial multiplication on the Burnside group}
\label{subsec:BurnsideRing}

We now assume that $\G$ is transitive, in which case $\G$ has a single orbit type and a definable quotient as $\lvert\G\rvert$ is a point;
see Definition~\ref{def:OrbitTypeGrpoid}.
We further assume that the isotropy groups of $\G$, which are mutually isomorphic as topological groups, are compact; for a transitive groupoid $\G$
with noncompact isotropy groups, the only admissible $\G$-space is $\emptyset$ by Condition iv.\@ of Definition~\ref{def:BurnsideSet} so that $B(\G) = \{[\emptyset]\}$.
Let $(X,\alpha)$ and $(Y, \beta)$ be definable $\G$-spaces and define the definable $\G$-space
\[
    X\ftimes{\alpha}{\beta}Y
        = \{ (x,y) \in X\times Y : \alpha(x) = \beta(y) \}
\]
with anchor map $A\colon X\ftimes{\alpha}{\beta}Y\to\G_0$ defined by
\[
    A(x,y) = \alpha(x) = \beta(y)
\]
and action map defined as follows. If $g\in \G_1$ such that $s(g) = A(x,y)$, then $s(g) = \alpha(x) = \beta(y)$
so that $g\ast x$ and $g\ast y$ are defined. We set
\[
    g\ast(x,y) = (g\ast x, g\ast y).
\]
Note that $X\ftimes{\alpha}{\beta}Y$ is an affine definable space by Lemma~\ref{lem:FibProdDefinable} and then the anchor and action maps are clearly definable.

Now, if classes $[X]$ and $[Y]$ in $B(\G)$ admit representatives $(X,\alpha)$ and $(Y, \beta)$ such that $X\ftimes{\alpha}{\beta}Y$ is admissible, we define
\begin{equation}
\label{eq:BSTransProduct}
    [X][Y] = [X\ftimes{\alpha}{\beta}Y].
\end{equation}
Clearly, $X\ftimes{\alpha}{\beta}Y$ is definably homeomorphic to $Y\ftimes{\beta}{\alpha}X$ so that this product is commutative.
We first show that when the product is defined, it is well-defined.

\begin{proposition}
\label{prop:BSProductWellDefTransitive}
Let $\G$ be a transitive definable groupoid with compact isotropy groups and
let $(X,\alpha)$ and $(Y, \beta)$ be $\G$-spaces that represent classes in $B(\G)$. Assume that
$X \ftimes{\alpha}{\beta} Y$ is admissible. Then
$[X][Y] = [X \ftimes{\alpha}{\beta} Y]$
is well-defined, i.e., $[X \ftimes{\alpha}{\beta} Y]$ depends only on the classes $[X]$ and $[Y]$ in $B(\G)$.
\end{proposition}

It is not clear whether $X \ftimes{\alpha}{\beta} Y$ can have infinitely many orbit types as a $\G$-space.

\begin{proof}
Suppose $[X] = [X^\prime]$ and $[Y] = [Y^\prime]$ for $\G$-spaces $(X^\prime, \alpha^\prime)$ and $(Y^\prime, \beta^\prime)$,
and assume that $X \ftimes{\alpha}{\beta} Y$ and $X^\prime \ftimes{\alpha^\prime}{\beta^\prime} Y^\prime$ represent classes in $B(\G)$.
Fix an admissible subgroupoid $\HH$ of $\G$.
By Theorem~\ref{thrm:Trivialization} and Remark~\ref{rem:Trivialization}, as $\alpha$ is a continuous morphism of affine definable spaces,
we can express $\G_0 = A_1\sqcup\cdots\sqcup A_p$
such that on each $A_i$ there is a definable trivialization of $\alpha$ respecting $X^{\HH}$.
Specifically, there is for each $i$ an affine definable space $E_i$ and a map
$\lambda_i\colon \alpha^{-1}(A_i)\to E_i$ such that $(\alpha_{\vert \alpha^{-1}(A_i)}, \lambda_i)\colon\alpha^{-1}(A_i)\to A_i\times E_i$
is a definable homeomorphism, and such that
$(\alpha_{\vert \alpha^{-1}(A_i)}, \lambda_i)\big(\alpha^{-1}(A_i)^\HH\big) = A_i\times\widehat{E_i}$ for some definable subset $\widehat{E_i}$ of $E_i$.
In other words, $(\alpha_{\vert \alpha^{-1}(A_i)}, \lambda_i))$ identifies $\alpha^{-1}(A_i)$ with $A_i\times E_i$ in such a way that $\alpha$
corresponds to the projection onto $A_i$, and the restriction of $(\alpha_{\vert \alpha^{-1}(A_i)}, \lambda_i)$ to $\alpha^{-1}(A_i)^\HH$
identifies $\alpha^{-1}(A_i)^\HH$ with $A_i\times\widehat{E_i}$. In the same way, we express $\G_0 = B_1\sqcup\cdots\sqcup B_q$
so that $\beta^{-1}(B_j)$ is identified with $B_j\times F_j$ and each $\beta^{-1}(B_j)^{\HH}$ is identified with
$B_j\times\widehat{F_j}$, express $\G_0 = C_1\sqcup\cdots\sqcup C_r$
so that $(\alpha^\prime)^{-1}(C_k)$ is identified with $C_k\times E_k^\prime$ and each $(\alpha^\prime)^{-1}(C_k)^{\HH}$ is identified with
$C_k\times\widehat{E_k^\prime}$, and express $\G_0 = D_1\sqcup\cdots\sqcup D_s$
so that $(\beta^\prime)^{-1}(D_\ell)$ is identified with $D_\ell\times F_\ell^\prime$ and each $(\beta^\prime)^{-1}(D_\ell)^{\HH}$ is identified with
$D_\ell\times\widehat{F_\ell^\prime}$. Here, each $F_j$, $E_k^\prime$, and $F_\ell^\prime$ is an affine definable space with respective definable subsets
$\widehat{F_j}$, $\widehat{E_k^\prime}$, and $\widehat{F_\ell^\prime}$.
We will suppress the maps $\lambda_i$ and their analogs for the other three trivializations for simplicity and identify each space with its trivialization,
i.e., identify $\alpha^{-1}(A_i)$ with $A_i\times E_i$, etc.

Let
\[
    Z_{ijk\ell} = A_i \cap B_j \cap C_k \cap D_\ell.
\]
Then $\G_0$ is the union of the finite number of sets $Z_{ijk\ell}$, and the preimage of $Z_{ijk\ell}$ under $\alpha$ is definably homeomorphic
to $Z_{ijk\ell} \times E_i$, the preimage under $\beta$ is homeomorphic to $Z_{ijk\ell} \times F_j$, etc. As above, we identify
$\alpha^{-1}(Z_{ijk\ell})$ with $Z_{ijk\ell} \times E_i$, etc., and then with respect to these identifications,
as $\alpha$ and $\beta$ correspond to the projections onto $Z_{ijk\ell}$, we have by Corollary~\ref{cor:FibProductTrivial} that
$(Z_{ijk\ell} \times E_i)\ftimes{\alpha}{\beta} (Z_{ijk\ell} \times F_j)$ is definably homeomorphic to $Z_{ijk\ell} \times E_i \times F_j$.

Now, as $\G_0$ is the union of the $Z_{ijk\ell}$, $[X] = [X^\prime]$, and $\HH$ is admissible, we have
$\chi(X^\HH) = \chi\big((X^\prime)^\HH\big)$, i.e.,
\[
    \chi\left(\bigsqcup_{i,j,k,\ell} Z_{ijk\ell} \times\widehat{E_i} \right)
    =
    \chi\left(\bigsqcup_{i,j,k,\ell} Z_{ijk\ell} \times\widehat{E_k^\prime} \right).
\]
By the additivity of $\chi$ and Proposition~\ref{prop:Multiplicative}, it follows that
\begin{equation}
\label{eq:BSProductEs}
    \sum\limits_{i,j,k,\ell} \chi( Z_{ijk\ell} ) \chi( \widehat{E_i} )
    =
    \sum\limits_{i,j,k,\ell} \chi(Z_{ijk\ell})\chi(\widehat{E_k^\prime}).
\end{equation}
Similarly, as $[Y] = [Y^\prime]$, we have
\begin{equation}
\label{eq:BSProductFs}
    \sum\limits_{i,j,k,\ell} \chi( Z_{ijk\ell} ) \chi( \widehat{F_j} )
    =
    \sum\limits_{i,j,k,\ell} \chi(Z_{ijk\ell})\chi(\widehat{F_\ell^\prime}).
\end{equation}

Note that the definable homeomorphism class of the space $E_i$ does not depend on $i$. If $X=\emptyset$, then clearly $E_i=\emptyset$ for all $i$, so assume not.
Then the image of $\alpha$ is a nonempty union of orbits in $\G_0$ so that as $\G$ is transitive, $\alpha$ is surjective.
Moreover, for any $p_1,p_2\in\G_0$, there is an arrow $g\in\G_1$ with $s(g) = p_1$ and $t(g) = p_2$. The action of $g$ on $X$ yields a definable homeomorphism
from $\alpha^{-1}(p_1)$ to $\alpha^{-1}(p_2)$. Then as $E_i$ is given by the preimage under $\alpha$ of a point in $A_i\subseteq\G_0$,
each $E_i$ is definably homeomorphic to a common affine definable space $E$, and similarly for $F$, $E^\prime$, and $F^\prime$.

We next claim that the definable homeomorphism classes of $\widehat{E_i}$, $\widehat{E_k^\prime}$, $\widehat{F_j}$, and $\widehat{F_\ell^\prime}$
do not depend on the indices $i$, $j$, $k$, and $\ell$. Let $x_1, x_2\in X$, and let
$\alpha(x_1) = p_1$ and $\alpha(x_2) = p_2$. If $p_1$ and $p_2$ are in the same $\HH$-orbit, then there is an $h\in\HH_1$ such that $s(h) = p_1$ and $t(h) = p_2$. Then either $p_1 = p_2$, in which case the result is trivial, or the action of $h$ maps $\alpha^{-1}(p_1)$ to $\alpha^{-1}(p_2)$, and the action of $h^{-1}$ maps $\alpha^{-1}(p_2)$ to $\alpha^{-1}(p_1)$. In this latter case, $\alpha^{-1}(p_1)^{\HH} =\alpha^{-1}(p_2)^{\HH} = \emptyset$.

Now suppose $p_1$ and $p_2$ are not in the same $\HH$-orbit. Then by Definition~\ref{def:AdmissableSubgroupoid}~iv.,
there is a definable bisection $\sigma$ of $\G$ over $\HH p_1$ with target $\Ad(\sigma)_0(\HH p_1) = \HH p_2$ such that $\Ad(\sigma)$ restricts to a definable isomorphism
$\HH\ltimes(\HH p_1)\to\HH\ltimes\HH(p_2)$. It follows that if $y\in\alpha^{-1}(p_1)$ is fixed by $h\in s^{-1}(p_1)\cap\HH_1$, in which case $h\in\HH_{p_1}^{p_1}$, then
\[
    (\sigma(p_1)h\sigma(p_1)^{-1})\ast(\sigma(p_1)\ast y)
        =   \sigma(p_1)\ast(h\ast y)
        =   \sigma(p_1)\ast y,
\]
i.e., $\sigma(p_1)\ast y$ is fixed by $\sigma(p_1)h\sigma(p_1)^{-1}$.
Therefore, the action of $\sigma(p_1)$ defines a definable map $\alpha^{-1}(p_1)^\HH\to\alpha^{-1}\big(\Ad(\sigma)_0(p_1)\big)^\HH$
whose inverse is given by the action of $\sigma(p_1)^{-1}$. Then as $\Ad(\sigma)_0(p_1) \in\HH p_2$, we have by the previous case that
$\alpha^{-1}(p_1)^\HH$ and $\alpha^{-1}(p_2)^\HH$ are definably homeomorphic.

In either case, it follows that $\widehat{E_i}$ does not depend on $i$ and is homeomorphic to a common affine definable space
$\widehat{E}$, and similarly for $\widehat{F}$, $\widehat{E^\prime}$, and $\widehat{F^\prime}$.
Then Equation~\eqref{eq:BSProductEs} becomes
\[
    \chi( \widehat{E} )\sum\limits_{i,j,k,\ell} \chi( Z_{ijk\ell} )
    =
    \chi(\widehat{E^\prime}) \sum\limits_{i,j,k,\ell} \chi(Z_{ijk\ell}),
\]
implying that either $\sum_{i,j,k,\ell} \chi( Z_{ijk\ell} ) = 0$ or $\chi( \widehat{E} ) = \chi(\widehat{E^\prime})$;
similarly, Equation~\eqref{eq:BSProductFs} implies that either $\sum_{i,j,k,\ell} \chi( Z_{ijk\ell} ) = 0$ or
$\chi( \widehat{F} ) = \chi(\widehat{F^\prime})$.

In either case, we therefore have
\begin{align*}
    \chi\left( (X\ftimes{\alpha}{\beta}Y)^\HH \right)
        &=      \chi\left(X^\HH\ftimes{\alpha}{\beta}Y^\HH\right)
        \\&=    \chi\left(\bigsqcup_{i,j,k,\ell} Z_{ijk\ell} \times\widehat{E_i}\times\widehat{F_j} \right)
        \\&=    \sum\limits_{i,j,k,\ell} \chi( Z_{ijk\ell} ) \chi( \widehat{E_i} ) \chi( \widehat{F_j} )
        \\&=    \chi( \widehat{E} ) \chi( \widehat{F} )\sum\limits_{i,j,k,\ell} \chi( Z_{ijk\ell} )
        \\&=    \chi( \widehat{E^\prime} ) \chi( \widehat{F^\prime} )\sum\limits_{i,j,k,\ell} \chi( Z_{ijk\ell} )
        \\&=    \sum\limits_{i,j,k,\ell} \chi( Z_{ijk\ell} ) \chi( \widehat{E^\prime} ) \chi( \widehat{F^\prime} )
        \\&=    \chi\left(\bigsqcup_{i,j,k,\ell} Z_{ijk\ell} \times\widehat{E_k^\prime}\times\widehat{F_\ell^\prime} \right)
        \\&=    \chi\left((X^\prime)^\HH\ftimes{\alpha^\prime}{\beta^\prime}(Y^\prime)^\HH\right)
        \\&=    \chi\left((X^\prime\ftimes{\alpha^\prime}{\beta^\prime}Y^\prime)^\HH\right).
\end{align*}
It follows that $[X][Y] = [X^\prime][Y^\prime]$ so that the product is well-defined on classes of $\G$-spaces
for which it is defined.
\end{proof}

Now, let us note that if $(X,\alpha)$, $(Y,\beta)$, and $(Z,\delta)$ are definable $\G$-spaces representing classes in
$B(\G)$ such that $[X][Y]$ and $[X][Z]$ are defined, $[X\ftimes{\alpha}{\beta\sqcup\delta}(Y\sqcup Z)]$ is defined as well, and
\[
    [X][Y] + [X][Z] = [X\ftimes{\alpha}{\beta\sqcup\delta}(Y\sqcup Z)].
\]
In particular, $X\ftimes{\alpha}{\beta\sqcup\delta} (Y\sqcup Z) = X\ftimes{\alpha}{\beta}Y\sqcup X\ftimes{\alpha}{\delta} Z$
as definable $\G$-spaces, and the same decomposition holds for the $\HH$-fixed point sets for any admissible subgroupoid $\HH$.
The definable spaces $(X\ftimes{\alpha}{\beta}Y)\ftimes{\beta\circ\pr_2}{\delta}Z$ and $X\ftimes{\alpha}{\beta\circ\pr_1}(Y\ftimes{\beta}{\delta}Z)$
are clearly $\G$-equivariantly definably homeomorphic $\G$-spaces so that the product in $B(\G)$ is associative when each of the intermediate fibred products
are admissible.
Finally, the hypotheses that $\G$ is transitive with compact isotropy groups imply that $\G_0$ is admissible as a $\G$-space with the anchor $\id_{\G_0}$, and
$X\ftimes{\alpha}{\id_{\G_0}}\G_0 = \{ \big( x, \alpha(x)\big) \in X\times\G_0\}$ is $\G$-equivariantly definably homeomorphic
to $X$. Hence, the class $[\G_0]$ is a multiplicative identity in $B(\G)$.
Combining these observations with Theorem~\ref{thrm:BurnsideGenTransitive}, we have therefore proven the following.

\begin{theorem}[Partial product in $B(\G)$]
\label{thrm:Product}
Let $\G$ be a transitive definable groupoid with compact isotropy groups.
Then Equation~\eqref{eq:BSTransProduct} yields a commutative associative partial product on $B(\G)$, with identity $[\G_0]$,
that is defined for all classes $[X]$ and $[Y]$ with respective admissible representatives $(X,\alpha)$ and $(Y,\beta)$ such that $X\ftimes{\alpha}{\beta}Y$
is admissible. The partial product distributes over the group operation in $B(\G)$ when it is defined.

If $X \ftimes{\alpha}{\beta} Y$ is admissible and the subsets $\{X_\nu\}_{\nu=1}^N$ of $X$ and $\{Y_\eta\}_{\eta=1}^M$ of $Y$ are defined as
in Theorem~\ref{thrm:BurnsideGenTransitive}, then for any choice of $x_\nu\in X_\nu$ and $y_\eta\in Y_\eta$,
\[
    [X \ftimes{\alpha}{\beta} Y]
        =   \sum\limits_{\nu,\eta} \chi\big( \lvert X_\nu\rvert\big)\chi\big(\lvert Y_\eta\rvert \big) [\G x_\nu][\G y_\eta].
\]
\end{theorem}

Note that if $X \ftimes{\alpha}{\beta} Y$ is admissible, then each $[\G x_\nu][\G y_\eta]$ is defined.

If $\G$ is not transitive, then Equation~\eqref{eq:BSTransProduct} does not yield a well-defined product in $B(\G)$.
We illustrate this with the following example, which was pointed out by ChatGPT (OpenAI) and independently verified by the authors.

\begin{example}
\label{ex:NonTransProductGPT}
Let $\G$ be the groupoid with $\G_0 = \{p_1, p_2\}$ and $\G_1$ consisting only of units. Let $X = \{x\}$ and $Y = \{y\}$
be singleton $\G$-spaces with anchor maps $\alpha\colon x\mapsto p_1$ and $\beta\colon y\mapsto p_2$. Then as the only admissible
subgroupoid of $\G$ is itself and $\chi(X^\G) = \chi(X) = 1 = \chi(Y) = \chi(Y^\G)$, we have $[X] = [Y]$ in $B(\G)$. However, $X\ftimes{\alpha}{\alpha}X = \{(x,x)\}$
and $X\ftimes{\alpha}{\beta}Y = \emptyset$ so that $[X\ftimes{\alpha}{\alpha}X] \neq [X\ftimes{\alpha}{\beta}Y]$ in $B(\G)$.
\end{example}

Note that when $X$ and $Y$ are transitive $\G$-spaces, $X \ftimes{\alpha}{\beta} Y$ is not necessarily transitive.
As a simple example, if $G$ is a nontrivial group treated as a groupoid with one object, then $G$ is a transitive $G$-space
with respect to the action by left-translation, but $G\times G$ with the diagonal action is never transitive.

To end this paper, let us note the following.
Though the Burnside group class, i.e., the equivariant Euler characteristic,
has additive and multiplicative properties, it is not an additive nor multiplicative topological invariant
of orbit space definable groupoids in the sense of Definition~\ref{def:OrbSpDefGpdInvar}.
For $\G$-spaces $(X,\alpha)$ and $(Y,\beta)$ representing classes in $B(\G)$, $\chi_{\G}^{\eq}$ is additive
in the sense that $[X \sqcup Y] = [X] + [Y]$, i.e., $\chi_{\G}^{\eq}(X\sqcup Y) = \chi_{\G}^{\eq}(X) + \chi_{\G}^{\eq}(Y)$. This follows from
the definition of the sum in $B(\G)$; see Section~\ref{subsec:BurnsideGroup}. If $\G$ is transitive, it is as well multiplicative in the sense that
$[X\ftimes{\alpha}{\beta}Y] = [X][Y]$, i.e., $\chi_{\G}^{\eq}(X\ftimes{\alpha}{\beta}Y) = \chi_{\G}^{\eq}(X)\chi_{\G}^{\eq}(Y)$,
when the former is defined.
However, if $\G$ and $\HH$ are definable groupoids that are as well orbit space definable, $\chi^{\eq}(\G) = \chi_{\G}^{\eq}(\G_0) = [\G_0]$,
$\chi^{\eq}(\HH) = \chi_{\HH}^{\eq}(\HH_0) = [\HH_0]$, $\chi^{\eq}(\G\sqcup\HH) = \chi_{\G\sqcup\HH}^{\eq}(\G_0\sqcup\HH_0) = [\G_0\sqcup\HH_0]$, and
$\chi^{\eq}(\G\times\HH) = \chi_{\G\times\HH}^{\eq}(\G_0\times\HH_0) = [\G_0\times\HH_0]$ take values in
$B(\G)$, $B(\HH)$, $B(\G\sqcup\HH)$, and $B(\G\times\HH)$, respectively, and hence are not
elements of a common ring.

In fact, the invariants $\chi^{\un}(\G\ltimes X)$ and $\chi_\G^{\eq}(X)$ contain distinct information, as we illustrate with the following.

\begin{example}
\label{ex:UnECvsEqEC1}
Let $\G$ be a topological groupoid such that $\G_0 = \{p_1, p_2\}$ consists of two points with the discrete topology and $\G_1$ consists of only isotropy arrows, with
$s^{-1}(p_1) = t^{-1}(p_1) = G_1$ and $s^{-1}(p_2) = t^{-1}(p_2) = G_2$ for distinct compact Lie groups $G_1$ and $G_2$.
Giving $G_1$ and $G_2$ definable structures as in Remark~\ref{rem:LieGroup}, $\G$ is clearly a definable groupoid with finitely many orbit types, a definable quotient,
and a definable orbit type partition.
Note that $\lvert\G\rvert$ consists of two points with distinct weak orbit types so that $\G$ is orbit space definable, and
an admissible subgroupoid $\HH$ of $\G$ is given by any choice of closed subgroups $H_i\leq G_i$ for $i=1,2$.
Because the orbits are singletons and have distinct orbit types, Condition~iii.\@ of Definition~\ref{def:AdmissableSubgroupoid} is vacuous,
and Condition iv.\@ is trivially satisfied by a bisection given by the restriction of the unit map.

Let $X$ be any affine definable space with $\chi(X)\neq 0$. We consider $X$ as a $\G$-space in two ways, denoted $(X_i,\alpha_i)$ for $i=1,2$ with $X_1 = X_2 = X$.
We let $\alpha_i\colon X_i\to\G_0$ be the constant map $\alpha_i(x) = p_i$ and define $g\ast x = x$ for each $x\in X_i$ and $g\in\G_1$ such that $s(g) = p_i$.
Then as $X_i^{\HH} = X_i$ for any admissible $\HH$ and $i=1,2$, it follows that $[X_1] = [X_2]$ in $B(\G)$. That is,
$\chi_{\G}^{\eq}(X_1) = \chi_{\G}^{\eq}(X_2) = \chi(X)[pt]$ where $[pt]$ is the class in $B(\G)$ of a point with any trivial $\G$-action.
However, as the isotropy group of any point in $X_i$ is isomorphic to $G_i$ as a topological group, we have $\chi^{\un}(\G\ltimes X_i) = \chi(X)T^{[G_i]}$ for
$i=1,2$. Hence $\chi^{\un}(\G\ltimes X_1)\neq\chi^{\un}(\G\ltimes X_2)$.
\end{example}

\begin{example}
\label{ex:UnECvsEqEC2}
Let $\G$ again be a definable groupoid such that $\G_0 = \{p_1, p_2\}$ consists of two points and $\G_1$ contains only isotropy arrows, with
$s^{-1}(p_1) = t^{-1}(p_1) = \Sp^1$ and $s^{-1}(p_2) = t^{-1}(p_2) = \{1\}$.
Let $(X_1,\alpha_1)$ be the $\G$-space with $X_1 = \Sp^2$ the sphere, constant anchor $\alpha_1(x) = p_1$, and action given by the usual action
of $\Sp^1$ on $\Sp^2$ by rotations about the $z$-axis. Let $(X_2,\alpha_2)$ be the $\G$-space with $X_2 = I\sqcup\{q_1,q_2\}$ where $I$ is an open interval.
Define anchor map $\alpha_2\colon X_2\to\G_0$ by $\alpha_2(q_1) = \alpha_2(q_2) = p_1$ and $\alpha_2(x) = p_2$ for $x\in I$, and give
$X_2$ the trivial $\G$-action. For both $(X_1,\alpha_1)$ and $(X_2, \alpha_2)$, the orbit space $\lvert X_i\rvert$ consists of an open interval
with trivial isotropy and two points with $\Sp^1$-isotropy so that $\chi^{\un}(\G\ltimes X_i) = -T^{[e]} + 2T^{[\Sp^1]}$ for $i=1,2$.
However, considering the admissible subgroupoid $\HH$ of $\G$ consisting only of units, $\chi(X_1^{\HH}) = \chi(\Sp^2) = 2$, while
$\chi(X_2^{\HH}) = \chi\big(I\sqcup\{q_1,q_2\}\big) = 1$. Hence $[X_1]\neq[X_2]$ in $B(\G)$, i.e., $\chi_{\G}^{\eq}(X_1) \neq \chi_{\G}^{\eq}(X_2)$.
\end{example}


\section*{Declaration of generative AI and AI-assisted technologies in the manuscript preparation process}

The contents of this paper were human-generated with the exception of Example~\ref{ex:NonTransProductGPT}.
During the preparation of this work the authors used ChatGPT (OpenAI) and Claude (Anthropic) in order to check a complete draft of the paper, leading to some corrections.
After using these tools, the authors addressed the corrections manually and take full responsibility for the content of the published article.


\bibliographystyle{amsplain}
\bibliography{universalEC}

\providecommand{\bysame}{\leavevmode\hbox to3em{\hrulefill}\thinspace}
\providecommand{\MR}{\relax\ifhmode\unskip\space\fi MR }
\providecommand{\MRhref}[2]{%
  \href{http://www.ams.org/mathscinet-getitem?mr=#1}{#2}
}
\providecommand{\href}[2]{#2}
\begin{thebibliography}{10}

\bibitem{AdemLeidaRuan}
Alejandro Adem, Johann Leida, and Yongbin Ruan, \emph{Orbifolds and stringy
  topology}, Cambridge Tracts in Mathematics, vol. 171, Cambridge University
  Press, Cambridge, 2007.

\bibitem{BaumLocalIso}
P.~F. Baum, \emph{Local isomorphism of compact connected {L}ie groups}, Pacific
  J. Math. \textbf{22} (1967), 197--204.

\bibitem{BekeInvarEC}
Tibor Beke, \emph{Topological invariance of the combinatorial {E}uler
  characteristic of tame spaces}, Homology Homotopy Appl. \textbf{13} (2011),
  165--174.

\bibitem{Boekholt}
Sven Boekholt, \emph{Compact {L}ie groups with isomorphic homotopy groups}, J.
  Lie Theory \textbf{8} (1998), no.~1, 183--185.

\bibitem{BoucBurnside}
Serge Bouc, \emph{Burnside rings}, Handbook of algebra, {V}ol. 2, Handb.
  Algebr., vol.~2, Elsevier/North-Holland, Amsterdam, 2000, pp.~739--804.

\bibitem{Bredon}
Glen~E. Bredon, \emph{Introduction to compact transformation groups}, Pure and
  Applied Mathematics, Vol. 46, Academic Press, New York-London, 1972.

\bibitem{CrainicMestre}
Marius Crainic and Jo\~{a}o~Nuno Mestre, \emph{Orbispaces as differentiable
  stratified spaces}, Lett. Math. Phys. \textbf{108} (2018), no.~3, 805--859.

\bibitem{CrainicStruchiner}
Marius Crainic and Ivan Struchiner, \emph{On the linearization theorem for
  proper {L}ie groupoids}, Ann. Sci. \'{E}c. Norm. Sup\'{e}r. (4) \textbf{46}
  (2013), no.~5, 723--746.

\bibitem{CurryGhristEAEulerCalc}
Justin Curry, Robert Ghrist, and Michael Robinson, \emph{Euler calculus with
  applications to signals and sensing}, Advances in applied and computational
  topology, Proc. Sympos. Appl. Math., vol.~70, Amer. Math. Soc., Providence,
  RI, 2012, pp.~75--145.

\bibitem{delHoyoOrbispace}
Matias del Hoyo, \emph{Lie groupoids and their orbispaces}, Port. Math.
  \textbf{70} (2013), 161--209.

\bibitem{delHoyoDeMeloLinearization}
Matias del Hoyo and Mateus de~Melo, \emph{On invariant linearization of {L}ie
  groupoids}, Lett. Math. Phys. \textbf{111} (2021), no.~4, Paper No. 112, 14.

\bibitem{DeVito}
Jason DeVito, \emph{A {K}rull-{S}chmidt theorem for {L}ie groups?},
  \url{https://mathoverflow.net/questions/195019/a-krull-schmidt-theorem-for-lie-groups/442205#442205},
  accessed on March 6, 2023.

\bibitem{DuistermaatKolk}
J.~J. Duistermaat and J.~A.~C. Kolk, \emph{Lie groups}, Universitext,
  Springer-Verlag, Berlin, 2000.

\bibitem{KaoSpinCatfied}
Laiachi El~Kaoutit and Leonardo Spinosa, \emph{Categorified groupoid-sets and
  their {B}urnside ring}, Turkish J. Math. \textbf{43} (2019), no.~5,
  2069--2096.

\bibitem{KaoSpinBurnsideTheory}
\bysame, \emph{On {B}urnside theory for groupoids}, Bull. Math. Soc. Sci. Math.
  Roumanie (N.S.) \textbf{66(114)} (2023), no.~1, 41--87.

\bibitem{FarsiPflaumSeaton2}
Carla Farsi, Markus~J. Pflaum, and Christopher Seaton, \emph{Differentiable
  stratified groupoids and a de {R}ham theorem for inertia spaces}, J. Geom.
  Phys. \textbf{187} (2023), Paper No. 104806, 33.

\bibitem{FarsiScullWattsPAMS}
Carla Farsi, Laura Scull, and Jordan Watts, \emph{Classifying spaces and
  {B}redon (co)homology for transitive groupoids}, Proc. Amer. Math. Soc.
  \textbf{148} (2020), no.~6, 2717--2737.

\bibitem{FarsiSeaGenTwistSec}
Carla Farsi and Christopher Seaton, \emph{Generalized twisted sectors of
  orbifolds}, Pacific J. Math. \textbf{246} (2010), no.~1, 49--74.

\bibitem{FarsiSeaGenOrbEuler}
\bysame, \emph{Generalized orbifold {E}uler characteristics for general
  orbifolds and wreath products}, Algebr. Geom. Topol. \textbf{11} (2011),
  no.~1, 523--551.

\bibitem{FarSeaJLMS}
\bysame, \emph{Orbifold {E}uler characteristics of non-orbifold groupoids}, J.
  Lond. Math. Soc. (2) \textbf{106} (2022), no.~3, 2342--2378.

\bibitem{GZ-EquivariantAnalogues}
S.~M. Gusein-Zade, \emph{Equivariant analogues of the {E}uler characteristic
  and {M}acdonald type formulas}, Uspekhi Mat. Nauk \textbf{72} (2017),
  no.~1(433), 3--36.

\bibitem{GZEMH-HigherOrbEulerCompactGroup}
S.~M. Gusein-Zade, I.~Luengo, and A.~Melle-Hern\'{a}ndez, \emph{Higher-order
  orbifold {E}uler characteristics for compact {L}ie group actions}, Proc. Roy.
  Soc. Edinburgh Sect. A \textbf{145} (2015), no.~6, 1215--1222.

\bibitem{GZLMH-Universal}
\bysame, \emph{The universal {E}uler characteristic of {$V$}-manifolds},
  Funktsional. Anal. i Prilozhen. \textbf{52} (2018), no.~4, 72--85.

\bibitem{HofmannMorriss}
Karl~H. Hofmann and Sidney~A. Morris, \emph{The structure of compact groups---a
  primer for the student---a handbook for the expert}, fourth ed., De Gruyter
  Studies in Mathematics, vol.~25, De Gruyter, Berlin, 2020.

\bibitem{LeinsterEulerCharCategory}
Tom Leinster, \emph{The {E}uler characteristic of a category}, Doc. Math.
  \textbf{13} (2008), 21--49.

\bibitem{MackenzieGT}
Kirill C.~H. Mackenzie, \emph{General theory of {L}ie groupoids and {L}ie
  algebroids}, London Mathematical Society Lecture Note Series, vol. 213,
  Cambridge University Press, Cambridge, 2005.

\bibitem{Metzler}
David Metzler, \emph{Topological and smooth stacks},  (2003),
  \texttt{arXiv:math/0306176 [math.DG]}.

\bibitem{MoerdijkMrcun}
I.~Moerdijk and J.~Mr\v{c}un, \emph{Introduction to foliations and {L}ie
  groupoids}, Cambridge Studies in Advanced Mathematics, vol.~91, Cambridge
  University Press, Cambridge, 2003.

\bibitem{ParkSuhLinEmbed}
Dae~Heui Park and Dong~Youp Suh, \emph{Linear embeddings of semialgebraic
  {$G$}-spaces}, Math. Z. \textbf{242} (2002), no.~4, 725--742.

\bibitem{PflaumBook}
Markus~J. Pflaum, \emph{Analytic and geometric study of stratified spaces},
  Lecture Notes in Mathematics, vol. 1768, Springer-Verlag, Berlin, 2001.

\bibitem{PPTOrbitSpace}
Markus~J. Pflaum, Hessel Posthuma, and Xiang Tang, \emph{Geometry of orbit
  spaces of proper {L}ie groupoids}, J. Reine Angew. Math. \textbf{694} (2014),
  49--84.

\bibitem{Pillay}
Anand Pillay, \emph{On groups and fields definable in {$o$}-minimal
  structures}, J. Pure Appl. Algebra \textbf{53} (1988), no.~3, 239--255.

\bibitem{PronkScull}
Dorette Pronk and Laura Scull, \emph{Translation groupoids and orbifold
  cohomology}, Canad. J. Math. \textbf{62} (2010), no.~3, 614--645.

\bibitem{SatakeGB}
Ichir\^{o} Satake, \emph{The {G}auss-{B}onnet theorem for {$V$}-manifolds}, J.
  Math. Soc. Japan \textbf{9} (1957), 464--492.

\bibitem{TamanoiMorava}
Hirotaka Tamanoi, \emph{Generalized orbifold {E}uler characteristic of
  symmetric products and equivariant {M}orava {$K$}-theory}, Algebr. Geom.
  Topol. \textbf{1} (2001), 115--141.

\bibitem{TamanoiCovering}
\bysame, \emph{Generalized orbifold {E}uler characteristics of symmetric
  orbifolds and covering spaces}, Algebr. Geom. Topol. \textbf{3} (2003),
  791--856.

\bibitem{Tanabe}
Masato Tanabe, \emph{Canonical stratification of definable {L}ie groupoids}, J.
  Singul. \textbf{26} (2023), 63--75.

\bibitem{Thurston}
William~P. Thurston, \emph{The geometry and topology of 3-manifolds}, Princeton
  University Math Dept., Princeton, New Jersey, 1978, Lecture Notes.

\bibitem{tomDieckBurnsideI}
Tammo tom Dieck, \emph{The {B}urnside ring of a compact {L}ie group. {I}},
  Math. Ann. \textbf{215} (1975), 235--250.

\bibitem{tomDieckTGRepThry}
\bysame, \emph{Transformation groups and representation theory}, Lecture Notes
  in Mathematics, vol. 766, Springer, Berlin, 1979.

\bibitem{TuNonHausdorff}
Jean-Louis Tu, \emph{Non-{H}ausdorff groupoids, proper actions and
  {$K$}-theory}, Doc. Math. \textbf{9} (2004), 565--597.

\bibitem{vandenDriesBook}
Lou van~den Dries, \emph{Tame topology and o-minimal structures}, London
  Mathematical Society Lecture Note Series, vol. 248, Cambridge University
  Press, Cambridge, 1998.

\bibitem{ViroIntegralEulerChar}
O.~Ya. Viro, \emph{Some integral calculus based on {E}uler characteristic},
  Topology and geometry---{R}ohlin {S}eminar, Lecture Notes in Math., vol.
  1346, Springer, Berlin, 1988, pp.~127--138.

\bibitem{WeinsteinLineariz}
Alan Weinstein, \emph{Linearization of regular proper groupoids}, J. Inst.
  Math. Jussieu \textbf{1} (2002), no.~3, 493--511.

\bibitem{WilliamsHaar}
Dana~P. Williams, \emph{Haar systems on equivalent groupoids}, Proc. Amer.
  Math. Soc. Ser. B \textbf{3} (2016), 1--8.

\bibitem{Zung}
Nguyen~Tien Zung, \emph{Proper groupoids and momentum maps: linearization,
  affinity, and convexity}, Ann. Sci. \'{E}cole Norm. Sup. (4) \textbf{39}
  (2006), no.~5, 841--869.

\end{thebibliography}

\end{document}